\documentclass{article}
\usepackage{amsthm, amssymb, amsmath}
\usepackage{enumerate}
\usepackage{amsthm}
\usepackage{graphicx}
\numberwithin{equation}{section}
\newtheorem{theorem}{Theorem}[section]

\newtheorem{lemma}[theorem]{Lemma}

\newtheorem{proposition}[theorem]{Proposition}

\huge

\begin{document}
\title{On the dynamics of formation of generic singularities of mean curvature flow}
\author{Zhou Gang\footnote{present address: Department of Mathematical Sciences, Binghamton University, Binghamton, 13902, NY, gzhou@math.binghamton.edu, partly supported by NSF grant DMS-1308985 and DMS-1443225.}}
\maketitle
\centerline{Department of Mathematics, California Institute of Technology, Pasadena, CA, 91125}
\setlength{\leftmargin}{.1in}
\setlength{\rightmargin}{.1in}
\normalsize \vskip.1in
\setcounter{page}{1} \setlength{\leftmargin}{.1in}
\setlength{\rightmargin}{.1in}
\large

\date

\setlength{\leftmargin}{.1in}
\setlength{\rightmargin}{.1in}
\normalsize \vskip.1in
\setcounter{page}{1} \setlength{\leftmargin}{.1in}
\setlength{\rightmargin}{.1in}
\large

\section*{Abstract}
We study formation of generic singularities under mean curvature flow by combining the different approaches and results, namely the techniques used by the author and the collaborators, where we obtained detailed informations when the initial hypersurfaces are close to cylinders, and these results by Colding and Minicozzi, which included all the generic blowups. Here we choose to study the cases where the rescaled flow converge to the cylinder $\mathbb{S}^{1}\times \mathbb{R}^3,$ and we extend the region controlled by Colding and Minicozzi to find a finer description of a neighborhood of the singularity.

\tableofcontents

\section{Introduction}

Here we study mean curvature flow (MCF) for a family of $n-$dimensional hypersurface embedded in $\mathbb{R}^{n+1}$ satisfying the equation
\begin{align}\label{eq:MeanCur}
\partial_{t}{\bf{x}}_t=-H{\bf{n}},
\end{align}where $H$ is the mean curvature, ${\bf{n}}$ is the normal vector at the point $\bf{x}$ in the surface.  

There exists an extensive literature for the dynamics of formation of singularities. In \cite{Huis1} Huisken studied convex hypersurfaces. For nonconvex ones, Huisken invented in \cite{Huis2} an
energy functional for rescaled MCF,
\begin{align}
F(\Sigma):=(4\pi)^{-\frac{n}{2}} \int_{\Sigma} e^{-\frac{|\omega|^2}{4}} d\mu(\omega),\label{eq:defFSi}
\end{align}where $\mu(\omega)$ is the area element at $\omega$ on the hypersurface $\Sigma$. An important property is that, suppose MCF forms a singularity at time $T$ and at ${\bf{x}}=0,$
and let $\Sigma_t$ be the hypersurface for the rescaled MCF $ \frac{{\bf{x}_t}}{\sqrt{T-t}}$ at time $t$, then it was shown that $\frac{d}{dt}F(\Sigma_t)\leq 0$, and $\frac{d}{dt}F(\Sigma_t)= 0$ only when $\Sigma_t$ is a cylinder $\mathbb{R}^{k}\times \mathbb{S}^{n-k}_{\sqrt{2(n-k)}}$ and its rotations, where $k=0,\ \cdots,\ n-1,$ and $\mathbb{S}^{n-k}_{\sqrt{2(n-k)}}$ is the $n-k$-dimensional sphere with radius $\sqrt{2(n-k)}.$

Suppose that in the initial hypersurface $\Sigma_0$ satisfies the condition, see \cite{CoMi2012},
\begin{align}\label{eq:generic}
\lambda(\Sigma_0):=\sup_{t_0,\ x_0} (4\pi t_0)^{-\frac{n}{2}}\int_{\Sigma_0} e^{-\frac{|x-x_0|^2}{4t_0}} d\mu<\infty,
\end{align}and suppose that the hypersurface forms a singularity at time $T$ and at ${\bf{x}}=0$, then Colding and Minicozzi proved in \cite{ColdingMiniUniqueness} that the rescaled MCF
$\frac{{\bf{x}_t}}{\sqrt{T-t}}$ converges to a unique cylinder $\mathbb{R}^{k}\times \mathbb{S}^{n-k}_{\sqrt{2(n-k)}}$ and their rotations, as $t\rightarrow T.$ Here $k=0,\ \cdots,\ n-1.$

Different from the works above, in \cite{GaKn20142, GaKnSi, GS2008, DGSW} Knopf, Sigal, the author, together with some other collaborators, applied a new set of techniques, namely modulational equations, propagator estimates, and finding optimal coordinates to study the blowup of nonlinear heat equations, and neckpinch for MCF. Different from most of the known works, the energy functional \eqref{eq:defFSi} did not play any role. By the new techniques,
in \cite{GaKn20142} we proved that the rescaled MCF converges to a unique cylinder, and obtained some detailed estimates, for a limited class of generic initial surfaces. After that Colding and Minicozzi proved the uniqueness for all generic blowups in \cite{ColdingMiniUniqueness}. 

In the present paper we unify the different approaches. We choose to study the regimes where the limit cylinder is $\mathbb{R}^3\times \mathbb{S}_{\sqrt{2}}$. The reason is that this regime has not been well understood as the others, for example $\mathbb{R}^1\times \mathbb{S}^{k}_{\sqrt{2k}}$. Huisken and Sinestrari in \cite{HuisSine2009} studied the cases where the limit cylinders are $\mathbb{S}^{k}\times \mathbb{R}^{4-k}$, with $k=4,3$. See also the works of Hamilton in \cite{Hamilton1997} for Ricci flow. In \cite{AltAngGiga1995}, Altschuler, Angenent and Giga studied the flow through singularities for surfaces of rotation. For the other related works, see \cite{sesum2008, BrHui2016, MR3662439, MR3602529}.

The other cases, where the limit cylinders are $\mathbb{R}^{k}\times \mathbb{S}^{n-k}_{\sqrt{2(n-k)}}$, $k=0,\ \cdots,\ n-1,$ will be addressed in our subsequent papers. 

The main goal of the present and subsequent papers is to understand a small, but fixed, neighborhood of singularities of MCF. Here we make some preparation by obtaining a description finer than that in \cite{ColdingMiniUniqueness}. From the results in \cite{ColdingMiniUniqueness}, without loss of generality we suppose that the rescaled MCF converges to $\mathbb{R}^3\times\mathbb{S}^1_{\sqrt{2}}$ as $t\rightarrow T$, then in a (possibly shrinking) neighborhood of $x=0$ the MCF takes the form
\begin{align}
{\bf{x}}_t=\left[
\begin{array}{c}
x\\
u(x,\theta,t) cos\theta\\
u(x,\theta,t) sin\theta
\end{array}
\right],
\end{align} where $x\in \mathbb{R}^3$, $u$ is a positive function, is periodic in $\theta\in [0,2\pi)$, and defined in a set $|x|\leq c(t)$ for some $c(t)>0$. The corresponding part of the rescaled MCF takes the form
\begin{align}\label{eq:ppp}
\frac{1}{\sqrt{T-t}}\left[
\begin{array}{c}
x\\
u(x,\theta,t) cos\theta\\
u(x,\theta,t) sin\theta
\end{array}
\right] =\left[
\begin{array}{ccc}
y\\
v(y, \theta,\tau)cos\theta\\
v(y, \theta,\tau)sin\theta
\end{array}
\right],
\end{align}where $v$ is a function defined as $$v(y,\theta,\tau)=\frac{1}{\sqrt{T-t}}u(x,\theta,t)$$ and the variables $y$ and $\tau$ are defined as 
\begin{align}
y:=\frac{1}{\sqrt{T-t}}x, \ \text{and}\ \tau:=-\ln(T-t).\label{def:ytau}
\end{align}The results in \cite{ColdingMiniUniqueness} imply that, for any fixed $y$ and $\theta$, $$\lim_{\tau\rightarrow \infty}v(y,\theta,\tau)\rightarrow \sqrt{2}.$$

In the present paper we clear a hurdle in the way to our goal, which is to understand a fixed neighborhood of the singularity. Ideally, to understand a fixed neighborhood of the singularity $0$ of MCF, by \eqref{def:ytau} we should control $v(y, \theta, \tau)$ when $|y|\leq ce^{\frac{1}{2}\tau}$ for some $c>0$, since it corresponds to the set, for MCF, $\Big\{x\ \Big| \ |x|\leq c\Big\}$. However we are not ready for this. To see the reason, we will decompose the function $v$ as
\begin{align}
v(y,\theta,\tau)=\sqrt{\frac{2+y^{T}B(\tau)y}{2a(\tau)}}+\xi(y,\theta,\tau)\label{eq:prelimXi}
\end{align} where $B$ is a $3\times 3$-symmetric-real-matrix-valued function, $a\approx \frac{1}{2}$ is a scalar function. We will prove that the first term is the main part. An obvious obstacle emerges:  we need that $B\geq 0$, or is almost semi-positive definite, to make $\sqrt{2+y^{T}By}$ well defined when $|y|$ is large. 

To prove that $B$ is almost semi-positive definite is one of the main objectives of the present paper. 

To achieve this goal we face a dilemma. Before proving $B$ is almost semi-positive definite we can only consider a relatively small region, so that $\sqrt{1+y^{T}By}$ is well defined. On the other hand, we have to consider a sufficiently large region to extract useful information. It is not hard to see the reason, because technically, to restrict the consideration to a certain neighborhood, we need to impose some cutoff function and hence need to control various terms produced by it. Thus if the considered set is too small, then the terms produced by the cutoff function will obscure the information we want to extract.

To overcome these difficulties, we restrict our consideration to the part inside a ball slightly larger than $ B_{10\sqrt{\ln \tau}}(0)$. This is allowed since from \cite{ColdingMiniUniqueness} we derive that $|B(\tau)|\leq \tau^{-\frac{1}{2}-\epsilon_0}$ for some $\epsilon_0>0$, see Lemma \ref{LM:ColdMini} below, which makes the first term in \eqref{eq:prelimXi} is well defined. More importantly, in one of the adopted norms, specifically $\|e^{-\frac{1}{8}|y|^2}\cdot \|_2,$ the adverse effect produced by the imposition of the cutoff function $\chi_{R}$, to be defined in \eqref{eq:reCutoff}, is negligible since it is bounded by 
\begin{align}
\int_{\mathbb{R}^3}|\chi^{'}_{R}(y)|e^{-\frac{1}{4}|y|^2} d^3y\leq \int_{|y|\geq 10\sqrt{\ln \tau}}e^{-\frac{1}{4}|y|^2} d^3y\ll \tau^{-20}.\label{eq:adverse}
\end{align}
We emphasize that we can actually study a much larger region, for example the part inside the ball $B_{\tau^{\frac{1}{4}}}(0).$ The unpleasant thing is that we will obtain similar results, but will have to analyze different parts more carefully. Hence we choose to study the smaller one.

By proving that the $3\times 3$ symmetric matrix $B$ is sufficiently close to be semi-positive definite and obtaining other estimates, 
we makes it possible to study the large neighborhood $B_{ce^{\frac{1}{2}\tau}}(0)$ in a subsequent paper \cite{gang2018description}, which corresponds to MCF inside the ball $B_{c}(0)$, where we address the problems of mean convexity and the isolation of singularities, see also \cite{choi2018ancient, choi2019ancient}. 

Next we discuss the techniques. Compared to the known works, here we use two types of norms, specifically the weighted $L^\infty$-norms $\|(1 +|y|)^{-k}\cdot\|_{\infty},\ k\geq 1,$ and the weighted $L^2$-norm $\|e^{-\frac{1}{8}|y|^2}\cdot\|_{2}$. They are both useful. Controlling the solution in the weighted $L^2$-norm makes it convenient to apply the known results and techniques, since it was used in \cite{ColdingMiniUniqueness} and 
in the study of Type I blowup of nonlinear heat equations, see e.g. \cite{GK1,GK2, GK3, HV2, FK1992, MZ1997, MR1230711},
\begin{align}
\partial_{t} u(x,t)=\Delta u(x,t)+u^{p}(x,t),\ \text{for}\ p>1.\label{eq:NLH}
\end{align} Moreover as discussed above the norm 
$\|e^{-\frac{1}{8}|y|^2}\cdot\|_{2}$ is very effective to limit the adverse effect produced by the cutoff functions, see the discussion around \eqref{eq:adverse}.

One of our technical advantages is to use the weighted $L^\infty$-norms $\|(1+ |y|)^{-k}\cdot\|_{\infty}$, $k=1,2,3,$ to derive point-wise estimates when $|y|$ is large. This plays a crucial role in controlling a sufficiently large neighborhood.  The advantage of adopting these norms is obvious: if one relies on Huisken's energy functional defined in \eqref{eq:defFSi}, then it is nearly impossible to obtain pointwise estimates for the rescaled hypersurface when $|\omega|\approx 10\sqrt{\ln \tau}$ since the weight $e^{-\frac{1}{4}|\omega|^2}$ decays too fast. But our chosen $L^\infty$-norms work much better in controlling the remainder $\xi$ in \eqref{eq:prelimXi} when $|y|$ is large.

We can use these norms because we rely on the propagator estimate to generate decay rates, see Theorem \ref{THM:frequencyWise} below.

The present paper depends on some known results and techniques. The results of Colding and Minicozzi \cite{ColdingMiniUniqueness}, specifically the decay rates (or converging rates) of certain functions, will be used, through Lemma \ref{LM:ColdMini} below. Their key technique, namely Lojasiewicz inequalities, will not play any role here. Thus the present paper is largely a separate step. However we believe that the two sets of techniques can be integrated better, and the proof in \cite{ColdingMiniUniqueness} and the present paper can be simplified considerably. Besides that, we need the techniques of central manifold method, which was used in studying Type I (generic) blowup of nonlinear heat equations, see \cite{GK1,GK2, HV2, FK1992, MZ1997}. The methods in Filippas and Liu \cite{MR1230711} are especially helpful in deriving sharp decay rates in \eqref{eq:b1234}-\eqref{eq:B3Decay} below. However we emphasize that MCF is fundamentally different from nonlinear heat equation, where the coordinate is given, while here to prove the existence of a good coordinate is a central part of our consideration.

The paper is organized as follows: the main theorem will be stated in Section \ref{sec:MainThm}. In Section \ref{sec:Effec} we decompose the graph function $v$ into different parts, and find governing equations for them. The main theorem will be proved in Section \ref{sec:pmainTHM}. In Sections \ref{sec:beta2Eqn}-\ref{sec:estM420} we estimate the different components of $v$, in various norms.

In the present paper, the notation $A\lesssim B$ signifies that there exists a universal constant $C$ such that $A\leq CB.$ The weighted $L^{\infty}$ norm $\|\langle y\rangle^{-k} f\|_{\infty}$ stands for $\|(1+|y|^2)^{-\frac{k}{2}}f\|_{\infty}.$ We define an inner product $\langle\cdot,\cdot\rangle_{\mathcal{G}}$ and hence the norm $\|\cdot\|_{\mathcal{G}}$ such that for any functions $f,g$ 
\begin{align}
\begin{split}\label{def:Gin}
\langle f,\ g\rangle_{\mathcal{G}}=&\int_{\mathbb{R}^3}\int_{0}^{2\pi} e^{-\frac{1}{4}|y|^2}f(y,\theta)\bar{g}(y,\theta) \ d\theta d^3y,\\
\|f\|_{\mathcal{G}}=&\langle f,\ f\rangle_{\mathcal{G}}^{\frac{1}{2}},
\end{split}
\end{align}and accordingly $f\perp_{\mathcal{G}} g$ signifies that $\langle f,\ g\rangle_{\mathcal{G}}=0.$

\section{Main Theorem}\label{sec:MainThm}

We assume the initial hypersurface $
\Sigma_{0}$ satisfies the condition, see \cite{CoMi2012, CIM13}
\begin{align}
\lambda(\Sigma_{0})<\infty,\label{eq:generic1}
\end{align}where $\lambda(\Sigma)$ is defined in \eqref{eq:generic}. Then it was proved in \cite{ColdingMiniUniqueness} that the limit cylinder is unique.

Based on this, we suppose that the blowup point is the origin and the blowup time is $T>0,$
and suppose that the limit cylinder is $\mathbb{R}^{3}\times \mathbb{S}^{1}_{\sqrt{2}}$, with $\mathbb{S}^{1}_{\sqrt{2}}$ being the 1-dimensional torus with radius $\sqrt{2},$ defined as, 
\begin{align}\label{eq:limitCylin}
\left[
\begin{array}{ccc}
y\\
\sqrt{2}cos\theta\\
\sqrt{2}sin\theta
\end{array}
\right], \ \text{with}\ y:=(y_{1},\ y_2,\  y_{3})^{T}\in \mathbb{R}^3,\ \theta\in [0,2\pi).
\end{align}

Thus in a (possibly shrinking) neighborhood of the origin, MCF can be parametrized by 
\begin{align}\label{eq:mcfu}
{\bf{x}}_t=\left[
\begin{array}{c}
x\\
u(x,\theta,t) cos\theta\\
u(x,\theta,t) sin\theta
\end{array}
\right],
\end{align}where $x\in \mathbb{R}^3$, $u$ is a positive function, is periodic in $\theta\in [0,2\pi)$, and is defined in the set $|x|\leq c(t)$ for some $c(t)>0$. The corresponding part of the rescaled MCF takes the form 
\begin{align}\label{eq:representation}
\frac{1}{\sqrt{T-t}}\left[
\begin{array}{c}
x\\
u(x,\theta,t) cos\theta\\
u(x,\theta,t) sin\theta
\end{array}
\right]=\left[
\begin{array}{ccc}
y\\
v(y, \theta,\tau)cos\theta\\
v(y, \theta,\tau)sin\theta
\end{array}
\right],
\end{align} where $v$ is a function defined in terms of $u$,
\begin{align}
u(x,\theta,t)=&\sqrt{T-t} \ v(y,\theta, \tau), \label{eq:rescaled}
\end{align} $y$ and $\tau$ are the rescaled spatial and time variables defined as
\begin{align}
y:=\frac{x}{\sqrt{T-t}},\ \tau:=-ln(T-t).\label{eq:scaling}
\end{align}

We are ready to state the main result. Recall the definition of $\mathcal{G}-$inner product in \eqref{def:Gin}. 
\begin{theorem}\label{THM:TwoReg}
Suppose the condition \eqref{eq:generic} holds, the blowup point is the origin, and the limit cylinder is the one parametrized by \eqref{eq:limitCylin}.

Then when $\tau$ is large and 
\begin{align}
|y|\leq 10\sqrt{\ln \tau},
\end{align}the rescaled MCF can be parametrized as in \eqref{eq:representation}, and moreover
one (and only one) of the following two possibilities must hold.
\begin{itemize}
\item[Case 1] For the first possibility, $v$ takes the form, for some $l\in \{1,2,3\}$, up to a rotation in $\mathbb{R}^3$, $y\rightarrow U y$,
\begin{align}
v(y, \theta,\tau)=\sqrt{\frac{2+\frac{1}{\tau}\sum_{k=1}^{l}y_k^2}{2a(\tau)}}+\eta(y, \sigma,\tau),
\end{align} where, for some $C>0,$ the function $a$ satisfies the estimate
\begin{align}
|a(\tau)-\frac{1}{2}|\leq C\tau^{-1},\label{eq:aparame}
\end{align} and
$\eta$ is considered remainder, its sharp $\mathcal{G}-$norm decay rate is
\begin{align}
\Big\|1_{|y|\leq 10\sqrt{\ln \tau}}\eta(\cdot,\tau)\Big\|_{\mathcal{G}}\leq C \tau^{-2},\label{eq:estEta}
\end{align} and in the weighted $L^{\infty}-$norm: for any $m+|k|+l=3$ and $m\geq 1,$
\begin{align}
\Big\|\langle y\rangle^{-m}1_{|y|\leq 10\sqrt{\ln \tau}}\partial_{\theta}^{l}\nabla_{y}^{k}\eta(\cdot,\tau)\Big\|_{\infty}\leq C (\sqrt{ln\tau})^{-m-1}. \label{eq:InftyEst}
\end{align}
\item[ Case 2]  The second possibility is that $v$ converges to $\sqrt{2}$ rapidly, specifically,
\begin{align}
\Big\|1_{|y|\leq 10\sqrt{\ln \tau}}\Big(v(\cdot,\tau)-\sqrt{2}\Big)\Big\|_{\mathcal{G}}\leq C\tau^{-3},
\end{align} and for any $m+|k|+l=3$ and $m\geq 1,$  
\begin{align}
\Big\|\langle y\rangle^{-m}1_{|y|\leq 10\sqrt{\ln \tau}}\partial_{\theta}^{l}\nabla_{y}^{k}\Big(v(\cdot,\tau)-\sqrt{2}\Big)\Big\|_{\infty}\leq C (\sqrt{ln\tau})^{-m-1}. \label{eq:InftyEst2}
\end{align}
\end{itemize}

\end{theorem}

Here $1_{|y|\leq Y}$, for any constant $Y>0$, is the standard Heaviside function, defined as
\begin{align}\label{eq:heavi}
1_{|y|\leq Y}(y):=\left[
\begin{array}{ll}
1\ \text{if}\ |y|\leq Y,\\
0\ \text{otherwise}.
\end{array}
\right.
\end{align}

The theorem will be proved in Section \ref{sec:pmainTHM}.

About the sharpness of the estimates, we have the following comments.
\begin{itemize}
\item[(A)] For the first case, the estimates \eqref{eq:aparame} and \eqref{eq:estEta} are sharp, by the known results for the blowup problem of nonlinear heat equations and rotationally symmetric MCF. However \eqref{eq:InftyEst} is not sharp, since we have to control terms produced by the cutoff functions, and the considered region is too small to make it sharp.
\item[(B)] About the two cases in the theorem, the first possibility is the most generic, and there are examples for the second, for example the initial hypersurface is $\mathbb{R}^3\times \mathbb{S}^1_{\sqrt{2}}$.
\end{itemize}


\section{the Effective Equations}\label{sec:Effec}

Recall that we assume that the blowup point is the origin, and assume the limit cylinder is the standard $\mathbb{R}^3\times \mathbb{S}^1_{\sqrt{2}}$.
Hence when $t$ sufficiently close to the blowup time $T$, in a (possibly shrinking) neighborhood of the origin, the MCF takes the form, 
\begin{align}\label{eq:oriMCF}
\left[
\begin{array}{c}
x\\
u(x,\theta,t) cos\theta\\
u(x,\theta,t) sin\theta
\end{array}
\right],
\end{align} where $x\in \mathbb{R}^3$, $u$ is a positive function, is periodic in $\theta\in [0,2\pi)$, and is defined in a set $|x|\leq c(t)$ for some $c(t)>0$. 

For the rescaled MCF, the part in \eqref{eq:oriMCF} becomes
\begin{align}\label{eq:vParame}
\left[
\begin{array}{c}
y\\
v(y,\theta,\tau) cos\theta\\
v(y,\theta,\tau) sin\theta
\end{array}
\right]=\frac{1}{\sqrt{T-t}}\left[
\begin{array}{c}
x\\
u(x,\theta,t) cos\theta\\
u(x,\theta,t) sin\theta
\end{array}
\right],
\end{align} where the variable $y$ and the function $v$ are naturally defined, and $\tau$
is defined as
\begin{align}
\tau:=-\ln(T-t).
\end{align}

To initiate our study,  we will derive some preliminary estimates for $v$ from \cite{ColdingMiniUniqueness} in Lemma \ref{LM:ColdMini} below.

As said earlier the parametrization (\ref{eq:vParame}) works only for a bounded set.
To measure the size of the controlled set, we define two functions $R_0, \ R_1:\ \mathbb{R}^{+}\rightarrow \mathbb{R}^{+}$ by the following identities:
\begin{align}
e^{\frac{1}{8}R^2_0(\tau)}=\tau^{\frac{18}{25}},\ \text{equivalently}, \ R_{0}(\tau):=\frac{12}{5}\sqrt{\ln\tau},\label{eq:defR0T}
\end{align} and
\begin{align}
e^{\frac{1}{8}R^2_1(\tau)}=\tau^{\frac{7}{10}},\ \text{equivalently}, \ R_{1}(\tau):=\sqrt{\frac{28}{5}}\sqrt{\ln\tau}.\label{eq:defR1T}
\end{align} Obviously $R_1<R_0.$ We need both functions since \eqref{eq:cm1} and \eqref{eq:cm2} below hold when $|y|\leq R_0$, and they will be derived from \cite{ColdingMiniUniqueness}. Then \eqref{eq:IniWeighted} will be derived from those two, thus it only holds in a smaller set, which is chosen to be $\Big\{y\ \Big|\ |y|\leq R_1\Big\}$.

We are ready to give some estimates for $v$, derived from \cite{ColdingMiniUniqueness}.
\begin{lemma}\label{LM:ColdMini}
There exists a constant $M$ such that if $\tau\geq M$ and $|y|\leq R_0(\tau)$, $v$ satisfies the estimates
\begin{align}
|v-\sqrt{2}|+ |\nabla_y v|+ |\partial_{\theta} v|\leq \tau^{-\frac{18}{25}} e^{\frac{1}{8}|y|^2},\label{eq:cm1}
\end{align} and for some constant $C$,
\begin{align}
\sum_{|k|+l=2}^{10}|\partial_{\theta}^{l}\nabla_{y}^{k}v|\leq C.\label{eq:cm2}
\end{align}

When $|y|\leq R_1(\tau)$ there exists a constant $\beta>0$ such that
\begin{align}
\Big\|\Big(v(\cdot,\tau)-\sqrt{2}\Big)1_{|y|\leq R_1}\Big\|_{\infty}+ \sum_{|k|+l=1,2,3,4}\Big\|1_{|y|\leq R_1}\nabla_{y}^{k}\partial_{\theta}^{l}v(\cdot,\tau)\Big\|_{\infty}\leq \tau^{-\beta}.\label{eq:IniWeighted}
\end{align}

\end{lemma}
Here $1_{|y|\leq R_1}$ is the Heaviside function taking value $1$ when $|y|\leq R_1(\tau),$ and $0$ otherwise.

The lemma will be proved in Section \ref{sec:LMColdMini}.

These estimates are important because they initiate our analysis, to help us to prove that the parametrization (\ref{eq:vParame}) works in a much larger region.

To measure the size of the considered region, we define a function $R:[\tau_0,\infty)\rightarrow \mathbb{R}^{+}$ as
\begin{align}
R(\tau):=\sqrt{\frac{26}{5}\ln\tau+100 \ln(1+\tau-\tau_0)},\label{eq:defRTau}
\end{align}where $\tau_0>0$ is a sufficiently large constant to be chosen later. There are two reasons of defining such a function $R:$
\begin{itemize}
\item[(1)] We need to make Lemma \ref{LM:ColdMini}
applicable when $\tau\in [\tau_0,\tau_0+20]$, in this time interval $R(\tau)<R_1(\tau)$. We need this to initiate our bootstrap, which is to be used to prove Main Theorem \ref{THM:TwoReg}, see the discussion before Lemma \ref{LM:Lya}.
\item[(2)] When $\tau\geq \tau_0+20,$ we apply our techniques to control an increasingly large region, so that when $\tau$ is sufficiently large, we have $ \dfrac{R(\tau)}{\sqrt{\ln \tau}}> 10.$ This is our goal.
\end{itemize}

In the rest of this section we prepare for proving Theorem \ref{THM:TwoReg} by decomposing $v$ according to the spectrum of the linearized operator in Lemma \ref{LM:Lya} below, and then deriving governing equations for these parts in the region $\big\{y\ |\ |y|\leq (1+\epsilon)R(\tau)\big\}$ in Theorem \ref{THM:CoMini} below.

To achieve this we use a bootstrap argument. This is necessary. Before proving the desired \eqref{eq:decomVToW}-\eqref{eq:weightPrelim2} in Lemma \ref{LM:Lya} and \eqref{eq:Beqn}-(\ref{eq:weightLInf}) in Theorem \ref{THM:CoMini} in the desired region, we need to prove the existence of the solution and find some primitive estimates. On the other hand, for technical reasons, only after proving some sufficiently good estimates in an interval, say $\tau\in[\tau_0,\tau_1]$, we can prove the existence of solution and find some primitive estimates in a slightly larger interval $[\tau_0,\tau_1+\kappa]$ for some small $\kappa>0$, as shown in Lemma \ref{LM:regularity}.

Consequently we need to bootstrap to prove the desired results.

To initiate the bootstrap we need Lemma \ref{LM:ColdMini}, as discussed after \eqref{eq:defRTau} above.

Now we start the bootstrap by stating the first result.
Recall the definitions of the inner product $\langle \cdot, \cdot\rangle_{\mathcal{G}}$ and $\perp_{\mathcal{G}}$ in \eqref{def:Gin}.
\begin{lemma}\label{LM:Lya}
Suppose that, for some sufficiently small constant $\delta>0$, in the space and time region 
\begin{align}
\tau\in [\tau_0,\tau_1], \ \text{and}\ y\in \Big\{ y\ \Big| \ |y|\leq (1+\epsilon)R(\tau)\Big\}
\end{align} with $\tau_1\geq \tau_0+20$, the following estimates hold
\begin{align}
\Big|v(y,\theta,\tau)- \frac{1}{2}\Big|+\sum_{|k|+l=1}^4\Big|\nabla_{y}^{k}\partial_{\theta}^{l} v(y,\theta,\tau)\Big|\leq \delta.\label{eq:condition1}
\end{align} Here $\tau_0\gg 1$ is the same to that in the definition of $R(\tau)$ and $\epsilon$ is a positive constant in the definition of cutoff function $\chi_R$, see \eqref{eq:reCutoff} below. 

Then in the same space and time region the following estimates hold.

There exist unique functions $a$, $\alpha_l, \ l=1,2$, a $3\times 3$ symmetric real matrix-valued function $B$, and $3-$dimensional vector-valued functions $\Omega_{k},\ k=1,2,3,$ such that
\begin{align}
\begin{split}\label{eq:decomVToW}
v(y,\theta,\tau)= &V_{a(\tau),B(\tau)}(y)+{\Omega}_1(\tau)\cdot y +{\Omega}_2(\tau)\cdot y cos\theta + {\Omega}_3(\tau)\cdot y sin\theta\\ 
&+\alpha_1(\tau) cos\theta +\alpha_2(\tau) sin\theta +w(y,\theta,\tau),
\end{split}
\end{align}
and the function $\chi_{R}w$ is $\mathcal{G}-$orthogonal to the following 18 functions
\begin{align}
\begin{split}\label{eq:orthow} 
 &1,\  cos\theta,\  sin\theta,  \ 
y_k,\ \frac{1}{2}y_k^2-1,\  y_k cos\theta,
  y_k sin\theta,\   y_m y_n\ \text{with}\  m\not=n \ \text{and}\ k, m,n=1,2,3,
\end{split}
\end{align}
where $V_{a,B}$ and $\chi_{R}$ are two functions to be defined in \eqref{eq:defVaB} and \eqref{eq:reCutoff} below,
moreover
\begin{align}
\Big|a-\frac{1}{2}\Big|+ |B|+ \sum_{n=1}^3|\Omega_n|+ \sum_{l=1,2}|\alpha_l|&\lesssim \tau^{-\frac{3}{5}}, \label{eq:paramePrelim}\\
\sum_{|k|+l=0,1}\Big\|\nabla_{y}^k\partial_{\theta}^l\chi_{R}w\Big\|_{\mathcal{G}}&\lesssim \tau^{-\frac{3}{5}},\label{eq:weightPrelim1}\\
\sum_{|k|+l=2}\Big\|\nabla_y^{k}\partial_{\theta}^l\chi_{R} w\Big\|_{\mathcal{G}}&\lesssim \tau^{-\frac{3}{10}}.\label{eq:weightPrelim2}
\end{align}
\end{lemma}
The lemma will be proved in Section \ref{sec:DecomVW} below. The reasons of decomposing $v$ as in \eqref{eq:decomVToW} will be explained after \eqref{eq:Tchi3}, when we are ready.

Here the function $V_{a,B}$, for any $3\times 3$ symmetric real matrix $B$ and scalar $a$,  is defined as
\begin{align}
V_{a,B}(y):=\sqrt{\frac{2+ y^{T}B y}{2a}}.\label{eq:defVaB}
\end{align}As analyzed in \cite{GS2008}, $\sqrt{2+ y^{T}B y}$ is a solution to the equation
\begin{align}
-\frac{1}{2}y\cdot\nabla_{y}v+\frac{1}{2}v-\frac{1}{v}=0.
\end{align}This equation is believed to be the ``main part" of the governing equation for $v$ \eqref{eq:scale1} below, thus we believe, and will prove, that $V_{a, B}$ is the main part of $v$, provided that $a$ and $B$ are chosen correctly. 

To prepare for defining the cutoff function $\chi_{R}$, we define a spherically symmetric cutoff function $\chi:\mathbb{R}^3\rightarrow \mathbb{R}$ such that it is in $C^{19,1}$ and satisfies the condition
\begin{align}\label{eq:defChi3}
\chi(z)=\chi(|z|)=\ \Big[
\begin{array}{lll}
1,\ \text{if}\ |z|\leq 1,\\
0,\ \text{if}\ |z|\geq 1+\epsilon.
\end{array}
\
\end{align}We require it decreases in $|z|$, and there exist constants $M_k=M_k(\epsilon), \ k=0, 1,\cdots, 5,$ such that for any $z$ satisfying $0\leq 1+\epsilon-|z|\ll 1$, $\chi$ satisfies the estimate
\begin{align}
\begin{split}\label{eq:properties}
\frac{d^{k}}{d|z|^{k}}\chi(|z|)=&10M_k(|z|-1-\epsilon)^{20-k}+\mathcal{O}\Big((|z|-1-\epsilon)^{21-k}\Big),\ k=0,1,2,3,4.
\end{split}
\end{align}
Such a function is easy to construct, we skip the details here. (\ref{eq:properties}) will be used in controlling terms produced by the cutoff function, see e.g. \eqref{eq:unboundtwo} below.

Now we define the cutoff function $\chi_{R}$ as 
\begin{align}
\chi_{R}(y):=\chi(\frac{y}{R}).\label{eq:reCutoff}
\end{align}To control terms produced by $\chi_{R}$, we define a constant
$\kappa(\epsilon)$ as
\begin{align}
\kappa(\epsilon):=\sum_{k=1}^{5}\sup_{|z|}\Big|\frac{d^{k}}{d|z|^{k}} \chi(|z|)\Big|.\label{eq:defKappa}
\end{align}

We are ready for the second step of the bootstrap. Here we prepare for improving the decay rates of various parts by estimating their governing equations.
\begin{theorem}\label{THM:CoMini}
Suppose the conditions in \eqref{eq:condition1} hold in the space-and~time region
\begin{align}
\tau\in [\tau_0, \tau_1],\  \text{and}\ |y|\leq (1+\epsilon)R(\tau).
\end{align}

Then the governing equations for the functions in \eqref{eq:decomVToW} satisfy the following estimates: for some constant $c>0$, 
\begin{align}
|\frac{d}{d\tau}B+ B^{T}B|\leq &c\Big(H_1 +\tilde\delta e^{-\frac{1}{5}R^2}\Big),\label{eq:Beqn}\\
|\Big(\frac{1}{a}\frac{d}{d\tau}-2\Big)\Big(a-\frac{1}{2}-\frac{1}{2}(b_{11}+b_{22}+b_{33})\Big)|\leq &c\Big(|B|^2+H_1 +\tilde\delta e^{-\frac{1}{5}R^2}\Big),\label{eq:Aeqn}\\
|\frac{d}{d\tau} \Omega_1-a\Big(1+\mathcal{O}(|B|)\Big)\Omega_1|\leq &c\Big(H_1 +\tilde\delta e^{-\frac{1}{5}R^2}\Big),\label{eq:beta1}\\
|\frac{d}{d\tau}\Omega_2|+ |\frac{d}{d\tau}\Omega_3|\leq &c\Big(H_2 +\tilde\delta e^{-\frac{1}{5}R^2}\Big),\label{eq:beta2Eqn}\\
|\frac{d}{d\tau}\alpha_1-\frac{1}{2} \alpha_1|+ |\frac{d}{d\tau}\alpha_2-\frac{1}{2}\alpha_2|\leq &c\Big(H_2 +\tilde\delta e^{-\frac{1}{5}R^2}\Big).\label{eq:alpha1Eqn}
\end{align}Here the functions $H_1$ and $H_2$ and the constant $\tilde{\delta}$ are to be defined in \eqref{eq:defH1}, \eqref{eq:defH2} and \eqref{eq:defTDelta} below.
$\chi_{R}w$ satisfies the following estimates, 
\begin{align}
\begin{split}\label{eq:wL2Est}
\Big[\frac{d}{d\tau}+\frac{1}{4}\Big]&\|\chi_{R}w\|_{\mathcal{G}}^2\\
\leq & c\Big[\Big(|B|+\sum_{k=1}^{3}|\Omega_k|+\sum_{l=1,2}|\alpha_l|\Big)^4+\tilde\delta\|\partial_{\theta}^2 \chi_{R}w\|_{\mathcal{G}}^2 +\tilde\delta e^{-\frac{1}{5}R^2}\Big],
\end{split}
\end{align} 
\begin{align}
\begin{split}\label{eq:YwL2Est}
\Big[\frac{d}{d\tau}+\frac{1}{4}\Big]&\|\nabla_y \chi_{R}w\|_{\mathcal{G}}^2\\
\leq &c \Big[\big[|B|+\sum_{k=1}^{3}|\Omega_k|+\sum_{l=1,2}|\alpha_l|\big]^4+\tilde\delta\|\partial_{\theta}^2 \chi_{R}w\|_{\mathcal{G}}^2+\tilde\delta e^{-\frac{1}{5}R^2}\Big],
\end{split}
\end{align}
\begin{align}
\begin{split}\label{eq:TwL2Est}
\Big[\frac{d}{d\tau}+\frac{1}{4}\Big]&\Big[\|\partial_{\theta}^2 \chi_{R}w\|_{\mathcal{G}}^2+\|\partial_{\theta}\nabla_y \chi_{R}w\|_{\mathcal{G}}^2\Big]\\
\leq&c \Big[\Big(\sum_{k=2}^{3}|\Omega_k|+\sum_{l=1,2}|\alpha_l|\Big)^2 \Big(|B|+\sum_{k=1}^{3}|\Omega_k|+\sum_{l=1,2}|\alpha_l|\Big)^2+\tilde\delta e^{-\frac{1}{5}R^2} \Big],
\end{split}
\end{align}
and lastly,
\begin{align}
\begin{split}\label{eq:TwYL2Est}
\Big[\frac{d}{d\tau}+\frac{1}{4}\Big]&\sum_{|k|=2}\|\nabla_{y}^{k}\chi_{R} w\|_{\mathcal{G}}^2\\
\leq &c \Big[\Big(|B|+\sum_{k=1}^{3}|\Omega_k|+\sum_{l=1,2}|\alpha_l|\Big)^4+\tilde\delta \sum_{|k|+l=1} \|\nabla_{y}^{k}\partial_{\theta}^{l} \chi_{R}w\|_{\mathcal{G}}^2
+\tilde\delta e^{-\frac{1}{5}R^2}\Big].
\end{split}
\end{align}

\eqref{eq:Beqn}-\eqref{eq:alpha1Eqn} and \eqref{eq:paramePrelim}-\eqref{eq:weightPrelim2} imply that
\begin{align}
|\frac{d}{d\tau}a|+ |\frac{d}{d\tau}B|+ \sum_{k=1}^{3}|\frac{d}{d\tau}\Omega_k|+ \sum_{l=1,2}|\frac{d}{d\tau}\alpha_l| \lesssim \tau^{-\frac{11}{20}}.\label{eq:TauAB}
\end{align}

Lastly,
\begin{align}
\begin{split}\label{eq:weightLInf}
\|\langle y\rangle^{-3} \chi_{R}w\|_{\infty}\leq &c\tilde\delta \kappa(\epsilon) R^{-4}(\tau),\\
\|\langle y\rangle^{-2}\nabla_{y}^{k}\partial_{\theta}^{l}\chi_{R} w\|_{\infty}\leq &c\tilde\delta \kappa(\epsilon) R^{-3}(\tau),\  |k|+l=1,\\
\|\langle y\rangle^{-1}\partial_{\theta}\nabla_{y}^{k}\chi_{R} w\|_{\infty}\leq &c\tilde\delta \kappa(\epsilon) R^{-2}(\tau),\ |k|+l=2.
\end{split}
\end{align}Here the constant $\kappa(\epsilon)$ is defined in \eqref{eq:defKappa}.

\end{theorem}
The theorem will be proved in subsequent sections. We will derive \eqref{eq:beta2Eqn} in detail in Section \ref{sec:beta2Eqn}. The difficulty is that, to prepare for proving that $\Omega_2$ and $\Omega_3$ decay rapidly, we need to observe many cancellations. \eqref{eq:alpha1Eqn} and \eqref{eq:beta1} can be derived similarly, and equations similar to \eqref{eq:Beqn} and \eqref{eq:Aeqn} were derived in our previous works \cite{GaKn20142} and \cite{DGSW, GS2008, GaKnSi}, thus we will skip these parts. 

\eqref{eq:wL2Est}-\eqref{eq:TwL2Est} will be proved in Sections \ref{sec:wL2Est} - \ref{sec:TwL2Est}. 
The proof of \eqref{eq:TwYL2Est} is similar and easier than that of \eqref{eq:YwL2Est}, hence will be skipped.
The proof of \eqref{eq:weightLInf} is the most involved, and will be reformulated in Section \ref{sec:ReforWeightLInf} below.

Next we define the functions $H_1$, $H_2$ and the constant $\tilde{\delta}$ used above.

$H_1$ is defined to control the various parts in the decomposition of $v$,
\begin{align}
\begin{split}\label{eq:defH1}
H_{1}(\tau):=&|B(\tau)|^3+|B(\tau)|^2 |a_{\tau}|+|\Omega_1(\tau)|^3+\sum_{k=2}^{3}|\Omega_k(\tau)|^2+\sum_{l=1,2}|\alpha_l(\tau)|^2\\
&+\Psi(\tau)\Big[|B(\tau)|+\sum_{k=1}^{3}|\Omega_k(\tau)|+\sum_{l=1,2}|\alpha_l(\tau)|\Big]+R^4(\tau) \Psi^2(\tau),
\end{split}
\end{align}
and $H_2$ is defined to control the rapidly decaying $\theta-$dependence parts of $v,$
\begin{align}
\begin{split}\label{eq:defH2}
H_2(\tau):=&\Big[\sum_{k=2,3}|\Omega_k(\tau)|+\sum_{l=1,2}|\alpha_l(\tau)|\Big]\Big[|B(\tau)|^2+\sum_{k=1}^{3}|\Omega_k(\tau)|^2+\sum_{l=1,2}|\alpha_l(\tau)|^2+\Psi(\tau)\Big]\\
&+R^4(\tau)\Psi(\tau)\Big[\|\chi_{R}\partial_{\theta}^2 w(\cdot,\tau)\|_{\mathcal{G}}+\|\partial_{\theta}\nabla_y \chi_{R}w(\cdot,\tau)\|_{\mathcal{G}}\Big],
\end{split}
\end{align}
where the term $\Psi(\tau)$ is defined as
\begin{align}
\begin{split}
\Psi(\tau):=&\sum_{|k|+l=0,1,2}\Big\|\nabla_y^{k} \partial_{\theta}^{l}\chi_{R}w(\cdot,\tau)\Big\|_{\mathcal{G}}.
\end{split}
\end{align}$\tilde\delta$ is a positive constant defined as, recall the constant $\delta$ from (\ref{eq:condition1}),
\begin{align}
\tilde\delta:=\delta+\tau_{0}^{-\frac{1}{2}}.\label{eq:defTDelta}
\end{align}


Now we state the last result in the bootstrap. 
\begin{lemma}\label{LM:regularity}
Suppose that in the region
\begin{align}
\tau\in [\tau_0, \tau_1],\ \text{and} \ y\in \Big\{y\ \Big|\ |y|\leq R(\tau)\Big\}
\end{align}
with $\tau_1-\tau_0\geq 20$, $v$ satisfies the estimates,
\begin{align}
\Big|v(\cdot,\tau)-\sqrt{2}\Big|+\sum_{|k|+l=1,2}\Big|\nabla_{y}^{k}\partial_{\theta}^{l} v(\cdot,\tau)\Big|\leq R^{-\frac{1}{2}}(\tau).\label{eq:PointSmall}
\end{align}

Then for any $\delta>0,$ provided that $\tau_0$ is sufficiently large, there exists some small constant $\kappa=\kappa(\delta)>0$ such that, at the time
$
\tau=\tau_1+\kappa$ and in the region $ |y|\leq (1+\frac{1}{2}\kappa)\Big(R(\tau)-1\Big), 
$
\begin{align}
\Big|v(\cdot,\tau)-\sqrt{2}\Big|+\sum_{|k|+l=1}^4 \Big|\nabla_{y}^{k}\partial_{\theta}^{l} v(\cdot,\tau)\Big|\leq \delta.\label{eq:extenSmooth}
\end{align}
\end{lemma}
\begin{proof}
The main tools are the standard techniques of local smooth extension, see \cite{EckerBook}, and comparing the rescaled MCF to MCF used in \cite{ColdingMiniUniqueness}. See also \cite{MR485012, white05, MR1196160, CIM13}.

To make the tools applicable, we fix a time $\tau_1$ and define a new MCF by rescaling the one in \eqref{eq:MeanCur},
\begin{align}
{\bf{y}}_{s}:=\frac{1}{\lambda} {\bf{x}}_t
\end{align} where $s$ is the new time variable, $\lambda$ is a constant, defined as \begin{align}
s:=\lambda^{-2}(t-t_1), \ \text{and}\ \lambda:=\sqrt{T-t_1}.
\end{align} Recall that $\tau$ is defined by $\tau=-\ln (T-t)$, here $t_1$ is the unique time $t$ such that $\tau(t)=\tau_1.$

For the new MCF, the part \eqref{eq:mcfu} of the old one becomes, 
 \begin{align}\label{eq:flow2}
 \left[
\begin{array}{ccc}
z\\
p(z, \theta, s)\cos\theta\\
p(z,\theta,s) \sin \theta
\end{array}
\right]=\frac{1}{\lambda}
\left[
\begin{array}{c}
x\\
u(x,\theta,t) cos\theta\\
u(x,\theta,t) sin\theta
\end{array}
\right],
\end{align}
where $z=\frac{x}{\lambda}$ is the new spatial variable, $p(z,\theta,s)$ is a function defined in terms of $u$, and hence of $v$ through \eqref{eq:rescaled},
\begin{align}
\begin{split}\label{eq:restart}
p(z,\theta,s):
=&\frac{1}{\lambda} u(\lambda z, \theta,  \lambda^2 s+t_1)\\
=&\sqrt{1-s}\ v\big(\frac{ z}{\sqrt{1-s}},\theta, -\ln(1-s)+\tau_1\big).
\end{split}
\end{align} 
We derive estimates for $p$ through those for $v$ in \eqref{eq:extenSmooth} in the region $\tau\in [\tau_0, \tau_1]$  and $|y|\leq R(\tau)$. They imply that, when 
\begin{align}
s\in [-1,0]\ \text{and} \ |z|\leq R(\tau_1), 
\end{align}
$p$ is uniformly bounded in $C^2.$ 
Especially, when $s=0$, since $p(z,\theta,0)=v(z,\theta,\tau_1)$, \eqref{eq:PointSmall} implies
\begin{align}
\Big|p(z,\theta,0)-\sqrt{2}\Big|+\sum_{|k|+l=1,2} \Big|\nabla_{z}^{k}\partial_{\theta}^{l} p(z,\theta,0)\Big|\leq 10  R^{-\frac{1}{2}}(\tau_1).
\end{align}

We are ready to apply the techniques of local smooth extension and interpolation between the estimates on the derivatives. For any $\delta>0$, provided that $\tau_0$ is large enough so that $R^{-1}(\tau)\leq R^{-1}(\tau_0)$ is small enough, there exists a positive constant $\kappa$ such that in the region
\begin{align}
0\leq s\leq \kappa\ \text{and} \ |z|\leq R(\tau_1)-1, 
\end{align}$p$ satisfies the estimates
\begin{align}
\Big|p(z,\theta,s)-\sqrt{2}\Big|+ \sum_{|k|+l=1}^{4}\Big|\nabla_{x}^{k}\partial_{\theta}^{l} p(z,\theta,s)\Big|\leq \delta.
\end{align}

This and the identity in \eqref{eq:restart} imply the desired estimates for $v$ in \eqref{eq:extenSmooth}.
\end{proof}

\subsection{the governing equation for $\chi_{R}w$}
In this subsection we derive a governing equation for $\chi_{R}w$, and then present an intuitive reason of decomposing $v$ in \eqref{eq:decomVToW}.

Before deriving a governing equation for $\chi_{R}w$ we need one for $v,$ which comes from that for $u$.

We derive a governing equation for $u$ from the $0-$level set of the function $f$, $$f(x_1,x_2,\cdots, x_5,t)=0,$$ where $f$ is defined as
\begin{align}
f(x_1,x_2,\cdots,x_5,t):=\sqrt{x_4^2+x_5^2}-u(x_1,x_2,x_3,\theta,t), 
\end{align} with $\theta$ being the angle on the $x_4,\ x_5$ plane. It was shown in \cite{MR1770903} that $f$ satisfies the equation
\begin{align}
\partial_{t}f=\sum_{i,j}\Big(\delta_{ij}-\frac{f_{x_i}f_{x_j}}{|Df|^2}\Big)f_{x_ix_j}.
\end{align}
By this we derive a parabolic differential equation for $u$,
\begin{align}
\begin{split}\label{eq:MCF}
\partial_{t}u=&\sum_{ k,\ l= 1}^{3}\big[\delta_{k,l}-\frac{\partial_{x_k}u\partial_{x_l}u}{\Upsilon}\big]\partial_{x_k}\partial_{x_l}u+u^{-2}\frac{1+|\nabla_x u|^2}{\Upsilon}\partial_{\theta}^2 u+
u^{-2}\frac{2\partial_{\theta}u}{\Upsilon}\sum_{l=1}^{3}\partial_{x_l}u\partial_{x_l}\partial_{\theta}u\\
&+\frac{1}{\Upsilon}\frac{(\partial_{\theta}u)^2}{u^3}-\frac{1}{u},
\end{split}
\end{align} where the function $\Upsilon$ is defined as 
\begin{align*}
\Upsilon:=1+|\nabla_x u|^2+(\frac{\partial_{\theta}u}{u})^2.
\end{align*}

\eqref{eq:MCF} implies a governing equation for $v$, though the identity \eqref{eq:rescaled},
\begin{align}
\partial_{\tau}v=\Delta_{y} v+v^{-2}\partial_{\theta}^2 v-\frac{1}{2}y\cdot\nabla_{y}v+\frac{1}{2}v-\frac{1}{v}+N_1(v)\label{eq:scale1}
\end{align} where $N_1(v)$ will be treated as remainder, and is defined as
\begin{align}
\begin{split}\label{eq:defFu}
N_1(v):=\Upsilon^{-1}\Big[&-\sum_{k=1}^{3}(\partial_{y_k}v)^{2} \partial^{2}_{y_k}v
-v^{-2} (v^{-1}\partial_{\theta}v)^{2} \partial^{2}_{\theta}v+v^{-2}2\partial_{\theta}v
\sum_{l=1}^{3}\partial_{y_l}v\partial_{y_l}\partial_{\theta}v\\
&+\frac{(\partial_{\theta}v)^2}{v^3}-\sum_{i\not= j}\partial_{y_i} v \partial_{y_j} v
\partial_{y_i}\partial_{y_j}v\Big].
\end{split}
\end{align}

From the decomposition of $v$ in \eqref{eq:orthow} we derive a governing equation for $w$:
\begin{align}
\partial_{\tau}w=-Lw+F(B,a)+G(\Omega,\ \alpha)+N_1(v)+N_2(\eta),\label{eq:eqnw}
\end{align}
where the linear operator $L$ is defined as
\begin{align}
L:=-\Delta_y+\frac{1}{2}y\cdot \nabla_{y}-V^{-2}_{a,B}\partial_{\theta}^2-\frac{1}{2}-V^{-2}_{a,B} ,\label{eq:aBeqn}
\end{align} 
the term $N_2(\eta)$ is defined as
\begin{align}
\begin{split}\label{eq:defN2eta}
N_2(\eta):= &-v^{-1}+V_{a,B}^{-1}-V_{a,B}^{-2}\eta+\big(v^{-2}-V^{-2}_{a,B}\big)\partial_{\theta}^2 \eta\\
=&-V_{a,B}^{-2}v^{-1} \eta^2-v^{-2}V^{-2}_{a,B}(v+V_{a,B})\eta \partial_{\theta}^2 \eta,
\end{split}
\end{align}
the function $\eta$ is defined as
\begin{align}
\begin{split}\label{eq:decomW}
\eta:=&v-V_{a,B}
={\Omega}_1\cdot y +{\Omega}_2\cdot y cos\theta + {\Omega}_3\cdot y sin\theta +\alpha_1 cos\theta +\alpha_2 sin\theta 
+w
\end{split}
\end{align} 
and
the function $F(B,a)$ is defined as
\begin{align}
\begin{split}\label{eq:source}
F(B,a)
:=& -\frac{y^{T}(\partial_{\tau}B +B^{T}B) y}{2 \sqrt{2a} \sqrt{2+y^{T}B y}}+\frac{1}{\sqrt{2a}\sqrt{2+y^{T}By}}\Big[\frac{a_{\tau}}{a}+1-2a+b_{11}+b_{22}+b_{22}\Big]\\
&+\frac{y^{T}B^{T}B y \ y^{T}B y}{2 \sqrt{2a} (2+y^{T}B y)^{\frac{3}{2}}}+\frac{a_{\tau}}{(2a)^{\frac{3}{2}}} \frac{y^{T}By}{\sqrt{2+y^{T}By}},
\end{split}
\end{align}
and the function $G(\Omega,\alpha)=G(\Omega_1, \Omega_2,\ \Omega_3,\ \alpha_1,\ \alpha_2)$ is defined as
\begin{align}
G(\Omega, \alpha)
:=&-\Big[L+\frac{d}{d\tau}\Big] \Big[{\Omega}_1\cdot y +{\Omega}_2\cdot y cos\theta + {\Omega}_3\cdot y sin\theta +\alpha_1 cos\theta +\alpha_2 sin\theta \Big]\nonumber\\
=&\Big[\frac{2a}{2+y^{T}By}\Omega_1-\frac{d}{d\tau}\Omega_1\Big]\cdot y-\frac{d}{d\tau}\Omega_2 \cdot y cos\theta-\frac{d}{d\tau}\Omega_3\cdot y sin\theta\nonumber\\
&+\Big[\frac{1}{2}\alpha_1-\frac{d}{d\tau}\alpha_1\Big] cos\theta+\Big[\frac{1}{2}\alpha_2-\frac{d}{d\tau}\alpha_2\Big] sin\theta.\nonumber
\end{align}
Here certain terms cancel each other by the identities $\frac{d^2}{d\theta^2}cos\theta+cos\theta=0$ and $\frac{d^2}{d\theta^2}sin\theta+sin\theta=0.$

Impose the cutoff function $\chi_{R}$ onto \eqref{eq:eqnw} to derive an equation for $\chi_{R}w$
\begin{align}
\partial_{\tau}(\chi_{R}w)=&-L(\chi_{R}w)+\chi_R \Big( F(B,a)+G(\Omega, \alpha)
+N_1(v)+N_2(\eta)\Big)+\mu(w),\label{eq:tildew3}
\end{align} where the term $\mu(w)$ is defined as 
\begin{align}
\mu(w):=\frac{1}{2}\big(y\cdot\nabla_{y}\chi_{R}\big)w+\big(\partial_{\tau}\chi_{R}\big)w-\big(\Delta_{y}\chi_{R}\big)w-2\nabla_{y}\chi_{R}\cdot  \nabla_{y}w.\label{eq:Tchi3}
\end{align}

Now we are ready to explain the reason of imposing the orthogonality condition \eqref{eq:orthow} on $\chi_{R}w$.

Since $|B|+ |a-\frac{1}{2}|\ll1$, the linear operator $L$ is ``approximately" $L_0$, defined as $$ L_0:=-\Delta_y+\frac{1}{2}y\cdot \nabla_{y}-\frac{1}{2}\partial_{\theta}^2-1.$$
$L_0$ is conjugate to a self-adjoint operator $\mathcal{L}_0$, defined as
\begin{align}
\mathcal{L}_0:=e^{-\frac{1}{8} |y|^2} L_{0}e^{\frac{1}{8} |y|^2} =\sum_{k=1}^{3}\Big(-\partial_{y_k}^2+\frac{1}{16} y_k^2-\frac{1}{4}\Big)-1-\frac{1}{2}\partial_{\theta}^2.
\end{align} 
Here $\sum_{k=1}^{3}\Big(-\partial_{y_k}^2+\frac{1}{16} y_k^2-\frac{1}{4}\Big)$ is a harmonic oscillator with eigenvalues $\frac{1}{2}k,\ k=0,1,2,\cdots,$ and those of $-\partial_{\theta}^2$ are $\{n^2|\ n=0,1,\cdots\}.$ Since these two operators commute, the eigenvalues of $\mathcal{L}_0$ are $-1+\frac{1}{2}(k+n^2)$ with $k,n=0,1,2,\cdots .$
The directions corresponding to the positive eigenvalues of $\mathcal{L}_0$ are easy to control since, on the linear level, they decay exponentially fast.  Thus the difficulty is to control the directions with nonpositive eigenvalues. The nonpositive eigenvectors and eigenvalues are the followings: 
\begin{itemize}
\item $e^{-\frac{1}{8} |y|^2}$ with eigenvalue $-1;$
\item  $e^{-\frac{1}{8} |y|^2} y_k, \ k=1,2,3, \ e^{-\frac{1}{8} |y|^2} cos\theta,\ e^{-\frac{1}{8} |y|^2} sin\theta$ with eigenvalue $-\frac{1}{2};$ 
\item $e^{-\frac{1}{8} |y|^2}(\frac{1}{2}y_k^2-1)$, $e^{-\frac{1}{8} |y|^2} y_k cos\theta$, $e^{-\frac{1}{8} |y|^2} y_k sin\theta$ and $e^{-\frac{1}{8} |y|^2} y_m y_n, \ m\not=n, \ \ k, m,n=1,2,3,$ with eigenvalue $0.$
\end{itemize}

By making $e^{-\frac{1}{8} |y|^2}\chi_{R}w$ orthogonal to these eigenvectors in \eqref{eq:orthow}, at least intuitively, we expect $e^{-\frac{1}{8} |y|^2}\chi_{R}w$ to decay. This is the reason we impose the orthogonality conditions \eqref{eq:orthow}


\section{Proof of the Main Theorem \ref{THM:TwoReg}}\label{sec:pmainTHM}
We will prove the desired results in three steps. In the first step, which is in subsection \ref{subsec:effEqn} below, we prove the decomposition \eqref{eq:decomVToW} and the estimates \eqref{eq:Beqn}-\eqref{eq:weightLInf} hold in the time interval $\tau\in [\tau_0,\infty)$ and in the region $\Big\{ y\ \Big|\ |y|\leq (1+\epsilon)R(\tau)\Big\}$. 

In this step we prove part of Main Theorem \ref{THM:TwoReg}, specifically, \eqref{eq:InftyEst} is exactly \eqref{eq:weightLInf}. This is the most involved part of the present paper.

Also, in this step it is important to prove \eqref{eq:Beqn}- \eqref{eq:TwYL2Est},  even though they are not part of Main Theorem \ref{THM:TwoReg},  because they help our induction, and eventually provide the wanted estimates. Here we have to observe some cancellations in \eqref{eq:zeros} before proving \eqref{eq:beta2Eqn}.

In the second step, which is in subsection \ref{subsec:thetaPart}, we prove that the $\theta-$dependent components of $v$ decay rapidly. They include $\Omega_2,\ \Omega_3$, $\alpha_1$, $\alpha_2$, $\|e^{-\frac{1}{8}|y|^2} \partial_{\theta}\chi_{R} w\|_{2}$, $\|e^{-\frac{1}{8}|y|^2} \partial_{\theta}^2\chi_{R} w\|_{2}$ and $\|e^{-\frac{1}{8}|y|^2} \partial_{\theta}\nabla_y\chi_{R} w\|_{2}$. The main result is Lemma \ref{LM:improve1}. 

In the last step, which is in subsection \ref{subsec:nontheta}, we focus on the other parts, specifically, $a$, $B$ and $\Omega_1$, and $\sum_{|k|\leq 2}\|e^{-\frac{1}{8}|y|^2} \nabla_y^k \chi_{R} w\|_{2}$. The main result is Lemma \ref{LM:assigned}. Since we already proved that the $\theta-$dependent part decays rapidly, this step is similar to the study of spherically symmetric MCF.

Lemmas \ref{LM:improve1} and \ref{LM:assigned} imply the desired estimates for the corresponding parts in Theorem \ref{THM:TwoReg}. Thus we will complete the proof of Main Theorem \ref{THM:TwoReg} after carrying out these steps.

The following estimate will be used very often in the rest of the paper: Suppose that $f: [0,\ 2\pi]\rightarrow \mathbb{R}$ is a smooth periodic function with $f(0)=f(2\pi)$, then
\begin{align}\label{eq:poincare}
\|\partial_{\theta}f\|_2\leq \|\partial_{\theta}^2 f\|_2.
\end{align}The proof is easy: Fourier-expand $f$ to find $f(\theta)=\sum_{n=-\infty}^{\infty}f_n e^{in\theta}$, thus
\begin{align}
\|\partial_{\theta}f\|_2=\sqrt{\sum_{n=-\infty}^{\infty}n^2|f_n|^2}\leq \sqrt{\sum_{n=-\infty}^{\infty}n^4|f_n|^2}=\|\partial_{\theta}^2f\|_2.
\end{align}

\subsection{Proof of that \eqref{eq:paramePrelim}-\eqref{eq:weightLInf} hold for $\tau\in [\tau_0, \infty)$}\label{subsec:effEqn}
The strategy is to apply a standard bootstrap argument. The intuitive ideas were presented before Lemma \ref{LM:Lya}. Here present a rigorous version.

To initiate the bootstrap we prove that \eqref{eq:condition1} holds when $|y|\leq (1+\epsilon)R(\tau)$ and $\tau\in [\tau_0, \ \tau_0+20]$, for some small $\epsilon>0$ and some sufficiently large $\tau_0$.

This is made true by Lemma \ref{LM:ColdMini}, which provides estimates for $v$ when $|y|\leq R_1(\tau)$ for any large $\tau.$ If $\epsilon$ is small enough and $\tau_0$ is large enough, it can provide estimates in the region $|y|\leq (1+\epsilon) R(\tau)$ when $\tau\in [\tau_0,\ \tau_0+20]$, since for some constant $\epsilon_0>0$,  
\begin{align}
\frac{R_1(\tau)}{R(\tau)}\geq 1+\epsilon_0.
\end{align}where, recall that $R_1$, $R_0$, $R$ are defined in \eqref{eq:defR0T}, \eqref{eq:defR1T} and \eqref{eq:defRTau}.

For the next step of bootstrap, suppose that \eqref{eq:condition1} holds in a time interval $[\tau_0,\ \tau_1],$ with $\tau_1-\tau_0\geq 20.$ 

This makes Lemma \ref{LM:Lya} and Theorem \ref{THM:CoMini} applicable in the same interval. Thus, the decomposition of $v$, \eqref{eq:paramePrelim} and \eqref{eq:weightLInf} imply that \eqref{eq:PointSmall} holds here. Consequently Lemma \ref{LM:regularity} becomes applicable, and implies that \eqref{eq:condition1} holds when $\tau\in [\tau_0, \tau_1+q]$ for some $q>0.$  

Thus \eqref{eq:condition1} actually holds in an interval larger than the one we assumed!

By induction and continuity, all these results above hold in $[\tau_0,\ \infty).$

The proof is complete.

\subsection{Estimates for the $\theta-$dependent part of $v$}\label{subsec:thetaPart}
In this subsection we estimate the fast-decaying $\theta-$dependent parts of $v$. The result is:
\begin{lemma}\label{LM:improve1}
There exists a time $\tau_1>0$ such that for any $\tau\geq \tau_1$,
\begin{align}
\sum_{k=2,3}|\Omega_k(\tau)|+\sum_{l=1,2} |\alpha_l(\tau)|\leq &\tau^{-5},\label{eq:desirePi2}\\
\| \partial_{\theta}\chi_{R} w(\cdot,\tau)\|_{\mathcal{G}}+
\ \| \partial_{\theta}^2\chi_{R} w(\cdot,\tau)\|_{\mathcal{G}}+\| \partial_{\theta}\nabla_y\chi_{R} w(\cdot,\tau)\|_{\mathcal{G}}\lesssim &\tau^{-\frac{11}{2}}.\label{eq:desirePi3}
\end{align}
\end{lemma}
\begin{proof}
We prove the results by induction.

To initiate it we observe that, when $\tau$ is large enough, \eqref{eq:paramePrelim} in Theorem \ref{THM:CoMini} implies
 \begin{align}
 \sum_{k=2,3}|\Omega_k(\tau)|+\sum_{l=1,2} |\alpha_l(\tau)|\leq \tau^{-\frac{3}{5}}.
 \end{align}

For the second step of induction, suppose that there exists a time $\tau_1$ such that for $\tau\geq \tau_1$ 
\begin{align}
\sum_{k=2,3}|\Omega_k(\tau)|+\sum_{l=1,2} |\alpha_l(\tau)|\leq \tau^{-\gamma},\ \text{for some}\ \gamma\in [ \frac{3}{5}, \ 5].\label{eq:bootstrap}
\end{align}
Then we claim that there exists a $\tau_2\geq \tau_1$, such that when $\tau\geq \tau_2$, a better decay rate holds,
\begin{align}
\sum_{k=2,3}|\Omega_k(\tau)|+\sum_{l=1,2} |\alpha_l(\tau)|\leq \tau^{-\gamma-\frac{1}{5}}. \label{eq:desirePi22}
\end{align}

Assuming the claim holds, then we obtain the desired decay estimate \eqref{eq:desirePi2} after iterating finitely many times.

To complete the proof we need to prove the claim \eqref{eq:desirePi22}. \eqref{eq:desirePi3} will be proved as a byproduct in the induction, see \eqref{eq:thetaRemainder} below with $\gamma=5.$ Thus the proof of the desired Lemma \ref{LM:improve1} will be complete after we prove the claim.

To prepare for the proof we estimate the function $H_2$ defined in \eqref{eq:defH2}, since the governing equations for $\Omega_k, \ k=2,3,$ and $\alpha_l,\ l=1,2,$ depend on it.

$H_2$ depends on $\|e^{-\frac{1}{8}|y|^2}\partial_{\theta}^{l}\nabla_y^{k}\chi_{R}w(\cdot,\tau)\|_2,\ |k|+l=1,2, \ l\geq 1.$ To control them we start from their governing equations in \eqref{eq:TwL2Est}. For the terms on its right hand, when $\tau\in [\tau_0,\tau_1]$  \eqref{eq:paramePrelim} implies that
\begin{align}
\Big(\sum_{k=2}^{3}|\Omega_k(\sigma)|+\sum_{l=1,2}|\alpha_l(\sigma)|\Big)^2 \Big(|B|^2+\sum_{k=2}^{3}|\Omega_k(\sigma)|+\sum_{l=1,2}|\alpha_l(\sigma)|\Big)^2 \lesssim \tau^{-\frac{12}{5}}
\end{align}and when $\tau\geq \tau_1$, \eqref{eq:paramePrelim} and \eqref{eq:bootstrap} imply that
\begin{align}
\Big(\sum_{k=2}^{3}|\Omega_k(\sigma)|+\sum_{l=1,2}|\alpha_l(\sigma)|\Big)^2 \Big(|B|^2+\sum_{k=2}^{3}|\Omega_k(\sigma)|+\sum_{l=1,2}|\alpha_l(\sigma)|\Big)^2 \lesssim \tau^{-\frac{6}{5}-2\gamma}.
\end{align}

Rewrite \eqref{eq:TwL2Est} by applying the Duhamel's principle
\begin{align}
\sum_{|k|+l=2,\ l\geq 1}\Big\|\partial_{\theta}^{l}\nabla_y^{k}\chi_{R}w(\cdot,\tau)\Big\|_{\mathcal{G}}^2\lesssim A_1+A_2+A_3,\label{eq:improvThe}
\end{align}
where the terms $A_k,\ k=1,2,3,$ are defined as $$A_1(\tau):=e^{-\frac{1}{4}(\tau-\tau_0)} \tau_0^{-\frac{3}{5}}+e^{-\frac{1}{4}(\tau-\tau_1)}\int_{\tau_0}^{\tau_1} e^{-\frac{1}{4}(\tau_1-\sigma)} \Big[\sum_{k=2}^{3}|\Omega_k(\sigma)|+\sum_{l=1,2}|\alpha_l(\sigma)|\Big]^2 \sigma^{-\frac{6}{5}}\ d\sigma,$$
$$A_2(\tau):=\int_{\tau_1}^{\tau} e^{-\frac{1}{4}(\tau-\sigma)} \Big[\sum_{k=2}^{3}|\Omega_k(\sigma)|+\sum_{l=1,2}|\alpha_l(\sigma)|\Big]^2 \sigma^{-\frac{6}{5}}\ d\sigma,$$
and $$A_3(\tau):=\int_{\tau_0}^{\tau} e^{-\frac{1}{4}(\tau-\sigma)}  e^{-\frac{1}{5}R^2(\sigma)}\ d\sigma.$$Here $\tau_0^{-\frac{3}{5}}$ in the definition of $A_1$ is from $\displaystyle\sum_{|k|+l=2,\ l\geq 1}\|e^{-\frac{1}{8}|y|^2}\partial_{\theta}^{l}\nabla_y^{k}\chi_{R}w(\cdot,\tau_0)\|_2\leq \tau_0^{-\frac{3}{10}}$ in \eqref{eq:weightPrelim2}. 

Next we estimate $A_k,$ $k=1,2,3.$

Observe that $A_1$ decays exponentially fast, thus for some $\tau_2$,
\begin{align*}
A_1(\tau)\leq \tau^{-12},\ \text{when}\ \tau\geq \tau_2;
\end{align*} 
To control $A_2$, we apply the estimate in \eqref{eq:bootstrap} and L'H\^opital's rule to obtain that,
\begin{align*}
A_2(\tau)\lesssim \int_{\tau_1}^{\tau} e^{-\frac{1}{4}(\tau-\sigma)}  \sigma^{-\frac{6}{5}-2\gamma} \ d\sigma
\lesssim\tau^{-\frac{6}{5}-2\gamma}, \ \text{for any}\ \tau\geq \tau_1.
\end{align*}
For $A_3$, the rapid decay of $e^{-\frac{1}{5}R^2(\tau)}$ and L'H\^opital's rule imply that, for some large $\tau_2,$
\begin{align*}
A_3(\tau)\leq \tau^{-12},\ \text{when}\ \tau\geq \tau_2.
\end{align*}

Feed these back to \eqref{eq:improvThe} and find that, if $\tau$ is sufficiently large, then
\begin{align*}
\sum_{|k|+l=2,\ l\geq 1}\Big\|\partial_{\theta}^{l}\nabla_y^{k}\chi_{R}w(\cdot,\tau)\Big\|_{\mathcal{G}}^2\lesssim \tau^{-\frac{6}{5}-2\gamma},
\end{align*}recall that we assume that $\gamma\leq 5$.
This, together with \eqref{eq:poincare},  implies that, 
\begin{align}
\sum_{|k|+l\leq 2,\ l\geq 1}\Big\|\partial_{\theta}^{l}\nabla_y^{k}\chi_{R}w(\cdot,\tau)\Big\|_{\mathcal{G}}\leq 2\sum_{|k|+l=2,\ l\geq 1}\Big\|\partial_{\theta}^{l}\nabla_y^{k}\chi_{R}w(\cdot,\tau)\Big\|_{\mathcal{G}}\ \lesssim \tau^{-\frac{3}{5}-\gamma }.\label{eq:thetaRemainder}
\end{align}

This completes the treatment for $\sum_{|k|+l\leq 2,\ l\geq 1}\|\partial_{\theta}^{l}\nabla_y^{k}\chi_{R}w(\cdot,\tau)\|_{\mathcal{G}}.$

Now we turn to $\sum_{|k|+l=0,1,2,}\|\partial_{\theta}^{l}\nabla_y^{k}\chi_{R}w(\cdot,\tau)\|_{\mathcal{G}}$, also part of $H_2$. This is easier since we do not seek a rapid decay rate. Thus, for some large $\tau_2$,
\begin{align}
\sum_{|k|+l=0,1,2,}\|\partial_{\theta}^{l}\nabla_y^{k}\chi_{R}w(\cdot,\tau)\|_{\mathcal{G}}\lesssim \tau^{-\frac{6}{5}}, \ \text{when}\ \tau\geq \tau_2.\label{eq:slower}
\end{align}

Now we are ready to provide an estimate for $H_2.$ \eqref{eq:thetaRemainder}, \eqref{eq:slower}, \eqref{eq:paramePrelim} and \eqref{eq:bootstrap} imply that, if $\tau_2$ is sufficiently large, then for any $\tau\geq \tau_2,$
\begin{align}
H_2(\tau)\lesssim  R^{4}(\tau)\tau^{-\gamma-\frac{9}{5}}\leq  \tau^{-\gamma-\frac{13}{10}},\label{eq:onestep}
\end{align}where, recall that $R(\tau)=\mathcal{O}( \sqrt{\ln \tau})$.

After estimating $H_2$ we are ready to control $\Omega_k,\ k=2,3,$ and $\alpha_l, \ l=1,2$. 

We integrate the $\Omega_k-$equations in \eqref{eq:beta2Eqn} from $\tau$ to $\infty$, and use that $\Omega_{k}(\infty)=0$ implied by \eqref{eq:bootstrap}, and find
\begin{align}\label{eq:TauInfty}
|\Omega_{k}(\tau)|\lesssim \int_{\tau}^{\infty}[H_2(\sigma)+ e^{-\frac{1}{5}R^2(\sigma)}]\ d\sigma\leq \frac{1}{4} \tau^{-\frac{1}{5}-\gamma}, \ k=2,3,\ \text{for}\ \tau\geq \tau_2\geq \tau_1.
\end{align}
Similarly, by the $\alpha_k-$equation, $k=1,2,$ in \eqref{eq:alpha1Eqn} and that $\alpha_{k}(\infty)=0,$
\begin{align}
|\alpha_k(\tau)|\lesssim \int_{\tau}^{\infty}e^{\frac{1}{2}(\tau-\sigma)} [H_2(\sigma)+ e^{-\frac{1}{5}R^2(\sigma)}]\ d\sigma\leq \frac{1}{4}\tau^{-\frac{1}{5}-\gamma}.
\end{align}

These two estimates above imply the desired claim \eqref{eq:desirePi22}.

\end{proof}

\subsection{Decay rates for $a$, $B$, $\Omega_1$ and $\sum_{|k|\leq 2}\|e^{-\frac{1}{8}|y|^2} \nabla_y^{k} \chi_{R} w(\cdot,\tau)\|_{2}$}\label{subsec:nontheta}
In the last subsection we studied the rapidly decaying $\theta-$dependent parts of $v$. Here we will study the other slowly decaying parts. 

Technically the present situation is very close to spherically symmetric MCF, and nonlinear heat equations. This makes it possible to use the techniques in \cite{GK1,GK2, HV2, FK1992, MZ1997}, and in particular \cite{MR1230711}, where the blowup problem of nonlinear heat equation was studied. The present problem is actually easier because Theorem \ref{THM:CoMini} provides some preliminary estimates.

The main result is the following:
\begin{lemma}\label{LM:assigned}
There are only two possibilities, listed in [A] and [B] below.
\begin{itemize}
\item[(A)] There exists some constant $C>0$ such that
\begin{align}
\begin{split}
|a(\tau)-\frac{1}{2}|\leq & C\tau^{-1},\\
|\Omega_1(\tau)|+ \sum_{|k|\leq 2}\Big\|\nabla_y^{k} \chi_{R} w(\cdot,\tau)\Big\|_{\mathcal{G}}\leq & C\tau^{-2}.
\end{split}
\end{align}
 And up to a unitary transformation, as $\tau\rightarrow \infty,$
 \begin{align}
 B(\tau)=\tau^{-1} \left[
 \begin{array}{lll}
 b_1&0&0\\
 0&b_2&0\\
 0&0&b_3
 \end{array}
 \right]+\mathcal{O}(\tau^{-2}),
 \end{align} where $b_{k}, \ k=1,2,3$, only take two possible values, $0$ or $1.$
\item[(B)] If $\sum_{k=1}^3|b_k|=0$, then for some $C>0,$
\begin{align}
|B(\tau)|+ |a(\tau)-\frac{1}{2}|+ |\Omega_1(\tau)|+ \sum_{|k|\leq 2}\Big\| \nabla_y^{k} \chi_{R} w(\cdot,\tau)\Big\|_{\mathcal{G}}\leq C\tau^{-3}.
\end{align}
\end{itemize}
\end{lemma}
\begin{proof}
We will prove the desired results by induction. 

To initiate the induction, Theorem \ref{THM:CoMini}
implies that, if $\tau$ is large enough, then
\begin{align}
|B|+ |a-\frac{1}{2}|+ |\Omega_1|+ \sum_{|k|=0,1}\|\nabla_y^{k} \chi_{R} w(\cdot,\tau)\|_{\mathcal{G}}\leq \tau^{-\frac{3}{5}}.\label{eq:iniduct}
\end{align}

For the second step of induction, we suppose that, for some $\sigma\geq \frac{1}{10}$ and some large $\tau_1$, the following estimates hold for $\tau\in [\tau_1,\infty),$
\begin{align}
|B|+ |a-\frac{1}{2}|+ |\Omega_1|
+\sum_{|k|=0,1}\| \nabla_y^{k} \chi_{R} w(\cdot,\tau)\|_{\mathcal{G}}\leq \tau^{-\frac{1}{2}-\sigma}.\label{eq:rateSig}
\end{align}

In the next step of induction, we assume \eqref{eq:rateSig} holds.
And for different values of $\sigma,$ we claim the following different results:
\begin{itemize}
\item[(1)] If $\sigma$ satisfies
\begin{align}
\frac{3}{2}+3\sigma\leq \frac{11}{5},\label{eq:NoFini}
\end{align} then the decay rate in \eqref{eq:rateSig} can be improved significantly. Specifically there exists some $\tau_2\geq \tau_1$ such that when $\tau\geq \tau_2,$
\begin{align}
|B|+ |a-\frac{1}{2}|+ |\Omega_1|+ \sum_{|k|\leq 2}\| \nabla_y^{k} \chi_{R} w(\cdot,\tau)\|_{\mathcal{G}}\leq \tau^{-\frac{1}{2}-\tilde\sigma},\label{eq:Bdecay}
\end{align} where $\tilde\sigma$ is a constant defined as
$$
\tilde\sigma:=\frac{3}{20}+\frac{3}{2}\sigma.
$$
\item[(2)] If $\sigma$ is large enough to satisfy the estimate
\begin{align}
\frac{3}{2}+3\sigma= \frac{11}{5}+\epsilon_0,\ \text{for some}\ \epsilon_0>0,\label{eq:secdStep}
\end{align} then there are only two possibilities:
\begin{itemize}
\item[(A)] when $\tau$ is large enough, for some $C>0,$
\begin{align}
\begin{split}\label{eq:Generi}
|B|+|a-\frac{1}{2}|\leq & C\tau^{-1},\\
|\Omega_1|+ \sum_{|k|\leq 2}\Big\|\nabla_y^{k} \chi_{R} w(\cdot,\tau)\Big\|_{\mathcal{G}}\leq & C\tau^{-2}.
\end{split}
\end{align}
 And up to a unitary transformation, as $\tau\rightarrow \infty,$
 \begin{align}\label{eq:b1234}
 B=\tau^{-1} \left[
 \begin{array}{lll}
 b_1&0&0\\
 0&b_2&0\\
 0&0&b_3
 \end{array}
 \right]+\mathcal{O}(\tau^{-2}),
 \end{align} where $b_k,\ k=1,2,3,$ take only two possible values, $$b_{k}=0 \ \text{or}\ 1.$$
\item[(B)] 
If $b_1=b_2=b_3=0$, then the decay rate can be improved, for some $C>0,$
\begin{align}
|B|+ |a-\frac{1}{2}|+ |\Omega_1|+ \sum_{|k|\leq 2}\| \nabla_y^{k} \chi_{R} w(\cdot,\tau)\|_{\mathcal{G}}\leq C\tau^{-3}.\label{eq:B3Decay}
\end{align}
\end{itemize}
\end{itemize}

Assuming the claims hold, we are ready to prove Lemma \ref{LM:assigned}. 

\eqref{eq:NoFini} and \eqref{eq:Bdecay} imply that, after iterating the arguments  \eqref{eq:rateSig}-\eqref{eq:Bdecay} finitely many times, \eqref{eq:rateSig} holds for some $\sigma$ satisfying $\frac{3}{2}+3\sigma> \frac{11}{5}$ when $\tau$ is large enough.  Thus the condition \eqref{eq:secdStep} is satisfied, and furthermore \eqref{eq:Generi}-\eqref{eq:B3Decay} hold. 
These are exactly the desired estimates in Lemma \ref{LM:assigned}.

To complete the proof we have to prove the claims. Since the arguments are similar to those in \cite{GK1,GK2, HV2, FK1992, MZ1997, MR1230711}, we only sketch a proof here.

We start with proving that \eqref{eq:rateSig} and \eqref{eq:NoFini} imply \eqref{eq:Bdecay}. 

The governing equations for $a,\ B,\ \Omega_1$ and $\|e^{-\frac{1}{8}|y|^2} \nabla_y^{k} \chi_{R} w(\cdot,\tau)\|_{2}^2$ in Theorem \ref{THM:CoMini}, \eqref{eq:rateSig}, Lemma \ref{LM:improve1}, and the techniques in the proof of Lemma \ref{LM:improve1}, imply that that there exists some $Q\gg 1,$ such that if $\tau\geq Q$ then
\begin{align}
\Big|a-\frac{1}{2}-\frac{1}{2}(b_{11}+b_{22}+b_{33})\Big|+\sum_{|k|\leq 2}\Big\| \nabla_y^{k} \chi_{R} w\Big\|_{\mathcal{G}}\lesssim & \tau^{-1-2\sigma} + \tau^{-5},\label{eq:improvYD}\\
\Big|\frac{d}{d\tau}B+B^{T}B\Big|+ \Big|\Omega_1\Big|\lesssim &\tau^{-\frac{3}{2}-3\sigma}+\tau^{-5}.\label{eq:ImproveBe1}
\end{align}
This is the desired \eqref{eq:Bdecay}, except the estimate for $B$.

To prove the desired estimate for $B$, we let $b(\tau)$ to be the (absolute) largest eigenvalue of $B(\tau)$. Then \eqref{eq:ImproveBe1} implies that, for some constant $c$,
\begin{align}
\Big|\partial_{\tau}b(\tau)+b^2(\tau)\Big|\leq c \tau^{-\frac{3}{2}-3\sigma}.\label{eq:scalar}
\end{align}
To derive a decay rate for $b$, we compare the sizes of $b^2$ and $c \tau^{-\frac{3}{2}-3\sigma}.$ The idea is that, if $c \tau^{-\frac{3}{2}-3\sigma}$ dominates then obviously we have an upper bound for $b;$ if $b^2$ dominates, then we can derive a decay rate since the equation $\partial_{\tau}b+b^2=0$ can be solved exactly. 

Mathematically we discuss two excluding possibilities: 
\begin{itemize}
\item[(1)] for any large $M$, there exists $\tau_1\geq M$ such that $|b(\tau_1)|\leq 2\tau_1^{-\frac{13}{20}-\frac{3}{2}\sigma}$; \item[(2)] there exists some $\tau_1$ such that for any $\tau\geq \tau_1$, $$|b(\tau)|> 2\tau^{-\frac{13}{20}-\frac{3}{2}\sigma}.$$ 
\end{itemize}

For the first possibility we claim that there exists a large time $N$ such that
\begin{align}
|b(\tau)|\leq 2\tau^{-\frac{13}{20}-\frac{3}{2}\sigma}, \ \text{for any}\ \tau\geq N.\label{eq:desireSig}
\end{align}
This is the desire estimate for $B$ in \eqref{eq:Bdecay}. 

We prove the claim by contradiction. Suppose \eqref{eq:desireSig} does not hold for any $N$, then for any $M>0$ there exists a time $\tau_2\geq M$ and a maximal interval $(\tau_2, \tau_3)$ such that 
\begin{align}
\begin{split}\label{eq:assumNot}
|b(\tau_2)|=&2\tau_2^{-\frac{13}{20}-\frac{3}{2}\sigma},\\
|b(\tau)|>&2\tau^{-\frac{13}{20}-\frac{3}{2}\sigma}\ \text{when}\ \tau\in (\tau_2,\tau_3).
\end{split}
\end{align} 
Thus in this interval, we have $\frac{\tau^{-\frac{3}{2}-3\sigma}}{b^2}=\mathcal{O}(\tau^{-\frac{1}{5}})\ll 1,$ and hence \eqref{eq:scalar} is of the form
\begin{align}
\partial_{\tau}b+b^2\Big(1+\mathcal{O}(\tau^{-\frac{1}{5}})\Big)=0.\label{eq:scalar2}
\end{align}
This equation can be solved almost exactly. An important observation is that we must have $b(\tau_2)>0$, otherwise this equation implies that $|b|=-b$ will grow and blowup in finite time. This contradicts to the fact that $|b(\tau)|\lesssim \tau^{-\frac{3}{5}}$ by the estimate for $B$ in \eqref{eq:paramePrelim}. 
Thus we must have $b(\tau_2)=2\tau_2^{-\frac{13}{20}-\frac{3}{2}\sigma}.$ Use this to solve \eqref{eq:scalar2} to find that, for $\tau
\in [\tau_2,\ \tau_3]$,
\begin{align}
b(\tau)=\frac{1}{\frac{1}{b(\tau_2)}+(\tau-\tau_2)(1+o(1))}\leq 2\tau^{-\frac{13}{20}-\frac{3}{2}\sigma},\label{eq:solveODE}
\end{align} 
where in the second step we used that $\frac{13}{20}+\frac{3}{2}\sigma\leq 1.$ This is a contradiction to the inequality in \eqref{eq:assumNot}. Thus the claim \eqref{eq:desireSig} must hold.

For the second possibility,  we will rule out the possibility of its existence. 

We proceed by contradiction. Suppose that there exists a $\tau_1$ such that for any $\tau\geq \tau_1$ we have $|b(\tau)|> 2\tau^{-\frac{13}{20}-\frac{3}{2}\sigma}.$ Then \eqref{eq:scalar2} holds in the interval $[\tau_1,\infty)$. Argue as what is before \eqref{eq:solveODE} to find that $b(\tau_1)>0$ and for $\tau\geq \tau_1,$
\begin{align}
b(\tau)=\frac{1}{\frac{1}{b(\tau_1)}+(\tau-\tau_1)(1+o(1))}.
\end{align}
Thus $b$ decays like $\tau^{-1}$ for large time $\tau\gg \tau_1$, this contradicts to $|b(\tau)|>2\tau^{-\frac{13}{20}-\frac{3}{2}\sigma}$.

We completed the proof of Item [(1)].

For Part A of Item [(2)], we already proved \eqref{eq:Generi} as a byproduct of proving Item [(1)]: set $\sigma=\frac{1}{2}$ in \eqref{eq:rateSig}, \eqref{eq:improvYD} and \eqref{eq:ImproveBe1}. 

What is left is to prove that the symmetric matrix $B$ is almost diagonal, recall that it satisfies
\begin{align}
\begin{split}
    \Big|\partial_{\tau}B+B^{T}B\Big|\lesssim & \tau^{-3}\\ 
|B(\tau)|\lesssim & \tau^{-1}.
\end{split}
\end{align}The key idea is to analyze the eigenvalues of $B$, similar to study the largest eigenvalue $b(\tau)$ above. Since the detailed arguments have been used in studying the blowup problem of nonlinear heat equation, see \cite{MR1230711}. We choose to skip the details here. 

This completes proving the A-part of Item [(2)].

Now we prove the B-part in Item [(2)]. Since $b_1=b_2=b_3=0$, \eqref{eq:b1234} implies that $|B|=\mathcal{O}(\tau^{-2}).$ Now we show that it decays even faster than $\tau^{-2}$. For that
we rewrite \eqref{eq:Beqn} by integrating both sides from $\tau$ to $\infty,$
\begin{align}
\begin{split}\label{eq:newFormB}
|B|(\tau)\leq & \int_{\tau}^{\infty}|B|^2(s)\ ds+c\int_{\tau}^{\infty}H_1(s)+ \tilde\delta e^{-\frac{1}{5}R^2(s)}\ ds\\
\leq & c\Big(1+o(|B(\tau)|)\Big)\int_{\tau}^{\infty}H_1(s)+ \tilde\delta e^{-\frac{1}{5}R^2(s)}\ ds.
\end{split}
\end{align} What is left is similar to the proof of Lemma \ref{LM:improve1}. By \eqref{eq:newFormB}, \eqref{eq:Aeqn}, \eqref{eq:beta1}, \eqref{eq:wL2Est}, \eqref{eq:YwL2Est} and \eqref{eq:TwYL2Est} we obtain the desired estimates. We skip the details here.
 
\end{proof}


\section{Proof of Lemma \ref{LM:Lya}}\label{sec:DecomVW}
\begin{proof}
The tool is the standard fixed point theorem. To prepare for its application we formulate the problem into a convenient form.

To make $\chi_{R}w$ satisfy the orthogonality conditions, it suffices to choose scalar-functions $a,$ $\alpha_l,\ l=1,2,$ and $3-$dimensional-vector-valued functions $\Omega_k,\ k=1,2,3,$ and a $3\times 3$-symmetric-real-matrix-valued function $B$ to make
\begin{align}\label{eq:21func}
 \Big\langle \Gamma\Big(a(\tau), \alpha(\tau), \Omega(\tau), B(\tau)\Big),\ \Lambda_{l}\Big\rangle_{\mathcal{G}}=0,
\end{align} where, recall the definition of $\langle \cdot, \ \cdot\rangle_{\mathcal{G}}$ in \eqref{def:Gin}, and $\Lambda_l$, $l=1,\cdots,18,$ are the 18 functions listed in \eqref{eq:orthow}, and $\Gamma$ is a function defined as
\begin{align}
\begin{split}
\Gamma(a,\ \alpha,\ \Omega,\ B)
:=&\chi_{R}\Big[v-V_{a,B}-{\Omega}_1\cdot y -{\Omega}_2\cdot y cos\theta- {\Omega}_3\cdot y sin\theta  -\alpha_1 cos\theta -\alpha_2 sin\theta \Big].
\end{split}
\end{align}

To make the fixed point theorem applicable we make the following two observations.

The first one is that, the $18\times 18$ matrix $A:=[a_{j,l}]$, with $a_{j,l}$ defined as, $$a_{j,l}:=\nabla_{\sigma_j}\Big\langle \Gamma(a,\ \alpha,\ \Omega,\ B),\ \Lambda_l\Big\rangle_{\mathcal{G}},\ \text{with}\ l,j=1,2,\cdots, 18,$$ is uniformly invertible, provided that $|a-\frac{1}{2}|$, $|\alpha_l|$ and $|\Omega_k|$ and $|B|$ are sufficiently small. Here $\sigma_j, \ j=1, \ \cdots, 18,$ are the 18 parameters defining the scalars $a$, $\alpha_l,$ $l=1,2,$ the $3-$dimensional real vectors $\Omega_k,\ k=1,2,3,$ and $3\times 3$ symmetric real matrix $B$. This can be easily verified. We skip the details here.

The second observation is that
\begin{align}
\Big|\Big\langle \Gamma(\frac{1}{2},0,0,0),\ \Lambda_{l}\Big\rangle_{\mathcal{G}}\Big| \leq 2\Big\langle \chi_{R}\big|v(\cdot,\tau)-\sqrt{2}\big|, \ (1+|y|^2)\Big\rangle_{\mathcal{G}}\leq \tau^{-\frac{13}{20}}R_0^4(\tau).\label{eq:twoRegions}
\end{align}
To see this, we insert $1=1_{|y|\leq R_{0}(\tau)}+\big(1-1_{|y|\leq R_{0}(\tau)}\big)$ into the inner product to find
\begin{align}
\Big\langle \chi_{R}\big|v(\cdot,\tau)-\sqrt{2}\big|, \ (1+|y|^2)\Big\rangle_{\mathcal{G}}
=&\Big\langle \chi_{R}1_{|y|\leq R_{0}(\tau)}\big|v(\cdot,\tau)-\sqrt{2}\big|, \ (1+|y|^2)\Big\rangle_{\mathcal{G}}\nonumber\\
&+\Big\langle \chi_{R}\big(1-1_{|y|\leq R_{0}(\tau)}\big)\big|v(\cdot,\tau)-\sqrt{2}\big|, \ (1+|y|^2)\Big\rangle_{\mathcal{G}}\nonumber\\
=&Q_1+Q_2
\end{align}where $Q_1$ and $Q_2$ are naturally defined, and $1_{|y|\leq R_{0}(\tau)}$ is the standard Heaviside function, see \eqref{eq:heavi}. 
For $Q_1$, the estimate for $|v(\cdot,\tau)-\sqrt{2}|$ in \eqref{eq:cm1} implies that
\begin{align}
Q_1\leq \tau^{-\frac{13}{20}}.
\end{align} For $Q_2$, the integration takes place in the region $|y|\geq R_0(\tau)$, the rapid decay of $e^{-\frac{1}{4}|y|^2}$ and that $|v(\cdot,\tau)-\sqrt{2}|\leq 1$ imply
\begin{align}
Q_2\leq \tau^{-\frac{13}{20}} R_0^4(\tau),
\end{align} where, $R_0=\mathcal{O}(\sqrt{ln\tau})$ is defined in \eqref{eq:defR0T}. 

These two estimates imply the desired \eqref{eq:twoRegions}.

These two observations make the Fixed Point Theorem applicable after we limit the consideration to the following subset of a standard Banach space
\begin{align}
|a(\tau)-\frac{1}{2}|,\ |B(\tau)|,\ \sum_{k=1}^{3}|\Omega_{k}(\tau)|,\ \sum_{l=1,2}|\alpha_l(\tau)|\leq \tau^{-\frac{31}{50}}.\label{eq:sample}
\end{align} Now we apply the fixed point theorem to prove the desired \eqref{eq:paramePrelim}.

What is left is to prove \eqref{eq:weightPrelim1} and \eqref{eq:weightPrelim2}. 

To prepare for proving the first estimate in \eqref{eq:weightPrelim1} we divide $\chi_{R}w$ into two terms
\begin{align}
\|\chi_{R}w\|_{\mathcal{G}}\leq \|\chi_{R}1_{|y|\leq R_0}w\|_{\mathcal{G}}+\|\chi_{R}1_{|y|> R_0}w\|_{\mathcal{G}}.
\end{align} For the first term, the estimate for $v-\sqrt{2}$ in \eqref{eq:cm1}, the decomposition of $v$ and \eqref{eq:paramePrelim} imply that, recall that $\tau\gg 1,$
\begin{align}
\Big\|\chi_{R}1_{|y|\leq R_0}w(\cdot,\tau)\Big\|_{\mathcal{G}}\leq \frac{1}{2}\tau^{-\frac{3}{5}};
\end{align} for the second term, \eqref{eq:cm1} and the definition of $R_0(\tau)$ in \eqref{eq:defR0T} imply
\begin{align}
\|\chi_{R}1_{|y|> R_0}w\|_{\mathcal{G}}\lesssim \|1_{|y|> R_0}\|_{\mathcal{G}}\ll \frac{1}{2}\tau^{-\frac{3}{5}}.
\end{align}
These two estimates obviously imply the desired estimate for the first one in \eqref{eq:weightPrelim1},
\begin{align}
\|\chi_{R}w\|_{\mathcal{G}}\leq \tau^{-\frac{3}{5}}.
\end{align}

The proof of the other two estimates in \eqref{eq:weightPrelim1} is similar, thus is skipped.

Among those in \eqref{eq:weightPrelim2}, we start with treating $\|\partial_{y_1}^2\chi_{R}w\|_{\mathcal{G}}.$
Integrate by parts in $y_1$ to obtain
\begin{align*}
\|\partial_{y_1}^2\chi_{R}w\|_{\mathcal{G}}^2= &-\Big\langle \partial_{y_1} e^{-\frac{1}{4}|y|^2}\partial_{y_1}^2\chi_{R}w,\ \partial_{y_1}\chi_{R}w\Big\rangle\\
=&\frac{1}{2}\Big\langle y_1 \partial_{y_1}^2\chi_{R}w,\ \partial_{y_1}\chi_{R}w\Big\rangle_{\mathcal{G}}-\Big\langle  \partial_{y_1}^3\chi_{R}w,\ \partial_{y_1}\chi_{R}w\Big\rangle_{\mathcal{G}}.
\end{align*} This, together with that $\sum_{|k|\leq 3}|\nabla_{y_1}^{k}w|\leq 1$ when $|y|\leq (1+\epsilon)R$, and $\sum_{|k|=1,2}\|(\nabla_{y}^{k}\chi_{R}) (1+|y_1|)\|_{\mathcal{G}}\ll 1$, and \eqref{eq:weightPrelim1}, implies the desired estimate
\begin{align}
\|\partial_{y_1}^2\chi_{R}w\|_{\mathcal{G}}\lesssim \|\partial_{y_1}\chi_{R}w\|_{\mathcal{G}}^{\frac{1}{2}}\lesssim \tau^{-\frac{3}{10}}.
\end{align}

The other estimates in \eqref{eq:weightPrelim2} can be obtained similarly. We skip the details here.
\end{proof}


\section{Derivation of \eqref{eq:beta2Eqn}}\label{sec:beta2Eqn}
Here we only estimate the governing equation for $\Omega_2$, that for $\Omega_3$ can be derived identically. 

Recall that the definition of $R(\tau)$ in \eqref{eq:defRTau} depends on some large constant $\tau_0$. To facilitate later discussions we require $\tau_0$ to be large enough so that
\begin{align}
\kappa(\epsilon) R^{-\frac{1}{10}}(\tau_0)+\tau_0^{-\beta} R^{4}(\tau_0)\leq 1.\label{eq:EpsiTau}
\end{align}Recall the definitions of the inner product $\langle\cdot,\ \cdot\rangle_{\mathcal{G}},$ the norm $\|\cdot\|_{\mathcal{G}}$ from \eqref{def:Gin}.

We take a $\mathcal{G}$-inner product of the equation of $\partial_{\tau}( \chi_{R} w)$ in \eqref{eq:tildew3} and $ y cos\theta$ to find that\begin{align}
\begin{split}\label{eq:innercos}
\Big\langle \partial_{\tau}( \chi_{R} w),\  y cos\theta\Big\rangle_{\mathcal{G}}
=&\Big\langle -L(\chi_{R}w)+ (G+N_1+N_2)\chi_{R}+ \mu(w),\ y cos\theta\Big\rangle_{\mathcal{G}},
\end{split}
\end{align}where we use the identity $\langle F \chi_{R},\  y cos\theta\rangle_{\mathcal{G}}=0$, recall that $F$ is independent of $\theta.$

The orthogonality condition $\chi_R w\perp_{\mathcal{G}}   y cos\theta $ in \eqref{eq:orthow} will make many terms disappear. Indeed, for the term on the left hand side,
\begin{align}
\Big\langle \partial_{\tau}( \chi_{R} w),\  y cos\theta\Big\rangle_{\mathcal{G}}=\partial_{\tau}\Big\langle  \chi_{R} w,\  y cos\theta\Big\rangle_{\mathcal{G}}=0;
\end{align}
and for the first term on the right hand side, 
\begin{align}
\Big\langle -L(\chi_{R}w),\  y cos\theta\Big\rangle_{\mathcal{G}}=0.
\end{align}
based on the three facts: (1) use $\partial_{\theta}^2 cos\theta=-cos\theta$ to integrate by parts in $\theta$
\begin{align*}
\Big\langle \chi_{R} (V_{a,B}^{-2}\partial_{\theta}^2w+V_{a,B}^{-2}w), \  ycos\theta\Big\rangle_{\mathcal{G}}=0,
\end{align*} (2)  $e^{-\frac{1}{4}|y|^2} y $ is an eigenvector of $(-\Delta+\frac{1}{2}y\cdot \nabla_{y})^*$ and (3) $\chi_{R}w\perp_{\mathcal{G}}  y cos\theta$.

For the $G$-term, by parity, all the terms, except $\Omega_2-$term, make no contribution, and hence
\begin{align*}
\Big\langle G \chi_{R},\ y cos\theta\Big\rangle_{\mathcal{G}}=&- \frac{\pi}{3}\int_{\mathbb{R}^3}\chi_{R}(y) |y|^2 e^{-\frac{1}{4}|y|^2}\ dy\  \frac{d}{d\tau}\Omega_2
=- \frac{\pi}{3}\int_{\mathbb{R}^3} |y|^2 e^{-\frac{1}{4}|y|^2}\ dy \Big(1+o(1)\Big)\  \frac{d}{d\tau}\Omega_2.
\end{align*}

For the last term,
\begin{align}
\Big|\Big\langle \mu(w),\  y cos\theta\Big\rangle_{\mathcal{G}}\Big|\lesssim e^{-\frac{1}{5}R^2(\tau)} \tilde\delta,\label{eq:tildechi3}
\end{align} where the factor $\tilde\delta$ is from the bound
\begin{align}
\sum_{|k|+l=0}^4|\partial_{y}^{k}\partial_{\theta}^{l}w| \leq \delta+\tau^{-\frac{1}{2}}\leq \tilde\delta,\ \text{when}\ |y|\leq (1+\epsilon)R(\tau)=\mathcal{O}(\sqrt{\ln \tau}),\label{eq:estW}
\end{align}implied by \eqref{eq:decomVToW}, \eqref{eq:paramePrelim} and \eqref{eq:condition1}; and the factor $e^{-\frac{1}{5}R^2}$ is from 
\begin{align}
\int_{|y|\geq R} e^{-\frac{1}{4}|y|^2} |y|\ dy^3
&\leq  e^{-\frac{1}{5}R^2},
\end{align}
since $\mu(w)$ is supported by the set $\Big\{y\ \Big| \ |y|\in [R, \ (1+\epsilon )R]\Big\}$, and $e^{-\frac{1}{4}|y|^2}$ decays rapidly.

Collect the estimates to find a new form for \eqref{eq:innercos}
\begin{align}
|\frac{d}{d\tau}\Omega_2|\lesssim \Big|\Big\langle (N_1+N_2)\chi_{R}, \  y cos\theta\Big\rangle_{\mathcal{G}}\Big|+e^{-\frac{1}{5}R^2(\tau)} \tilde\delta.\label{eq:simN1N2}
\end{align}
We claim that, recall that $H_2$ is defined in \eqref{eq:defH2},
\begin{align}
&\Big|\Big\langle (N_1+N_2)\chi_{R}, \ y cos\theta\Big\rangle_{\mathcal{G}}\Big|\lesssim H_{2}+e^{-\frac{1}{5}R^2(\tau)} \tilde\delta^2. \label{eq:estN1N21}
\end{align}

Suppose the claim holds, \eqref{eq:simN1N2} and \eqref{eq:estN1N21} imply the desired estimate for $\Omega_2$ in \eqref{eq:beta2Eqn}.

What is left is to prove the claim \eqref{eq:estN1N21}. Here we need some cancellations.

We choose to study only four terms on the left hand side of \eqref{eq:estN1N21}, specifically $F_{k},\ k=1,2,3,4,$ where $F_1$ is defined in terms of $N_2$,
\begin{align}
F_1:=\Big\langle \chi_{R} N_2(\eta),\  y cos\theta\Big\rangle_{\mathcal{G}}.
\end{align} $F_{k},\ k=2,3,4,$ are parts of $\langle \chi_{R} N_1,\   y cos\theta\rangle_{\mathcal{G}},$
\begin{align}
\begin{split}
F_2:=&\Big\langle \chi_{R}\frac{(\partial_{y_1}v)^2}{1+|\nabla_y v|^2+v^{-2}(\partial_{\theta}v)^2}\partial_{y_1}^2 v,\  y cos\theta\Big\rangle_{\mathcal{G}},\\
F_3:=&\Big\langle \chi_{R} v^{-2} \frac{(v^{-1}\partial_{\theta}v)^{2} \partial^{2}_{\theta}v}{1+|\nabla_y v|^2+(\frac{\partial_{\theta}v}{v})^2},\  y cos\theta\Big\rangle_{\mathcal{G}},\\
F_4:=&\Big\langle \chi_{R} \frac{1}{1+|\nabla_y v|^2+(\frac{\partial_{\theta}v}{v})^2}\frac{(\partial_{\theta}v)^2}{v^3},\   y cos\theta\Big\rangle_{\mathcal{G}}.
\end{split}
\end{align}
The reason for choosing these terms, instead of the others, is following. An obstacle of proving $F_2$ decays rapidly is that all of the factors of $\frac{(\partial_{y_1}v)^2}{1+|\nabla_y v|^2+v^{-2}(\partial_{\theta}v)^2}\partial_{y_1}^2 v$ decay slowly. We have to integrate by parts in $\theta$ to make the rapidly decaying $\theta-$derivative of $v$ to contribute. To prove $F_4$ decays rapidly we need to observe some cancellations. 
Controlling $F_3$ is the easiest, we use it to illustrate the techniques of controlling most of the other terms.

We start with observing some cancellations in $F_1$. By definition
\begin{align*}
N_{2}(\eta) =-(v^{-1}-V_{a,B}^{-1}+V_{a,B}^{-2}\eta)+(v^{-2}-V_{a,B}^{-2})\partial_{\theta}^2 \eta.
\end{align*}
By taking a $\mathcal{G}-$inner product with $ y cos\theta$, the $\theta-$independent $V_{a,B}^{-1}$ does not contribute. Then we integrate by parts, twice, using the identities
$
\partial_{\theta}^2 cos\theta=-cos\theta
$ and $\partial_{\theta}v=\partial_{\theta}\eta$ to obtain
\begin{align*}
\Big\langle \chi_{R}V_{a,B}^{-2}\eta,\  y cos\theta\Big\rangle_{\mathcal{G}}=&-\Big\langle \chi_{R} V_{a,B}^{-2}\partial_{\theta}^2\eta,\  y cos\theta\Big\rangle_{\mathcal{G}},\\
\Big\langle \chi_{R}v^{-1},\  y cos\theta\Big\rangle_{\mathcal{G}}=&\Big\langle \chi_{R}v^{-2}\partial_{\theta}^2 \eta, \ y cos\theta\Big\rangle_{\mathcal{G}}-2\Big\langle v^{-3}(\partial_{\theta}\eta)^2, \  y cos\theta\Big\rangle_{\mathcal{G}}.
\end{align*}
This transforms $F_1$ into a simpler form
\begin{align}
F_1=2\Big\langle v^{-3}(\partial_{\theta}\eta)^2 \chi_{R},\  y cos\theta\Big\rangle_{\mathcal{G}}.\label{eq:similar}
\end{align}

Still we need to observe more cancellations. Decompose $\eta$ as in \eqref{eq:decomW} and insert the identity $1=\chi_{R}+(1-\chi_{R})$ before $w$ to find
\begin{align}
\partial_{\theta}\eta
=&\partial_{\theta} \Big[{\Omega}_2\cdot y cos\theta + {\Omega}_3\cdot y sin\theta +\alpha_1 cos\theta +\alpha_2 sin\theta\Big]+\partial_{\theta}\chi_{R}w+\partial_{\theta}(1-\chi_{R})w.
\end{align} 
Plug this into the definition of $F_1$ to find
\begin{align}
F_1=&2\Big\langle v^{-3}\Big[ \partial_{\theta} \big[{\Omega}_2\cdot y cos\theta + {\Omega}_3\cdot y sin\theta +\alpha_1 cos\theta +\alpha_2 sin\theta\big]\Big]^2\chi_{R},\  y cos\theta\Big\rangle_{\mathcal{G}}\nonumber\\
&+4\Big\langle v^{-3} \partial_{\theta} \big[{\Omega}_2\cdot y cos\theta + {\Omega}_3\cdot y sin\theta +\alpha_1 cos\theta +\alpha_2 sin\theta\big]\ \chi_{R}\partial_{\theta}\chi_{R}w,\  y cos\theta\Big\rangle_{\mathcal{G}}\nonumber\\
&+2\Big\langle v^{-3}(\partial_{\theta}\chi_{R}w)^2\ \chi_{R},\  y cos\theta\Big\rangle_{\mathcal{G}}+\mathcal{O}(\tilde\delta e^{-\frac{1}{5}R^2})\nonumber\\
=&F_{11}+F_{12}+F_{13}+\mathcal{O}(\tilde\delta^2 e^{-\frac{1}{5}R^2}),\label{eq:defF123}
\end{align}where the terms $F_{1,k},\ k=1,2,3,$ are naturally defined, and the rapidly decaying $\mathcal{O}(\tilde\delta e^{-\frac{1}{5}R^2})$ is used to control the terms having a factor $(1-\chi_{R})w$ and its derivatives, which is supported by the set $\Big\{ y\ \Big|\ |y|\geq R\Big\}$ ------ a good example is, 
\begin{align}
\begin{split}\label{eq:e5R}
\Big|\Big\langle v^{-3}\chi_{R}\ \partial_{\theta}\chi_{R}w\ &\partial_{\theta}(1-\chi_{R})w,\  y cos\theta\Big\rangle_{\mathcal{G}}\Big|\lesssim \tilde\delta^2 \int_{|y|\geq R} e^{-\frac{1}{4}|y|^2} |y|\ d^3y\lesssim \tilde\delta^2 e^{-\frac{1}{5}R^2},
\end{split}
\end{align}here we use the estimates in \eqref{eq:estW} to find $
|\Big(\partial_{\theta}\chi_{R}w\Big)\ \partial_{\theta}\Big((1-\chi_{R})w\Big)|\lesssim \tilde\delta^2.
$

For the terms in \eqref{eq:defF123}, we control $F_{12}$ and $F_{13}$ by direct computation
\begin{align}
\begin{split}\label{eq:f1213}
|F_{12}|\lesssim & \|\chi_{R}\partial_{\theta}w\|_{\mathcal{G}} \Big[\sum_{k=2,3}|\Omega_k|+\sum_{l=1,2}|\alpha_l|\Big],\\
|F_{13}|\lesssim & \Big\langle (\partial_{\theta}w)^2\ \chi_{R}^2,\  |y| \Big\rangle_{\mathcal{G}}
\lesssim  R\ \|\chi_{R}\partial_{\theta}w \|_{\mathcal{G}}^2.
\end{split}
\end{align}

Controlling $F_{11}$ is the most involved since we need to observe some cancellations. The facts 
$$\int_{0}^{2\pi} cos^3\theta\ d\theta=\int_{0}^{2\pi} sin^2\theta cos\theta\ d\theta=\int_{0}^{2\pi} sin\theta cos^2 \theta\ d\theta=0$$
make
\begin{align} 
\Big\langle \Big[ \partial_{\theta} \big[{\Omega}_2\cdot y cos\theta + {\Omega}_3\cdot y sin\theta +\alpha_1 cos\theta +\alpha_2 sin\theta\big]\Big]^2,\ cos\theta\Big\rangle_{\theta}=0.\label{eq:zeros}
\end{align} 
Here and in the rest of the paper $\langle \cdot,\ \cdot\rangle_{\theta}$ signifies the standard inner product in the $\theta$-variable. This forces the rapidly decaying $\theta-$dependent part of $v$ to contribute. Mathematically, we integrate by parts in $\theta$ to obtain
\begin{align*}
\Big\langle v^{-3} sin^2\theta,\ cos\theta\Big\rangle_{\theta}=&-\frac{1}{3}\Big\langle \partial_{\theta}v^{-3},\ sin^3\theta\Big\rangle_{\theta};\\
\Big\langle v^{-3} sin\theta cos\theta,\ cos\theta\Big\rangle_{\theta}=&\frac{1}{3}\Big\langle \partial_{\theta}v^{-3},\ cos^3\theta\Big\rangle_{\theta};\\
\Big\langle v^{-3} cos^2\theta,\ cos\theta\Big\rangle_{\theta}=& \Big\langle \partial_{\theta}v^{-3} ,\ sin\theta\Big\rangle_{\theta}+\frac{1}{3}\Big\langle \partial_{\theta}v^{-3},\ sin^3\theta\Big\rangle_{\theta}.
\end{align*} Plug these into the definition of $F_{11}$, and then decompose $v$ as in \eqref{eq:decomVToW} and insert $1=\chi_{R}+(1-\chi_{R})$ to find
\begin{align}
\begin{split}
|F_{11}|\lesssim &\big[\sum_{k=2,3}|\Omega_k|+\sum_{l=1,2}|\alpha_l|\big]^2\ \Big\langle |\partial_{\theta}v^{-3}|\ \chi_{R}, \  (1+|y|)^3\Big\rangle_{\mathcal{G}}\\
\lesssim & \Big[\sum_{k=2,3}|\Omega_k|+\sum_{l=1,2}|\alpha_l|\Big]^3+\Big[\sum_{k=2,3}|\Omega_k|+\sum_{l=1,2}|\alpha_l|\Big]^2 \|\chi_{R}\partial_{\theta}w \|_{\mathcal{G}}.
\end{split}
\end{align}
This, together with \eqref{eq:defF123} and \eqref{eq:f1213}, implies the desired estimate
\begin{align}
|F_1|\lesssim H_2+\tilde\delta^2 e^{-\frac{1}{5}R^2}.\label{eq:similar10}
\end{align}

For $F_4$, use the identity $\partial_{\theta}v=\partial_{\theta}\eta$ and compute directly to have 
\begin{align}
F_4=&\Big\langle \chi_{R} \frac{1}{1+|\nabla_y v|^2+(\frac{\partial_{\theta}v}{v})^2}\frac{(\partial_{\theta}\eta)^2}{v^3},\  y cos\theta\Big\rangle_{\mathcal{G}}\nonumber\\
=& \Big\langle \chi_{R} \frac{(\partial_{\theta}\eta)^2}{v^3},\   y cos\theta\Big\rangle_{\mathcal{G}}
-\Big\langle \chi_{R} \frac{|\nabla_y v|^2+(\frac{\partial_{\theta}v}{v})^2}{1+|\nabla_y v|^2+(\frac{\partial_{\theta}v}{v})^2}\frac{(\partial_{\theta}\eta)^2}{v^3},\   y cos\theta\Big\rangle_{\mathcal{G}}.
\end{align}
The first term is exactly $\frac{1}{2}F_1$, see \eqref{eq:similar}, which was estimated in \eqref{eq:similar10}. 
It is easy to control the second. These makes
\begin{align}
|F_4|\lesssim H_2+\tilde\delta^2 e^{-\frac{1}{5}R^2}.
\end{align}

For $F_2$ we start with transforming the expression. The idea is to force the rapidly decaying $\theta-$dependent parts of $v$ to contribute.
Integrate by parts in $\theta$ to have, 
\begin{align}
F_2=&-\Big\langle \chi_{R}\partial_{\theta}\Big[\frac{(\partial_{y_1}v)^2}{1+|\nabla_y v|^2+v^{-2}(\partial_{\theta}v)^2}\partial_{y_1}^2 v\Big],\ y sin\theta\Big\rangle_{\mathcal{G}}\label{eq:inteTheta}\\
=&-\Big\langle \chi_{R}\frac{(\partial_{y_1}v)^2}{1+|\nabla_y v|^2+v^{-2}(\partial_{\theta}v)^2}\partial_{y_1}^2\partial_{\theta} w,\ y sin\theta\Big\rangle_{\mathcal{G}}\nonumber\\
&-\Big\langle \chi_{R}\Big[\partial_{\theta}\frac{(\partial_{y_1}v)^2}{1+|\nabla_y v|^2+v^{-2}(\partial_{\theta}v)^2}\Big]\partial_{y_1}^2 V_{a,B},\  y sin\theta\Big\rangle_{\mathcal{G}}\nonumber\\
&-\Big\langle \chi_{R}\Big[\partial_{\theta}\frac{(\partial_{y_1}v)^2}{1+|\nabla_y v|^2+v^{-2}(\partial_{\theta}v)^2}\Big]\partial_{y_1}^2 w,\  y sin\theta\Big\rangle_{\mathcal{G}}\nonumber\\
=&F_{21}+F_{22}+F_{23}\label{eq:F2123}
\end{align}where the terms $F_{21}$, $F_{22}$ and $F_{23}$ are naturally defined, and in the second step we use the following two identities, implied by the
decomposition of $v$ in \eqref{eq:decomVToW},
\begin{align*}
\partial_{y_1}^2\partial_{\theta}v=\partial_{y_1}^2\partial_{\theta}w\ \text{and}\ \partial_{y_1}^2 v=\partial_{y_1}^2 V_{a,B}+\partial_{y_1}^2 w.
\end{align*}

For the term $F_{22},$ we obtain
\begin{align}
|F_{22}|\lesssim &|B|\Big\langle \chi_{R}  \Big[|\partial_{y_1}v||\nabla_{y}\partial_{\theta}v|+|\partial_{\theta}v||\partial_{\theta}^2v|+|\partial_{\theta}v|^2\Big],\  |y| \Big\rangle_{\mathcal{G}}\label{eq:f222}
\end{align} by two facts: (1) $|\partial_{\theta}v|\ll 1$ by \eqref{eq:condition1}; (2) since $|a-\frac{1}{2}|, \ |B|\leq \tau^{-\frac{3}{5}}$ by \eqref{eq:paramePrelim}, 
\begin{align}
\sum_{|k|=2}|\nabla_{y}^k V_{a,B}|\lesssim |B|+|B|^2 |y|^2\leq 2|B|\ \text{when} \ |y|\leq (1+\epsilon)R(\tau)=\mathcal{O}(\sqrt{\ln \tau}).\label{eq:y12Vab}
\end{align}
The slowest decaying term in \eqref{eq:f222} is $\Big\langle \chi_{R}  |\partial_{y_1}v||\partial_{y_1}\partial_{\theta}v|,\ |y| \Big\rangle_{\mathcal{G}}$ since the others depend more on the fast decaying $\theta-$derivatives of $v$. For this term, we decompose $v$ as in \eqref{eq:decomVToW} to find
\begin{align}
&|\partial_{y_1}v|\ |\partial_{y_1}\partial_{\theta}v|
\leq \Big[ |B||y|+\sum_{k=1}^{3}|\Omega_k|+|\partial_{y_1} w|\Big] \Big[\sum_{k=2,3}|\Omega_k|+|\nabla_{y}\partial_{\theta}w|\Big],
\end{align} where we use that, as in deriving \eqref{eq:y12Vab},
\begin{align}
|\nabla_{y}V_{a,B}|\lesssim |B||y|\ \text{when}\ |y|\leq (1+\epsilon)R=\mathcal{O}(\ln\tau).\label{eq:nablayV}
\end{align} Insert the identity $1=\chi_{R}+(1-\chi_{R})$ before $w$ and compute directly to obtain
\begin{align*}
\Big\langle \chi_{R}  |\partial_{y_1}v||\nabla_{y}\partial_{\theta}v|,&\ |y| \Big\rangle_{\mathcal{G}}\lesssim \sum_{k=2,3}|\Omega_k|\  \Big[ |B|+\sum_{k=1}^{3}|\Omega_k|+\|\partial_{y_1}\chi_{R}w\|_{\mathcal{G}}\Big]\\
&+\|\partial_{\theta}\nabla_{y}\chi_{R}w\|_{\mathcal{G}} \Big[ |B|+\sum_{k=1}^{3}|\Omega_k|+R\|\partial_{y_1}\chi_{R}w\|_{\mathcal{G}}\Big]+\tilde\delta e^{-\frac{1}{5}R^2}.
\end{align*} This and the definition of $H_2$ in \eqref{eq:defH2} imply
\begin{align}
|F_{22}|\lesssim & H_{2}+\tilde\delta e^{-\frac{1}{5}R^2}.\label{eq:estF22ex}
\end{align}

For $F_{21}$, to avoid estimating $\|\chi_{R} \partial_{y_1}^2\partial_{\theta} w\|_{\mathcal{G}}$, we integrate by parts in $y_1$ to find,
\begin{align}
F_{21}=&\Big\langle \partial_{y_1}\partial_{\theta}w,\ \partial_{y_1}\Big[\chi_{R} \frac{(\partial_{y_1}v)^2}{1+|\nabla_y v|^2+v^{-2}(\partial_{\theta}v)^2}e^{-\frac{1}{4}|y|^2} y cos\theta\Big]\Big\rangle.
\end{align}
Similar to obtaining \eqref{eq:estF22ex}
\begin{align}
\begin{split}
|F_{21}|\lesssim &\Big\|\partial_{y_1}\partial_{\theta}\chi_{R}w\Big\|_{\mathcal{G}} \Big[|B|^2+\sum_{k=1}^3|\Omega_k|^2+
\tilde\delta R^2\sum_{k=1,2}\Big\|\partial_{y_1}^{k}\chi_{R}w\Big\|_{\mathcal{G}}\Big]+\tilde\delta^2 e^{-\frac{1}{5}R^2}\\
\lesssim & H_{2}+\tilde\delta e^{-\frac{1}{5}R^2}.
\end{split}
\end{align}

Similarly for $F_{23}$,
\begin{align}
\begin{split}
|F_{23}|\lesssim& \Big\|\partial_{y_1}^2 \chi_{R}w\Big\|_{\mathcal{G}}\Big[|B|+\sum_{k=1}^{3}|\Omega_k|+\sum_{l=1,2}|\alpha_l|\Big]\Big[\sum_{k=2}^{3}|\Omega_k|+\sum_{l=1,2}|\alpha_l|\Big]\\
&+\tilde\delta \Big\|\partial_{y_1}^2 \chi_{R}w\Big\|_{\mathcal{G}} \Big[\|\nabla_{y}\partial_{\theta} \chi_{R}w\|_{\mathcal{G}}+\|\partial_{\theta}^2 \chi_{R}w\|_{\mathcal{G}}\Big]+\tilde\delta^2 e^{-\frac{1}{5}R^2}.
\end{split}
\end{align}

We estimated all the terms in \eqref{eq:F2123}. These estimates imply the desired estimate
\begin{align}
|F_2|\lesssim H_2+\tilde\delta e^{-\frac{1}{5}R^2}.
\end{align}

For $F_3,$ compute directly to find,
\begin{align}
|F_3|\lesssim \Big\langle \chi_{R} (\partial_{\theta}v)^2 |\partial_{\theta}^2 v|,\ |y|\Big\rangle_{\mathcal{G}}
\lesssim& \Big[\sum_{k=2,3}|\Omega_k|+\sum_{l=1,2}|\alpha_l|\Big]^3
+\tilde\delta R  \Big\| \partial_{\theta}^{2}\chi_{3,R}w\Big\|_{\mathcal{G}}^2+\tilde\delta^2 e^{-\frac{1}{5}R^2}\nonumber\\
\lesssim & H_2+\tilde\delta e^{-\frac{1}{5}R^2}.
\end{align}

\section{Proof of \eqref{eq:wL2Est} }\label{sec:wL2Est}
We start with deriving a governing equation for $\chi_{R}w$ from \eqref{eq:tildew3},
\begin{align}
\frac{1}{2}\frac{d}{d\tau}\Big\langle  w\chi_{R},\  w\chi_{R}\Big\rangle_{\mathcal{G}}
=&-\Big\langle e^{-\frac{1}{8}|y|^2} w\chi_{R}, \ \tilde{L} e^{-\frac{1}{8}|y|^2} w\chi_{R}\Big\rangle+\Big\langle  w\chi_{R},\  F\chi_{R}\Big\rangle_{\mathcal{G}}+\Big\langle  w\chi_{R},\  G\chi_{R}\Big\rangle_{\mathcal{G}}\nonumber\\
&+\Big\langle  w\chi_{R},\  N_1\chi_{R}\Big\rangle_{\mathcal{G}}+\Big\langle  w\chi_{R},\ N_2(\eta)\chi_{R}\Big\rangle_{\mathcal{G}}+H_{\chi}\nonumber\\
=&\frac{1}{2}\sum_{k=1}^{5}D_k+H_{\chi},\label{eq:L2norm}
\end{align} where the linear operator $\tilde{L}$ is defined as
\begin{align}
\tilde{L}:=e^{-\frac{1}{8} |y|^2} L e^{\frac{1}{8} |y|^2} =-\Delta_{y}+\frac{1}{16}|y|^2-\frac{3}{4}-V_{a,B}^{-2}\partial_{\theta}^2-\frac{1}{2}-V_{a,B}^{-2},\label{eq:defTL}
\end{align} $D_{k},\ k=1,2,\cdots,5,$ are naturally defined, and $H_{\chi}$ contains all the terms depending on some derivative of $\chi_{R}$,
\begin{align*}
H_{\chi}:=&\Big\langle  w\chi_{R},\ \mu(w)\Big\rangle_{\mathcal{G}}.
\end{align*}

It is easy to control $H_{\chi}$, since all the terms in $\mu(w)$ depend on some derivative of $\chi_{R}$, and hence they are supported by the set $\Big\{y\ \Big|\ |y|\in [R, \ (1+\epsilon )R]\Big\}$. Similar to obtaining \eqref{eq:e5R},
\begin{align}
|H_{\chi}|\lesssim e^{-\frac{1}{5}R^{2}}\tilde\delta^2.
\end{align}

For $D_1$ we need two estimates and then combine them together. We integrate by parts in $y-$ or $\theta-$variables and find
\begin{align}
\begin{split}\label{eq:D1Two}
D_1\leq &-2\Big\| \nabla_y (e^{-\frac{1}{8}|y|^2} w\chi_{R})\Big\|_2^2-\frac{1}{8} \Big\| y  w\chi_{R}\Big\|_{\mathcal{G}}^2- (1-\tau^{-\frac{1}{2}})\Big\|\partial_{\theta} w\chi_{R} \Big\|_{\mathcal{G}}^2\\
&+[\frac{7}{2}+\tau^{-\frac{1}{2}}]\Big\|   w\chi_{R}\Big\|_{\mathcal{G}}^2,
\end{split}
\end{align} where we apply \eqref{eq:paramePrelim} to find that, recall that $\tau\geq \tau_0\gg 1,$
\begin{align}
|V_{a,B}^{-2}-\frac{1}{2}|\leq |a-\frac{1}{2}|+|y^{T}B y|\leq \tau^{-\frac{1}{2}},\ \text{when}\ |y|\leq (1+\epsilon)R(\tau)=\mathcal{O}(\sqrt{ln\tau}).\label{eq:dfV12}
\end{align}

Now we derive a second estimate for $D_1.$
By \eqref{eq:orthow}, $e^{-\frac{1}{8}|y|^2}w\chi_{R}$ is orthogonal to all the eigenvectors with eigenvalues $0,\frac{1}{2},1$, of the linear operator $-\Delta_y+\frac{1}{16}|y|^2-\frac{3}{4}-\frac{1}{2}\partial_{\theta}^2$. Thus
\begin{align}
\Big\langle e^{-\frac{1}{8}|y|^2}w\chi_{R},\ \Big(-\Delta_y+\frac{1}{16}|y|^2-\frac{3}{4}-1-\frac{1}{2}\partial_{\theta}^2\Big)e^{-\frac{1}{8}|y|^2}w\chi_{R}\Big\rangle \geq \frac{1}{2} \|w\chi_{R}\|_{\mathcal{G}}^2.
\end{align}
This and \eqref{eq:dfV12} imply
\begin{align}
D_{1}\leq -(1-\tau^{-\frac{1}{2}})\|   w\chi_{R}\|_{\mathcal{G}}^2.\label{eq:D1One}
\end{align}

Now we have obtained two estimates for $D_1$, and are ready to derive the desired one: divide $D_1$ into $D_1=\frac{9}{10}D_1+\frac{1}{10}D_1$, apply \eqref{eq:D1One} on $\frac{9}{10}D_1$, and apply \eqref{eq:D1Two} on $\frac{1}{10}D_1$ to obtain
\begin{align}
\begin{split}\label{eq:dominant}
D_1\leq &-\frac{1}{10}\Big[\Big\| \nabla_y e^{-\frac{1}{8}|y|^2} w\chi_{R}\Big\|_2^2+\frac{1}{8} \Big\| y  w\chi_{R}\Big\|_{\mathcal{G}}^2+\frac{1}{2} \Big\|\partial_{\theta}w\chi_{R} \Big\|_{\mathcal{G}}^2\Big] -\frac{1}{2}\Big\|  w\chi_{R}\Big\|_{\mathcal{G}}^2.
\end{split}
\end{align}
The advantage is that the negative parts can be used to cancel some positive terms produced in the estimates for $|D_{k}|, \ k=2,3,4,5$ below.

For $D_2$, we claim that
for some $C>0$,
\begin{align}
|D_2| \leq \frac{1}{20} \| w\chi_{R}\|_{\mathcal{G}}^2+C \Big[H_{1}^2(\tau)+|B|^6+|a_{\tau}|^2 |B|^4\Big]+\tilde\delta  e^{-\frac{1}{5}R^2}.\label{eq:D2Est}
\end{align}
To see this, we observe that the orthogonality conditions imposed on $ w\chi_{R}$ cancel many terms. For example, for the first term of $F(B,a)$ in \eqref{eq:source},
\begin{align}
\begin{split}
\Upsilon:=&\Big|\Big\langle  w\chi_{R},\  \chi_{R}\frac{y^{T}(\partial_{\tau}B +B^{T}B) y}{2 \sqrt{2a} \sqrt{2+y^{T}B y}}\Big\rangle_{\mathcal{G}}\Big|\\
\lesssim &\Big| \Big\langle  w\chi_{R},\  y^{T}(\partial_{\tau}B +B^{T}B) y\ \Big[\frac{1}{2 \sqrt{2a} \sqrt{2+y^{T}B y}}-\frac{1}{2 \sqrt{2a} \sqrt{2}}\Big]\Big\rangle_{\mathcal{G}}\Big|+\tilde\delta e^{-\frac{1}{5}R^2}\\
\lesssim &\Big\| w\chi_{R}\Big\|_{\mathcal{G}} |B|\ H_1+\tilde\delta  e^{-\frac{1}{5}R^2}
\end{split}
\end{align} where in the second step we use the estimate
\begin{align}
\Big|\Big\langle w\chi_{R},\  (\chi_{R}-1)\frac{y^{T}(\partial_{\tau}B +B^{T}B) y}{2 \sqrt{2a} \sqrt{2+y^{T}B y}}\Big\rangle_{\mathcal{G}}\Big|\lesssim \tilde\delta e^{-\frac{1}{5}R^2}
\end{align} since the function $\chi_{R}-1$ is supported by the set $|y|\in [R,\ (1+\epsilon)R];$ and
$
\Big\langle  w\chi_{R},\ \frac{y^{T}(\partial_{\tau}B +B^{T}B) y}{2 \sqrt{2a} \sqrt{2}}\Big\rangle_{\mathcal{G}}=0
$ by the orthogonality conditions in \eqref{eq:orthow}; and we apply \eqref{eq:Beqn} to control $\partial_{\tau}B+B^{T}B$.

Similarly for the other terms in $D_2$,
\begin{align}
\begin{split}
|D_2| \lesssim &\|w\chi_{R}\|_{\mathcal{G}} |B|\ \Big[H_1+|B|^2+|a_{\tau}| |B| \Big]+\tilde\delta  e^{-\frac{1}{5}R^2},
\end{split}
\end{align}Finally, apply Schwartz inequality to obtain the desired \eqref{eq:D2Est}.

For $D_3$, the orthogonality conditions make many terms vanish. Compute directly to find
\begin{align*}
|D_3|=&\Big|\Big\langle  w\chi_{R},\ \Big(\frac{2a}{2+y^{T}B y}-a\Big)\Omega_1\cdot y\Big\rangle_{\mathcal{G}}\Big|+\tilde\delta e^{-\frac{1}{5}R^2}\\
\lesssim& |\Omega_1|\ |B|\ \|w\chi_{R}\|_{\mathcal{G}}+\tilde\delta e^{-\frac{1}{5}R^2},
\end{align*} hence, for some $C>0,$
\begin{align}
|D_3|\leq \frac{1}{20} \| w\chi_{R}\|_{\mathcal{G}}^2+C |\Omega_1|^2 |B|^2+\tilde\delta e^{-\frac{1}{5}R^2}.
\end{align}

The term $D_4$ is defined in terms of $N_1.$ By $\sum_{|k|+l=1}^4|\nabla_{y}^{k}\partial_{\theta}^{l}v|\ll 1$ in \eqref{eq:condition1},
\begin{align}
|N_1|\lesssim  |\nabla_y v|^2+|\partial_{\theta}v|^2.\label{eq:easyN1}
\end{align}
And thus
\begin{align}
|D_4|\lesssim& \Big\langle | w\chi_{R}|,\ \chi_{R}|\nabla_y v|^2\Big\rangle_{\mathcal{G}}+\Big\langle | w\chi_{R}|,\ \chi_{R}|\partial_{\theta}v|^2\Big\rangle_{\mathcal{G}}.\label{eq:easyD4}
\end{align} 
Similar to proving \eqref{eq:estF22ex}, we find: for some $C>0,$
\begin{align}
\begin{split}\label{eq:estD4}
|D_4|\leq& C\tilde\delta \Big[\| \nabla_y (e^{-\frac{1}{8}|y|^2} w\chi_{R})\|_2^2+\| y  w\chi_{R}\|_{\mathcal{G}}^2+\|\partial_{\theta} w\chi_{R} \|_{\mathcal{G}}^2\Big]+\frac{1}{20}\|   w\chi\|_{\mathcal{G}}^2\\
&+C \Big[|B|+\sum_{k=1}^{3}|\Omega_k|+\sum_{l=1,2}|\alpha_l|\Big]^4+
\tilde\delta e^{-\frac{1}{5}R^2}.
\end{split}
\end{align}

Next we estimate $D_5$. By the definition of $N_{2}(\eta)$ in \eqref{eq:defN2eta}
\begin{align*}
D_{5}&=2\Big\langle  w \chi_{R},\  V_{a,B}^{-2}v^{-1}\eta^2 \chi_{R}\Big\rangle_{\mathcal{G}}+2\Big\langle  w \chi_{R},\ v^{-2}V_{a,B}^{-2} (v+V_{a,B})\eta\partial_{\theta}^2\eta \chi_{R}\Big\rangle_{\mathcal{G}}\\
&=:D_{51}+D_{52},
\end{align*}with the terms $D_{51}$ and $D_{52}$ naturally defined.
For $D_{51}$, similar to obtaining \eqref{eq:estF22ex},
\begin{align}
|D_{51}|
\lesssim  \| w \chi_{R}\|_{\mathcal{G}} \ \Big[\sum_{k=1}^3 |\Omega_k|^2+\sum_{l=1,2}|\alpha_l|^2\Big]
+\tilde\delta \| w \chi_{R}\|_2^2+\tilde\delta e^{-\frac{1}{5}R^2}.\label{eq:estD51}
\end{align} 
For $D_{52}$, to avoid estimating $\|\partial_{\theta}^2 w \chi_{R} \|_{\mathcal{G}}$, we integrate by parts in $\theta$ to obtain
\begin{align*}
D_{52}  =&-2\Big\langle \partial_{\theta}  \chi_{R}w,\ v^{-2}V_{a,B}^{-2} (v+V_{a,B})\eta\partial_{\theta}\eta \chi_{R}\Big\rangle_{\mathcal{G}}\\
&-2\Big\langle  w \chi_{R},\ V_{a,B}^{-2}\Big(\partial_{\theta}v^{-2} (v+V_{a,B})\eta\Big)\partial_{\theta}\eta \chi_{R}\Big\rangle_{\mathcal{G}}.
\end{align*} This, together with $\partial_{\theta}v=\partial_{\theta}\eta$ and $|\partial_{\theta}v|\leq \delta\ll 1$ in \eqref{eq:condition1}, implies that
\begin{align}
|D_{52}|\lesssim \Big\langle |\partial_{\theta}  \chi_{R}w|,\ |\eta|\ |\partial_{\theta}\eta| \chi_{R}\Big\rangle_{\mathcal{G}}+\Big\langle |  \chi_{R}w|,\  |\partial_{\theta}\eta|^2 \chi_{R}\Big\rangle_{\mathcal{G}}.
\end{align}
What is left is similar to obtaining \eqref{eq:estF22ex}, and thus is skipped. This and \eqref{eq:estD51} imply
\begin{align}
\begin{split}\label{eq:estD5}
|D_5|\leq& C\tilde\delta \|\partial_{\theta} \chi_{R}w \|_{\mathcal{G}}^2+\frac{1}{20}\|   w\chi_{R}\|_{\mathcal{G}}^2+ C\Big[|B|+\sum_{k=1}^{3}|\Omega_k|+\sum_{l=1,2}|\alpha_l|\Big]^4+\tilde\delta e^{-\frac{1}{5}R^2}.
\end{split}
\end{align}

Until now we estimated all the terms $D_k,$ $k=1,\cdots,5.$ What is left is to obtain the desired \eqref{eq:wL2Est}
by using the negative parts of the estimate for $D_1$ in \eqref{eq:dominant} to remove various positive terms in the estimates for $\sum_{k=2}^5|D_{k}|$.


\section{Proof of \eqref{eq:YwL2Est}}\label{sec:YwL2Est}

We start with estimating $\|\partial_{y_l}\chi_{R} w\|^2_{\mathcal{G}},\ l=1,2,3$. Recall the definitions of the inner product $\langle \cdot, \ \cdot\rangle_{\mathcal{G}}$ and the norm $\|\cdot\|_{\mathcal{G}}$ from \eqref{def:Gin}.
Since all the techniques have been used in the previous sections, the proof will be sketchy.

Take a $y_l-$derivative on the $\chi_{R}w$-equation in \eqref{eq:tildew3} to obtain
\begin{align}
\begin{split}\label{eq:y2w}
\partial_{\tau}\partial_{y_l} \chi_{R}w =-(L+\frac{1}{2})\partial_{y_l} \chi_{R} w+\partial_{y_l} \chi_{R}\big(F+G+N_1+N_2\big)+\Lambda_l (\chi_{R}w)+\partial_{y_l} \mu(w),
\end{split}
\end{align}
where $\Lambda_l (\chi_{R}w)$ is produced when we change the orders of the operators $\partial_{y_l}$ and $L$,
\begin{align}
\Lambda_l (\chi_{R}w):=L\partial_{y_l} \chi_{R}w-\partial_{y_l} L \chi_{R}w+\frac{1}{2}\partial_{y_l} \chi_{R}w=-(\partial_{y_l} V_{a,B}^{-2}) \partial_{\theta}^2 \chi_{R}w- (\partial_{y_l} V_{a,B}^{-2})  \chi_{R}w,\label{eq:DefLamb}
\end{align} the $\frac{1}{2}$ in the linear operator is produced by commutation relation: for any function $g$, 
\begin{align}
\frac{1}{2}\nabla_{y} (y\cdot \nabla_y g)-\frac{1}{2}y\cdot \nabla_y\nabla_{y}g=\frac{1}{2}\nabla_{y} g.\label{eq:commt}
\end{align}

Take a $\mathcal{G}-$inner product with $\partial_{y_l} \chi_{R}w $ to obtain
\begin{align}
\partial_{\tau}\Big\langle  \partial_{y_l}\chi_{R}w ,\   \partial_{y_l}\chi_{R}w\Big\rangle_{\mathcal{G}}
=&-2\Big\langle e^{-\frac{1}{8} |y|^2} \partial_{y_l}\chi_{R}w, \ (\tilde{L}+\frac{1}{2}) \ e^{-\frac{1}{8} |y|^2} \partial_{y_l}\chi_{R}w\Big\rangle\nonumber\\
&+2\Big\langle  \partial_{y_l}\chi_{R}w, \  \partial_{y_l} \chi_{R}(F+G)\Big\rangle_{\mathcal{G}}\nonumber\\
&+2\Big\langle  \partial_{y_l}\chi_{R}w, \ \partial_{y_l} \chi_{R}N_1\Big\rangle_{\mathcal{G}}+2\Big\langle  \partial_{y_l}\chi_{R}w, \  \partial_{y_l} \chi_{R}N_2\Big\rangle_{\mathcal{G}}\nonumber\\
&+2\Big\langle  \partial_{y_l}\chi_{R}w,\  \Lambda_l(\chi_{R}w)\Big\rangle_{\mathcal{G}}+2\Big\langle  \partial_{y_l}\chi_{R}w, \  \partial_{y_l}\mu(w)\Big\rangle_{\mathcal{G}}\nonumber\\
=&\sum_{k=1}^{6}\Psi_{l,k},\label{eq:Psilk}
\end{align} where the linear operator $\tilde{L}$ was defined in \eqref{eq:defTL}, and the terms $\Psi_{l,k}$ are naturally defined.

For $\Psi_{l,1}$, the orthogonality conditions in \eqref{eq:orthow} make $e^{-\frac{1}{8} |y|^2} \partial_{y_l}\chi_{R}w$ orthogonal to all the eigenvectors with eigenvalues $0$ and $\frac{1}{2}$ of $-\Delta_y+\frac{1}{16}|y|^2-\frac{3}{4}-\frac{1}{2}\partial_{\theta}^2.$
Similar to proving \eqref{eq:dominant},
\begin{align}
\begin{split}
\Psi_{l,1}\leq 
&-\frac{1}{10} \Big[\|\nabla_y e^{-\frac{1}{8} |y|^2} \partial_{y_l}\chi_{R}w\|_2^2+\frac{1}{8}\|y  \partial_{y_l} \chi_{R}w\|_{\mathcal{G}}^2+\frac{1}{2}\| \partial_{y_l}\partial_{\theta}\chi_{R}w\|_{\mathcal{G}}^2\Big]-\frac{1}{2}\| \partial_{y_l}\chi_{R}w\|_{\mathcal{G}}^2.
\end{split}
\end{align}

For $\Psi_{l,2}$, compute directly and use \eqref{eq:Beqn}-\eqref{eq:alpha1Eqn} to obtain, for some $C>0,$
\begin{align}
|\Psi_{l,2}|\leq \frac{1}{20}\| \partial_{y_l}\chi_{R}w\|_{\mathcal{G}}^2+C \Big[|B|^3+|a_{\tau}||B|^2+|B||\Omega_1|\Big]^2+C\tilde\delta e^{-\frac{1}{5}R^2}.
\end{align}

For $\Psi_{l,3}$, integrate by parts in $y_l$ to avoid estimating higher-derivatives of $v$, and then use
\begin{align}
|N_1|\lesssim |\nabla_y v|^2+|\partial_{\theta}v|^2
\end{align} to find
\begin{align}
\begin{split}
|\Psi_{l,3}|\lesssim & \Big\langle |\partial_{y_l} e^{-\frac{1}{8} |y|^2} \partial_{y_l}\chi_{R}w|,\ \chi_{R}e^{-\frac{1}{8} |y|^2}\Big[
|\nabla_y v|^2+|\partial_{\theta}v|^2
\Big]\Big\rangle\\
&+\Big\langle |y_l|| \partial_{y_l}\chi_{R}w|,\ \chi_{R}\Big[
|\nabla_y v|^2+|\partial_{\theta}v|^2
\Big]\Big\rangle_{\mathcal{G}}.
\end{split}
\end{align}
Similar to obtaining \eqref{eq:estF22ex}, we find, for some $C>0,$
\begin{align}
\begin{split}\label{eq:estPsi5}
|\Psi_{l,3}|\leq &\tilde\delta  \Big[\|\nabla_y e^{-\frac{1}{8} |y|^2} \partial_{y_l}\chi_{R}w\|_2^2+\| y_l \nabla_{y}\chi_{R}w\|_{\mathcal{G}}^{2}+\|\partial_{\theta}\chi_{R}w\|_{\mathcal{G}}^2+\| \nabla_{y}\chi_{R}w\|_{\mathcal{G}}^2\Big]\\
&+ C\Big[|B|+\sum_{k=1}^{3}|\Omega_k|+\sum_{l=1,2} |\alpha_l|\Big]^4+\tilde\delta e^{-\frac{1}{5}R^2}.
\end{split}
\end{align}

To estimate $\Psi_{l,4}$, we integrate by parts in $y_l$, and then use the estimate 
\begin{align}
|N_2|\lesssim |\eta| |\partial_{\theta}^2\eta|+\eta^2
\end{align} to find, for some $C>0,$
\begin{align}
\begin{split}\label{eq:estPsi6}
|\Psi_{l,4}|\lesssim &\Big\langle \Big|\partial_{y_l} e^{-\frac{1}{8} |y|^2} \partial_{y_l}\chi_{R}w\Big|,\ \chi_{R}e^{-\frac{1}{8} |y|^2}\Big[
\Big|\eta\ \partial_{\theta}^2\eta\Big|+\eta^2
\Big]\Big\rangle\\
&+\Big\langle \Big|y_l  \partial_{y_l}\chi_{R}w\Big|,\ \chi_{R}\Big[
\Big|\eta\ \partial_{\theta}^2\eta\Big|+\eta^2
\Big]\Big\rangle_{\mathcal{G}}\\
\leq &\tilde\delta  \Big[\Big\|\nabla_y e^{-\frac{1}{8} |y|^2} \partial_{y_l}\chi_{R}w\Big\|_2^2+\Big\| y_l \nabla_{y}\chi_{R}w\Big\|_{\mathcal{G}}^{2}+\Big\| \partial_{\theta}^2\chi_{R}w\Big\|_{\mathcal{G}}^2+\Big\| \chi_{R}w\Big\|_{\mathcal{G}}^2\Big]\\
&\hskip 0.5cm+ C\Big[|B|+\sum_{k=1}^{3}|\Omega_k|+\sum_{l=1,2} |\alpha_l|\Big]^4+\tilde\delta e^{-\frac{1}{5}R^2}.
\end{split}
\end{align}

For $\Psi_{l,5}$, to avoid involving $\|\partial_{\theta}^2 \chi_{R}w\|_{\mathcal{G}}$, we integrate by parts in $\theta$ in the first term to obtain
\begin{align*}
\Psi_{l,5}=&2\Big\langle  \partial_{y_l}\partial_{\theta}\chi_{R}w,\  (\partial_{y_l} V_{a,B}^{-2}) \partial_{\theta} \chi_{R}w\Big\rangle_{\mathcal{G}}-2\Big\langle \partial_{y_l}\chi_{R}w,\  (\partial_{y_l} V_{a,B}^{-2})  \chi_{R}w\Big\rangle_{\mathcal{G}}.
\end{align*}
This, together with $
|\nabla_{y}V_{a,B}^{-2}|\lesssim \tau^{-\frac{1}{2}}\leq \tilde\delta
$ (see \eqref{eq:nablayV}),
implies that, for some $C>0,$
\begin{align}
\begin{split}
|\Psi_{l,5}|
&\leq C\tilde\delta \Big[ \| \partial_{y_l}\partial_{\theta}\chi_{R}w\|_{\mathcal{G}}^2+ \sum_{|k|+l=1}\|   \partial_{\theta}^{l}\nabla_y^{k} \chi_{R}w\|_{\mathcal{G}}^2\Big].
\end{split}
\end{align}

It is easy to control $\Psi_{l,6}$, since all the terms are supported by the set $|y|\in [R,\ (1+\epsilon)R]$. Reason as in \eqref{eq:tildechi3} to obtain
\begin{align}
|\Psi_{l,6}|\lesssim \tilde\delta e^{-\frac{1}{5}R^2}.
\end{align}
Now we have completed estimating all the terms $\Psi_{l,k}$ on the right hand side of \eqref{eq:Psilk}. These estimates and the identity $\partial_{\tau}\|\nabla_y \chi_{R}w\|^2_{\mathcal{G}}=\sum_{l,k}\Psi_{l,k}$ imply the desired result \eqref{eq:YwL2Est}.


\section{Proof of \eqref{eq:TwL2Est}}\label{sec:TwL2Est}

We claim that the following two estimates hold: for some $C>0,$
\begin{align}
\begin{split}\label{eq:half1}
\Big[\frac{d}{d\tau}+\frac{3}{10}\Big]\|  &\partial_{\theta}^2 \chi_{R}w\|_{\mathcal{G}}^2
\leq -\frac{1}{20} \Big[\Big\|\nabla_{y} e^{-\frac{1}{8} |y|^2} \partial_{\theta}^2\chi_{R}w\Big\|_2^2+\frac{1}{16}\Big\|y   \partial_{\theta}^2\chi_{R}w\Big\|_{\mathcal{G}}^2+\frac{1}{2}\Big\|\partial_{\theta}^3\chi_{R}w\Big\|_{\mathcal{G}}^2\Big]\\
&+C\tilde\delta \Big( \sum_{k=1}^{3}\Big[\Big\| y  \partial_{\theta} \partial_{y_{k}}\chi_{R}w\Big\|_{\mathcal{G}}^2+\Big\| \nabla_y e^{-\frac{1}{8} |y|^2} \partial_{\theta} \partial_{y_{k}}\chi_{R}w\Big\|_2^2\Big]+\Big\|  \partial_{\theta}\nabla_y \chi_{R}w\Big\|_{\mathcal{G}}^{2}\Big)\\
&+C\Big[\sum_{k=2}^{3}|\Omega_k|+\sum_{l=1,2}|\alpha_l|\Big]^2 \Big[|B|+\sum_{k=1}^{3}|\Omega_k|+\sum_{l=1,2}|\alpha_l|\Big]^2+\tilde\delta e^{-\frac{1}{5}R^2},
\end{split}
\end{align}
and
\begin{align}
\begin{split}\label{eq:half2}
\Big[\frac{d}{d\tau}+\frac{3}{10}\Big]&\|  \partial_{\theta}\nabla_y \chi_{R}w\|_{\mathcal{G}}^{2}
\leq -\frac{1}{20}\sum_{k=1}^{3}\Big[\frac{1}{16}\Big\| y  \partial_{\theta} \partial_{y_{k}}\chi_{R}w\Big\|_{\mathcal{G}}^2+\Big\| \nabla_y e^{-\frac{1}{8} |y|^2} \partial_{\theta} \partial_{y_{k}}\chi_{R}w\Big\|_2^2\Big]\\
&+\tilde{\delta} \Big\|  \partial_{\theta}^3 \chi_{R}w\Big\|_{\mathcal{G}}^2 +C\Big[\sum_{k=2}^{3}|\Omega_k|+\sum_{l=1,2}|\alpha_l|\Big]^2 \Big[|B|+\sum_{k=1}^{3}|\Omega_k|+\sum_{l=1,2}|\alpha_l|\Big]^2\\
&+\tilde\delta e^{-\frac{1}{5}R^2}.
\end{split}
\end{align}

Assuming the claims hold, we obtain the desired \eqref{eq:TwL2Est} by combining these two estimates and observing some obvious cancellations.

In what follows we prove the claims \eqref{eq:half1} and \eqref{eq:half2}.

\subsection{Proof of \eqref{eq:half1}}
We derive a governing equation for $\chi_{R}\partial_{\theta}^2 w$ by taking $\partial_{\theta}^2$ on \eqref{eq:tildew3} and using $\partial_{\theta}^2 F=0$ (since $F$ is $\theta-$independent) to find
\begin{align}
\partial_{\tau}\chi_{R}\partial_{\theta}^2 w=-L\partial_{\theta}^2\chi_{R}w+\partial_{\theta}^2\chi_{R}\Big[G+N_1+N_2\Big]+\partial_{\theta}^2\mu(w).
\end{align} 
Take a $\mathcal{G}-$inner product with $\chi_{R}\partial_{\theta}^2 w$ to find
\begin{align}
&\partial_{\tau}\Big\langle\chi_{R}\partial_{\theta}^2 w,\ \chi_{R}\partial_{\theta}^2 w\Big\rangle_{\mathcal{G}}\nonumber\\
=&-2\Big\langle  \chi_{R}e^{-\frac{1}{8}|y|^2}\partial_{\theta}^2 w,\ \tilde{L}\ \chi_{R}e^{-\frac{1}{8}|y|^2}\partial_{\theta}^2 w\Big\rangle+2\Big\langle  \chi_{R}\partial_{\theta}^2 w,\  \partial_{\theta}^2 \chi_{R} G\Big\rangle_{\mathcal{G}}\nonumber\\
&+2\Big\langle  \chi_{R}\partial_{\theta}^2 w,\ \partial_{\theta}^2\chi_{R}N_1\Big\rangle_{\mathcal{G}}+2\Big\langle  \chi_{R}\partial_{\theta}^2 w,\ \partial_{\theta}^2\chi_{R}N_2\Big\rangle_{\mathcal{G}}+2\Big\langle  \chi_{R}\partial_{\theta}^2 w,\  \partial_{\theta}^2 \mu(w)\Big\rangle_{\mathcal{G}}\nonumber\\
=&\sum_{k=1}^{5}\Psi_{k},\label{eq:theta2w2}
\end{align} where the linear operator $\tilde{L}$ was defined in \eqref{eq:defTL}, and the terms $\Psi_{k}$ are naturally defined. 

To control $\Psi_1$, we observe that $\partial_{\theta}^2\chi_{R} w$ is $\mathcal{G}$-orthogonal to the same 18 functions listed in \eqref{eq:orthow}, then by the techniques of proving \eqref{eq:dominant}
\begin{align}
\begin{split}\label{eq:phi1}
\Psi_1\leq &-\frac{1}{10} \Big[\|\nabla_{y} e^{-\frac{1}{8} |y|^2} \partial_{\theta}^2\chi_{R}w\|_2^2+\frac{1}{16}\|y   \partial_{\theta}^2\chi_{R}w\|_{\mathcal{G}}^2+\frac{1}{2}\| \partial_{\theta}^3\chi_{R}w\|_{\mathcal{G}}^2\Big]-\frac{1}{2} \|\chi_{R}\partial_{\theta}^2 w\|_{\mathcal{G}}^{2}.
\end{split}
\end{align}

For $\Psi_{2}$ we simplify the expression by observing some cancellation and then control it
\begin{align}
\begin{split}\label{eq:addOne}
|\Psi_2|=&2\Big|\Big\langle  \chi_{R}\partial_{\theta}^2 w,\  \partial_{\theta}^2 (\chi_{R}-1) G\Big\rangle_{\mathcal{G}}\Big|
\leq \tilde\delta e^{-\frac{1}{5}R^2},
\end{split}
\end{align} where we use the identity $$\langle  \chi_{R}\partial_{\theta}^2 w,\ \partial_{\theta}^2  G\rangle_{\mathcal{G}}=0,$$which holds since the $\theta-$derivative removes the $\theta-$independent part of $G$, and then the orthogonality conditions satisfied by $\partial_{\theta}^2\chi_{R} w$ remove the remaining part; and in the second step we use that $1-\chi_{R}$ is supported by the set $\Big\{y\ \Big|\ |y|\geq R\Big\}$, and \eqref{eq:paramePrelim}.

To prepare for estimating $\Psi_3$, we integrate by parts in $\theta$, in order to avoid estimating $\partial_{\theta}^2\nabla_{y}^{k}v,$ $|k|=2$,
\begin{align*}
\Psi_{3}=-2\Big\langle  \chi_{R}\partial_{\theta}^3 w,\  \partial_{\theta}\chi_{R}N_1\Big\rangle_{\mathcal{G}}.
\end{align*}
For the $\partial_{\theta}N_1-$term, compute directly to obtain
\begin{align}\label{eq:thetaN1}
|\partial_{\theta}N_1|\lesssim  \Big[|\nabla_y v|+|\partial_{\theta}v|\Big]\Big[|\nabla_{y}\partial_{\theta}v|+|\partial_{\theta}^2v|\Big]+|\nabla_y v| \sum_{|k|=2}|\nabla_{y}^{k}\partial_{\theta}v|.
\end{align}
We put this back into $\Psi_3$, and use the previously used techniques, for example those in obtaining \eqref{eq:estF22ex} to obtain, for some $C>0,$
\begin{align}
\begin{split}\label{eq:threeSte}
|\Psi_3|\leq &C\tilde\delta\Big\{ \sum_{k=1}^{3}\Big[\Big\| y \partial_{\theta} \partial_{y_{k}}\chi_{R}w\Big\|_{\mathcal{G}}^2+\Big\| \nabla_y e^{-\frac{1}{8} |y|^2} \partial_{\theta} \partial_{y_{k}}\chi_{R}w\Big\|_2^2\Big]+\Big\|  \partial_{\theta}\nabla_y \chi_{R}w\Big\|_{\mathcal{G}}^{2}\Big\}\\
&+C\tilde\delta \Big\{\Big\|\partial_{\theta}^3  \chi_{R}w\Big\|_{\mathcal{G}}^2+\Big\|y \partial_{\theta}^2  \chi_{R}w\Big\|_{\mathcal{G}}^2+\Big\|\nabla_{y}e^{-\frac{1}{8}|y|^2}\partial_{\theta}^3  \chi_{R}w\Big\|_2^2\Big\}\\
&+C\Big\{\sum_{k=2}^{3}|\Omega_k|+\sum_{l=1,2}|\alpha_l|\Big\}^2 \Big\{|B|+\sum_{k=1}^{3}|\Omega_k|+\sum_{l=1,2}|\alpha_l|\Big\}^2+\tilde\delta e^{-\frac{1}{5}R^2}.
\end{split}
\end{align}

Similarly for $\Psi_{4}$, we transform the expression by integrating by parts in $\theta$, 
\begin{align*}
\Psi_{4}=-2\Big\langle  \chi_{R}\partial_{\theta}^3 w,\ \partial_{\theta}\chi_{R}N_2\Big\rangle_{\mathcal{G}}.
\end{align*}
then by $|\eta|\leq \delta$ for $|y|\leq (1+\epsilon)R$ (see \eqref{eq:condition1}) and $\partial_{\theta}v=\partial_{\theta}\eta$,
\begin{align}\label{eq:thetaN2}
|\partial_{\theta}N_2|\lesssim |\partial_{\theta}^3\eta||\eta|+|\partial_{\theta}^2\eta||\partial_{\theta}\eta|+|\partial_{\theta}\eta| |\eta|.
\end{align}
Put this back into the definition of $\Psi_4$, for some $C>0,$
\begin{align}
\begin{split}
|\Psi_4|\leq &\frac{1}{40} \sum_{k=1}^{3}\Big\|\chi_{R}\partial_{\theta}^k w\Big\|_{\mathcal{G}}^2\\
&+C\Big[|B|^2+\sum_{k=1}^{3}|\Omega_k|^2+\sum_{l=1,2}|\alpha_l|^2\Big]\Big[\sum_{k=2}^{3}|\Omega_k|^2+\sum_{l=1,2}|\alpha_l|^2\Big]+ \tilde\delta e^{-\frac{1}{5}R^2}.
\end{split}
\end{align}

$\Psi_{5}$ is defined in terms of $\mu(w)$, which depends on the derivatives of $\chi_R$ and thus is supported by the set $|y|\geq R,$ 
\begin{align}
|\Psi_5|\lesssim \tilde\delta e^{-\frac{1}{5}R^2}.
\end{align}

Now we have completed estimating all the terms on the right hand side of \eqref{eq:theta2w2}. Collecting the estimates above, we complete the proof of \eqref{eq:half1}.

\subsection{Proof of \eqref{eq:half2}}
We start with estimating $\|\partial_{\theta}\partial_{y_l} \chi_{R}w\|_{\mathcal{G}}^2, \ l=1,2,3$. 

Derive a governing equation for $\partial_{\theta}\partial_{y_l}\chi_{R} w$ by taking a $\theta-$derivative on both sides of \eqref{eq:y2w},
\begin{align*}
\partial_{\tau}\partial_{\theta}\partial_{y_l} \chi_{R} w=&-(L+\frac{1}{2})\partial_{\theta}\partial_{y_l}\chi_{R}w+\partial_{\theta}\partial_{y_l} \Big[\chi_{R}\big(G+N_1+N_2\big)+\mu(w)\Big]+\partial_{\theta}\Lambda_l (\chi_{R}w),
\end{align*} where we used that $\partial_{\theta}F=0$, since $F$ is independent of $\theta$. Take a $\mathcal{G}-$inner product with $\partial_{\theta}\partial_{y_l} \chi_{R}w$ to derive,
\begin{align}
&\partial_{\tau}\Big\langle  \partial_{\theta}\partial_{y_l} \chi_{R}w,\  \partial_{\theta}\partial_{y_l} \chi_{R}w\Big\rangle_{\mathcal{G}}\nonumber\\
=&-2\Big\langle e^{-\frac{1}{8}|y|^2} \partial_{\theta}\partial_{y_l} \chi_{R}w,\ (\tilde{L}+\frac{1}{2}) e^{-\frac{1}{8}|y|^2} \partial_{\theta}\partial_{y_l} \chi_{R}w\Big\rangle+2\Big\langle \partial_{\theta}\partial_{y_l} \chi_{R}w,\  \partial_{\theta}\partial_{y_l} \chi_{R}G\Big\rangle_{\mathcal{G}}\nonumber\\
&+2\Big\langle \partial_{\theta}\partial_{y_l} \chi_{R}w,\  \partial_{\theta}\partial_{y_l} \chi_{R}(N_1+N_2)\Big\rangle_{\mathcal{G}}+2\Big\langle\partial_{\theta}\partial_{y_l} \chi_{R}w,\ \partial_{\theta}\partial_{y_l} \mu(w)\Big\rangle_{\mathcal{G}}\nonumber\\
&+2\Big\langle  \partial_{\theta}\partial_{y_l} \chi_{R}w,\ \partial_{\theta}\Lambda_l(\chi_{R}w)\Big\rangle_{\mathcal{G}}\nonumber\\
=&\sum_{k=1}^{5}\Pi_{l,k},\label{eq:thetay2w}
\end{align} here the terms $\Pi_{l, k}$ are naturally defined, the linear operator $\tilde{L}$ was defined in \eqref{eq:defTL}.

For $\Pi_{l,1}$, \eqref{eq:orthow} implies that $e^{-\frac{1}{8} |y|^2}\partial_{\theta}\partial_{y_l} \chi_{R} w$ 
is orthogonal to all the eigenvectors with eigenvalues $0,\frac{1}{2}$ of the linear operator $-\Delta_y+\frac{1}{16}|y|^2-\frac{3}{4}-\frac{1}{2}\partial_{\theta}^2.$ Similar to proving \eqref{eq:dominant}, 
\begin{align}
\begin{split}
\Pi_{l,1}
\leq &-\frac{1}{10} \Big[\Big\|\nabla_{y} e^{-\frac{1}{8} |y|^2} \partial_{\theta}\partial_{y_l}\chi_{R}w\Big\|_2^2+\frac{1}{16}\Big\|y   \partial_{\theta}\partial_{y_l}\chi_{R}w\Big\|_{\mathcal{G}}^2+\frac{1}{2}\Big\|
 \partial_{\theta}^2\partial_{y_l}\chi_{R}w\Big\|_{\mathcal{G}}^2\Big]\\
 &-\frac{1}{2} \|\partial_{\theta}\partial_{y_l}\chi_{R} w\|_{\mathcal{G}}^{2}.
 \end{split}
\end{align}

For $\Pi_{l,2}$, we argue as in \eqref{eq:addOne} to obtain
\begin{align}
|\Pi_{l,2}|=2\Big|\Big\langle \partial_{\theta}\partial_{y_l} \chi_{R}w,\  \partial_{\theta}\partial_{y_l}(\chi_{R}-1) G\Big\rangle_{\mathcal{G}}\Big|\leq \tilde\delta e^{-\frac{1}{5}R^2}.
\end{align}

For $\Pi_{l,3}$, we integrate by parts in $y_l$ to avoid estimating some high derivatives of $v$ to obtain
\begin{align}
\Pi_{l,3}=-2\Big\langle \partial_{y_l}e^{-\frac{1}{4}|y|^2} \partial_{\theta}\partial_{y_l} \chi_{R}w,\  \partial_{\theta} \chi_{R}(N_1+N_2)\Big\rangle,
\end{align} and then take absolute value,
\begin{align*}
|\Pi_{l,3}|\lesssim &\Big\langle |y_l| | \partial_{\theta}\partial_{y_l} \chi_{R}w|,  \chi_{R}|\partial_{\theta} (N_1+N_2)|\Big\rangle_{\mathcal{G}}\\
&+\Big\langle |\partial_{y_l} e^{-\frac{1}{8}|y|^2}\partial_{\theta}\partial_{y_l} \chi_{R}w|, \ e^{-\frac{1}{8}|y|^2}\chi_{R}\Big|\partial_{\theta} (N_1+N_2)\Big|\Big\rangle.
\end{align*}
$\partial_{\theta}N_k$, $k=1,2,$ were estimated in \eqref{eq:thetaN1} and \eqref{eq:thetaN2}. These and \eqref{eq:poincare} imply that, for some $C>0$,
\begin{align}
\begin{split}
|\Pi_{l,3}|\leq & \tilde\delta \Big[\|y_l \partial_{\theta}\partial_{y_l} \chi_{R}w\|_{\mathcal{G}}^2+\|\partial_{y_l}e^{-\frac{1}{8}|y|^2} \partial_{\theta}\partial_{y_l} \chi_{R}w\|_2^2+\|\chi_{R}\partial_{\theta}^3 w\|_{\mathcal{G}}^2+\| \partial_{\theta}\partial_{y_l} \chi_{R}w\|_{\mathcal{G}}^2\Big]\\
&+C\Big[|B|^2+\sum_{k=1}^{3}|\Omega_k|^2+\sum_{l=1,2}|\alpha_l|^2\Big]\Big[\sum_{k=2}^{3}|\Omega_k|^2+\sum_{l=1,2}|\alpha_l|^2\Big]+ \tilde\delta e^{-\frac{1}{5}R^2}.
\end{split}
\end{align}

It is easy to estimate $\Pi_{l,4}$ since all the terms in $\mu(w)$ are supported by the set $\Big\{ y\ \Big|\ |y|\geq R\Big\},$ 
\begin{align}
 |\Pi_{l,4}|\lesssim \tilde\delta e^{-\frac{1}{5}R^2}.
\end{align}

For $\Pi_{l,5},$ we use the definition of $\Lambda$ in \eqref{eq:DefLamb}, \eqref{eq:poincare} and that $|\nabla_{y}V_{a,B}|\lesssim \tau^{-\frac{3}{5}}|y|\leq \tilde\delta$ in \eqref{eq:nablayV} to obtain, for some $C>0,$
\begin{align}
|\Pi_{l,5}|\leq C \tilde\delta \Big[\|  \partial_{\theta}\partial_{y_l} \chi_{R}w\|_{\mathcal{G}}^2+\|  \partial_{\theta}^2 \chi_{R}w\|_{\mathcal{G}}^2 \Big].
\end{align}

Thus we have completed estimating all the terms on the right hand side of \eqref{eq:thetay2w}.
These estimates and the identity $\| \partial_{\theta}\nabla_y \chi_{R}w\|_{\mathcal{G}}^2=\sum_{l=1}^{3}\|  \partial_{\theta}\partial_{y_l} \chi_{R}w\|_{\mathcal{G}}^2$ imply the desired \eqref{eq:half2}.


\section{Proof of \eqref{eq:weightLInf}}\label{sec:ReforWeightLInf}

We will prove the desired results by apply Gronwall's inequality, see \eqref{eq:vecM}-\eqref{eq:gron} below.
To facilitate the discussions we define controlling functions $\mathcal{M}_{k},\ k=1,2,3,4,$ as
\begin{align}
\begin{split}\label{eq:defM}
\mathcal{M}_{1}(\tau):=&\max_{\tau_0\leq s\leq\tau} R^{4}(s)\Big\|\langle y\rangle^{-3} \chi_{R(s)}w(\cdot,s)\Big\|_{\infty},\\
\mathcal{M}_{2}(\tau):=&\max_{\tau_0\leq s\leq\tau} R^{3}(s)\Big\|\langle y\rangle^{-2} \nabla_y \chi_{R(s)}w(\cdot,s)\Big\|_{\infty},\\
\mathcal{M}_{3}(\tau):=&\max_{\tau_0\leq s\leq\tau} R^{3}(s)\Big\|\langle y\rangle^{-2} \partial_{\theta} \chi_{R(s)}w(\cdot,s)\Big\|_{\infty},\\
\mathcal{M}_{4}(\tau):=&\max_{\tau_0\leq s\leq\tau} R^{2}(s)\sum_{|k|+l=2}\Big\|\langle y\rangle^{-1} \nabla_{y}^{k}\partial_{\theta}^{l} \chi_{R(s)}w(\cdot,s)\Big\|_{\infty}.
\end{split}
\end{align}

These functions satisfy the following estimates: recall the definitions of the constants $\tilde\delta$ and $\kappa(\epsilon)$ from \eqref{eq:defTDelta} and \eqref{eq:defKappa},
\begin{proposition}\label{prop:weight} There exists a constant $C>0$, such that if the condition \eqref{eq:condition1} holds for $\tau\in [\tau_0,\tau_1]$, then in the same time interval,
\begin{align}
\mathcal{M}_1(\tau)\leq &C\Big[\tilde\delta \kappa(\epsilon)+1+\big(\tilde\delta+R^{-\frac{1}{4}}(\tau_0)\big)\sum_{k=1}^{3}\mathcal{M}_{k}+\tilde\delta R^{-1}(\tau_0) \sum_{k=1}^{3}\mathcal{M}_{k}^2 \Big],\label{eq:estM1}\\
\mathcal{M}_2(\tau)\leq &C\Big[\tilde\delta \kappa(\epsilon)+1+\big(\tilde\delta+R^{-\frac{1}{4}}(\tau_0)\big) \sum_{k=1}^{3}\mathcal{M}_{k}(\tau)  \Big],\label{eq:estM2}\\
\mathcal{M}_3(\tau)\leq &C\Big[\tilde\delta \kappa(\epsilon)+1+\big(\tilde\delta+R^{-\frac{1}{4}}(\tau_0)\big)\sum_{k=1}^{3}\mathcal{M}_{k}(\tau)  \Big],\label{eq:estM3}\\
\mathcal{M}_{4}(\tau)\leq &C\Big[\tilde\delta \kappa(\epsilon)+1+\mathcal{M}_3(\tau)+\big(\tilde\delta+R^{-\frac{1}{4}}(\tau_0)\big)\sum_{k=1}^{4}\mathcal{M}_{k}(\tau)\Big].\label{eq:estM4}
\end{align}
\end{proposition}
\eqref{eq:estM1}, \eqref{eq:estM2} and \eqref{eq:estM3} will be proved in Sections \ref{sec:estM1}, \ref{sec:estM2} and \ref{sec:estM3} respectively.  
In Sections \ref{sec:estM402} and \ref{sec:estM420} we will prove \eqref{eq:estM4} when $(|k|,l)=(2,0),\ (0,2).$ The cases $(|k|,l)=(1,1)$ can be treated similarly, hence its proof is skipped.

Assuming Proposition \ref{prop:weight}, we are ready to prove \ref{eq:weightLInf}.
\begin{proof}
The tool is the standard Gronwall inequality. To prepare for its application we define a new function $\mathcal{M}$ as 
\begin{align}
\mathcal{M}(\tau):=\displaystyle\sum_{k=1}^{3}\mathcal{M}_{k}(\tau).
\end{align} \eqref{eq:estM1}-\eqref{eq:estM3} do not depend on $\mathcal{M}_4$, and can be emerged into one inequality
\begin{align}\label{eq:vecM}
\mathcal{M}(\tau)\leq 3C\Big[\tilde\delta \kappa(\epsilon)+1\Big]+3C\big(\tilde\delta+R^{-\frac{1}{4}}(\tau_0)\big) \mathcal{M}(\tau)+C\tilde\delta R^{-1}(\tau_0) \mathcal{M}^2(\tau).
\end{align}

For $\mathcal{M}(\tau_0),$ the estimates on $v-\sqrt{2}$ and its derivatives in \eqref{eq:IniWeighted}, the decomposition of $v$ in \eqref{eq:decomVToW} and the estimates in \eqref{eq:paramePrelim} imply that if $\tau_0$ is sufficiently large to make $\tau_0^{-\alpha}R^{4}(\tau_0)\ll 1$, then $\sum_{m+|k|+l=3}\Big\|\langle y\rangle^{-m}\nabla_{y}^{k}\partial^{l}_{\theta}\chi_{R}w(\cdot,\tau_0)\Big\|_{\infty}\ll R^{-4}(\tau_0)$, and thus,
\begin{align}\label{eq:IniSm}
\mathcal{M}(\tau_0)\leq  1.
\end{align} 

\eqref{eq:vecM}, \eqref{eq:IniSm} and that $\tilde\delta+R^{-\frac{1}{4}}(\tau_0)\ll 1$ make Gronwall inequality applicable, thus
\begin{align}\label{eq:gron}
\mathcal{M}(\tau)\leq 4C\Big[ \tilde\delta \kappa(\epsilon)+1\Big],\ \text{when}\ \tau\in[ \tau_0,\ \tau_1].
\end{align}
Plug this into \eqref{eq:estM4} to find 
\begin{align}
\mathcal{M}_4(\tau)\leq 10C\Big[ \tilde\delta \kappa(\epsilon)+1\Big].
\end{align}

These, together with the definitions of $\mathcal{M}_{k}$, $k=1,2,3,4,$ imply the desired result.

\end{proof}

To complete the proof we prove \eqref{eq:weightLInf} in the subsequent subsections.

\section{Proof of \eqref{eq:estM1} }\label{sec:estM1}
Compared to the subsequent sections, the presentation here are the most detailed, so that we can skip some of the details later.

We start with casting problem into a convenient form.

Rewrite the equation for $\chi_{R} w$ in \eqref{eq:tildew3} as
\begin{align}
\begin{split}\label{eq:LinfW}
\partial_{\tau} \chi_{R}w=&-L_1(\chi_{R} w)+\mu(w)+(V_{a,B}^{-2}-\frac{1}{2})(\partial_{\theta}^2 \chi_{R}w +\chi_{R}w)+\big(F+G+N_1+N_2\big)\chi_{R}
\end{split}
\end{align} where, recall the definition of $\mu(w)$ in \eqref{eq:Tchi3}, and the operator $L_1$ is defined as
\begin{align}
L_1:=-\Delta_y+\frac{1}{2}y\cdot \nabla_y-\frac{1}{2}\partial_{\theta}^2-1.
\end{align}

Before estimating the terms on the right hand side we present the difficulties and ideas in overcoming them. We will prove, in Proposition \ref{Prop:weight3} below, that the terms $(V_{a,B}^{-2}-\frac{1}{2})(\partial_{\theta}^2 \chi_{R}w +\chi_{R}w)$ and $\big(F+G+N_1+N_2\big)\chi_{R}$ on the right hand side decay sufficiently rapidly, hence are not difficult to control.

The focus is on $\mu(w)$, since it contains a difficult term. To isolate it we observe that
\begin{align}
\mu(w)=\frac{1}{2}(y\cdot\nabla_{y}\chi_{R})w+\Gamma(w)
\end{align} where the term $\Gamma(w)$ is the easy part, and is defined as
\begin{align*}
\Gamma(w):=(\partial_{\tau}\chi_{R})w-(\Delta_{y}\chi_{R})w-2\nabla_{y}\chi_{R}\cdot  \nabla_{y}w.
\end{align*}
$\Gamma(w)$ decays rapidly:
the definition $\chi_{R}(y)=\chi(\frac{y}{R})$ makes $\nabla_y \chi_{R}=R^{-1}(\nabla_x\chi)(\frac{y}{R})$ and $\Delta \chi_{R}=R^{-2} (\Delta_x \chi)(\frac{y}{R})$ small. This, together with $|w|+ |\nabla_{y}w|\lesssim \tilde\delta$ (see \eqref{eq:estW}), implies, 
\begin{align}
\|\langle y\rangle^{-3}\Gamma(w)\|_{\infty}\lesssim \tilde\delta \kappa(\epsilon)R^{-4},\label{eq:gammaw}
\end{align} where, recall the definition of $\kappa(\epsilon)$ from \eqref{eq:defKappa}.

Now we turn to the difficult term $\frac{1}{2}(y\cdot\nabla_{y}\chi_{R})w$. The obstacle is caused by two facts: (1) $\frac{1}{2}y\cdot\nabla_{y}\chi_{R}$ is of order one in $L^{\infty}$ norm, since the definition $\chi_{R}(y)=\chi(\frac{y}{R})$ makes
\begin{align}
\sup_{y}\Big|\frac{1}{2}y\cdot\nabla_{y}\chi_{R}(y)\Big|=\sup_{x}\Big|\frac{1}{2}x\cdot \nabla_{x}\chi(x)\Big| ,
\end{align} consequently it can not be treated as a small term; (2) to make it even worse, the mapping $\chi_{R}w\rightarrow \frac{1}{2}(y\cdot \nabla_{y}\chi_{R})w=\frac{1}{2}\frac{y\cdot \nabla_{y}\chi_{R}}{\chi_{R}}\chi_{R}w$ is unbounded since $|\frac{y\cdot \nabla_{y}\chi_{R}}{\chi_{R}}|\rightarrow \infty$ as $|y|\rightarrow (1+\epsilon)R$.

To overcome the first difficulty, we observe, by requiring $\chi(z)=\chi(|z|)$ to be a decreasing function (see \eqref{eq:defChi3}), $\frac{1}{2}y\cdot\nabla_{y}\chi_{R}$ is non-positive. It is favorable since a non-positive multiplier on the right hand side of \eqref{eq:LinfW} should help $\chi_{R}w$ to decay more rapidly.
For the second difficulty, the strategy is to absorb ``most" of it into the linear operator. For that purpose we define a new non-negative smooth cutoff function $\tilde\chi_{R}(y)$ such that
\begin{align}
\tilde\chi_{R}(y)=\left[
\begin{array}{ll}
1 ,\ \text{if}\ |y|\leq R(1+\epsilon- R^{-\frac{1}{4}}),\\
0,\ \text{if}\ |y|\geq R(1+\epsilon-2 R^{-\frac{1}{4}})
\end{array}
\right.
\end{align} and require it to satisfy the estimate
\begin{align}
|\nabla_{y}^{k}\tilde{\chi}_{R}(y)|\lesssim R^{-\frac{3}{4}|k|},\ |k|=1, \ 2.
\end{align} Such a function is easy to construct, hence we skip the details. 

Then we decompose $\frac{1}{2} (y\cdot\nabla_y  \chi_{R}) w $ into two parts, recall that $\chi(|x|)\rightarrow 0$ rapidly as $|x|\rightarrow 1+\epsilon$, see \eqref{eq:properties},
\begin{align}
\frac{1}{2} (y\cdot\nabla_y \chi_{R}) w=\frac{1}{2} \frac{(y\cdot\nabla_y \chi_{R}) \tilde\chi_{R}}{\chi_{R}}\chi_{R} w+\frac{1}{2} (y\cdot\nabla_y \chi_{R}) (1-\tilde\chi_{R})w.\label{eq:unboundtwo}
\end{align} 
The following three observations will be used often in the rest of the paper:
\begin{itemize}
\item[(A)]
The multiplier $\frac{1}{2} \frac{(y\cdot\nabla_y \chi_{R}) \tilde\chi_{R}}{\chi_{R}}$ in \eqref{eq:unboundtwo} is bounded: \eqref{eq:properties} implies that, for some $c(\epsilon)>0,$
\begin{align}
\Big|\frac{y\cdot\nabla_{y} \chi_{R}\  \tilde\chi_{R}}{\chi_{R}}\Big|\leq c(\epsilon) R^{\frac{1}{4}}.\label{eq:NewPoten}
\end{align}
\item[(B)]
Provided that $R$ is sufficiently large, the second part in \eqref{eq:unboundtwo} decays rapidly, 
\begin{align}
\sum_{|k|=0,1,2}\Big\|\nabla_{y}^{k}\big\{(y\cdot\nabla_{y} \chi_{R}) (1-\tilde\chi_{R})w\big\}\Big\|_{\infty}\leq \tilde\delta R^{-5} ,\label{eq:tails}
\end{align}where we use that $\Big|\frac{d}{d|z|}\chi(|z|)\Big|\rightarrow 0$ rapidly as $|z|\rightarrow 1+\epsilon$, see \eqref{eq:properties},  
\begin{align}
\Big|\nabla_{y}^{k}\Big((y\cdot\nabla_{y} \chi_{R}) (1-\tilde\chi_{R})\Big)\Big|\leq R^{-5},
\end{align}  and $\sum_{|k|=0,1,2}|\nabla_{}^{k}w|\lesssim \tilde\delta$ proved in \eqref{eq:estW}. 
\item[(C)]
If $R$ is sufficiently large, then the properties of $\chi$ in \eqref{eq:properties} implies that
\begin{align}
\Big|\nabla_{y}^{k}\big[\frac{y\cdot\nabla_{y} \chi_{R}\  \tilde\chi_{R}}{\chi_{R}}\big]\Big|\leq R^{-\frac{1}{4}},\ |k|=1,2.\label{eq:deriveSmooth}
\end{align} 
\end{itemize}

Returning to the equation for $\chi_{R}w$ in \eqref{eq:LinfW}, we move the first part in \eqref{eq:unboundtwo} into the linear operator and leave the second to the remainder. The equation becomes
\begin{align}
\partial_{\tau}(\chi_{R}w)=&-L_2(\chi_{R} w)+
\Psi(w)+\chi_{R}(F+G+N_1+N_2),\label{eq:weightLinfw}
\end{align} where the linear operator $L_2$ is defined as
\begin{align*}
L_2:=L_1-\frac{1}{2}\frac{\tilde\chi_{R}\ y\cdot \nabla_{y} \chi_{R}  }{\chi_{R}}=-\Delta_y+\frac{1}{2}y\cdot \nabla_y-\frac{1}{2}\partial_{\theta}^2-1+\frac{1}{2}\Big|\frac{\tilde\chi_{R}\ y\cdot \nabla_{y} \chi_{R}  }{\chi_{R}}\Big|,
\end{align*} and $\Psi$ is a linear operator defined as
\begin{align}
\begin{split}\label{eq:defPsiw}
\Psi(w):=
&(\frac{1}{2}y\cdot \nabla_{y} \chi_{R})\ (1-\tilde\chi_{R}) w+(\partial_{\tau}\chi_{R})w-(\Delta_{y}\chi_{R})w-2\nabla_{y}\chi_{R}\cdot  \nabla_{y}w\\
&+(V_{a,B}^{-2}-\frac{1}{2})\partial_{\theta}^2 w \chi_{R}+(V_{a,B}^{-2}-\frac{1}{2})\chi_{R}w.
\end{split}
\end{align}

Next we prepare for estimating $\chi_{R}w$ by casting \eqref{eq:weightLinfw} into a convenient form.

Observing that the operator $L_2$, mapping $L^2$ space into itself, is not self-adjoint, while its conjugation $e^{-\frac{1}{8}|y|^2}L_2 e^{\frac{1}{8}|y|^2}$ is. We transform the equation accordingly
\begin{align}
\begin{split}\label{eq:selfadw}
\partial_{\tau} (e^{-\frac{1}{8}|y|^2}\chi_{R}w)=&-\mathcal{L}(e^{-\frac{1}{8}|y|^2}\chi_{R}w)+e^{-\frac{1}{8}|y|^2}\Big[\chi_{R}(F+G+N_1+N_2)+
\Psi(w)\Big],
\end{split}
\end{align}
with the linear operator $\mathcal{L}$ defined as
\begin{align}
\mathcal{L}:=e^{-\frac{1}{8}|y|^2} L_2 e^{\frac{1}{8}|y|^2}=-\Delta_{y}+\frac{1}{16}|y|^2-\frac{3}{4}-\frac{1}{2}
\partial_{\theta}^2-1+\frac{1}{2}\Big|\frac{\tilde\chi_{R}\ y\cdot \nabla_{y} \chi_{R}  }{\chi_{R}}\Big|.
\end{align}

By \eqref{eq:orthow}, $e^{-\frac{1}{8}|y|^2}\chi_{R}w$ is orthogonal to 18 functions, which are the eigenvectors of $-\Delta_y+\frac{1}{16}|y|^2-\frac{3}{4}-\frac{1}{2}\partial_{\theta}^2-1$ with eigenvalues $-1,\ -\frac{1}{2}$ and $0$. Denote, by $P_{18}$, the orthogonal projection onto the $L^2$ subspace orthogonal to these 18 functions. This makes
\begin{align}
P_{18}e^{-\frac{1}{8}|y|^2}\chi_{R}w=e^{-\frac{1}{8}|y|^2}\chi_{R}w.\label{eq:defP21}
\end{align}
Apply $P_{18}$, and then Duhamel's principle to \eqref{eq:selfadw} to obtain
\begin{align}
\begin{split}\label{eq:3w}
e^{-\frac{1}{8}|y|^2}\chi_{R}w=&U_1(\tau, \tau_0) e^{-\frac{1}{8}|y|^2}\chi_{R}w(\tau_0)\\
&+\int_{\tau_0}^{\tau}U_1(\tau,s) P_{18}e^{-\frac{1}{8}|y|^2}\Big(\chi_{R}[F+G+N_1+N_2]+\Psi(w)\Big)(s)\ ds,
\end{split}
\end{align} where $U_1(\sigma_1,\sigma_2)$ is the propagator generated by the linear operator $-P_{18}\mathcal{L}P_{18}$ from $\sigma_2$ to $\sigma_1$.

The propagator $U_1(\tau,s)$ generates a decay rate, and it plays a crucial role:
\begin{lemma}\label{LM:propagator}
There exists a constant $C$, such that for any function $g$ and for any $\sigma_1\geq \sigma_2$,
\begin{align}
\Big\|\langle y\rangle^{-3} e^{\frac{1}{8}|y|^2} U_1(\sigma_1,\sigma_2) P_{18}g\Big\|_{\infty}\leq C e^{-\frac{2}{5}(\sigma_1-\sigma_2)} \Big\|\langle y\rangle^{-3}e^{\frac{1}{8}|y|^2} g\Big\|_{\infty}.
\end{align}

\end{lemma}
The lemma will be proved in Appendix \ref{sec:propagator}.

Return to \eqref{eq:3w} and apply the propagator estimate to find
\begin{align}
\begin{split}
\Big\|\langle y\rangle^{-3}\chi_{R}w(\cdot,\tau)\Big\|_{\infty}&\lesssim e^{-\frac{2}{5}(\tau-\tau_0)}\Big\|\langle y\rangle^{-3}\chi_{R}w(\cdot,\tau_0)\Big\|_{\infty}\\
+&
\int_{\tau_0}^{\tau} e^{-\frac{2}{5}(\tau-s)} \Big\|\langle y\rangle^{-3}\Big(\chi_{R}[F+G+ N_1+N_2]+\Psi(w)\Big)(s)\Big\|_{\infty}\ ds.\label{eq:3wEnd}
\end{split}
\end{align}

The terms on the right hand side satisfy the following estimates:
\begin{proposition}\label{Prop:weight3}
Recall the constants $\tilde\delta$ and $\kappa(\epsilon)$ from \eqref{eq:defTDelta} and \eqref{eq:defKappa}.
\begin{align}
\|\langle y\rangle^{-3}\Psi(w)\|_{\infty} \lesssim & \kappa(\epsilon)\tilde\delta R^{-4}+\tau^{-\frac{1}{2}},\label{eq:est3Psi}\\
\|\langle y\rangle^{-3}\chi_{R}(F+G)\|_{\infty}\lesssim &\tau^{-\frac{1}{2}},\label{eq:est3FG}\\
\|\langle y\rangle^{-3}\chi_{R}N_1\|_{\infty}\lesssim &\tilde\delta \kappa(\epsilon) R^{-4}+\tilde\delta R^{-4} (\mathcal{M}_2+\mathcal{M}_3)+R^{-5}(\mathcal{M}_2^2+\mathcal{M}_3^2), \label{eq:est3N1}\\
\|\langle y\rangle^{-3}\chi_{R}N_2\|_{\infty}\lesssim &\tau^{-\frac{1}{2}} +\tilde\delta R^{-4}\mathcal{M}_1.\label{eq:est3N2}
\end{align}

\end{proposition}
The proposition will be proved in Subsection \ref{subsec:weight3}. 

We continue to study \eqref{eq:3wEnd} by applying Proposition \ref{Prop:weight3} to obtain
\begin{align}
\begin{split}\label{eq:integ}
\|\langle y\rangle^{-3}\chi_{R}w(\cdot,\tau)\|_{\infty}
\lesssim &e^{-\frac{2}{5}(\tau-\tau_0)} \tau_0^{-\alpha} +\tau^{-\frac{1}{2}}+\tilde\delta R^{-4} \big(\kappa(\epsilon)+\sum_{k=1}^{3}\mathcal{M}_k\big)+R^{-5}\sum_{l=2,3}\mathcal{M}_l^2,
\end{split}
\end{align}
here we used several facts: (a) $\|\langle y\rangle^{-3}\chi_{R}w(\cdot,\tau_0)\|_{\infty}\leq \tau_0^{-\alpha}$ from \eqref{eq:IniWeighted}, (b) $\mathcal{M}_k,\ k=1,2,3,4,$ are increasing functions by their definitions, see \eqref{eq:defM}; (c) we claim that, for any $k>0,$ there exists a constant $C_k$ such that
\begin{align}
\int_{\tau_0}^{\tau}  e^{-\frac{2}{5}(\tau-s)} R^{-k}(s)\ ds \leq C_{k} R^{-k}(\tau);\label{eq:Lhos1}
\end{align} and lastly (d) we claim that
\begin{align}
\int_{\tau_0}^{\tau}  e^{-\frac{2}{5}(\tau-s)} s^{-\frac{1}{2}}\ ds \lesssim  \tau^{-\frac{1}{2}}.\label{eq:Lhos2}
\end{align}
\eqref{eq:Lhos1} and \eqref{eq:Lhos2} will be proved shortly.

\eqref{eq:integ} implies the desired \eqref{eq:estM1}. We choose a sufficiently large $\tau_0$ to make $\tau_0^{-\alpha}R^{4}(\tau_0)\leq 1$ and $\tau^{-\frac{1}{2}}\leq R^{-4}(\tau)$ for $\tau\geq \tau_0,$ which is achievable since $R(\tau)=\mathcal{O}(\sqrt{\ln \tau})$ grows much slower than $\tau^{\beta}$ for any $\beta>0.$  Thus
\begin{align}
\|\langle y\rangle^{-3}\chi_{R}w(\cdot,\tau)\|_{\infty}\leq R^{-4}(\tau)+\tilde\delta R^{-4} \Big(\kappa(\epsilon)+\sum_{k=1}^{3}\mathcal{M}_k\Big)+R^{-5}\sum_{l=2,3}\mathcal{M}_l^2.\label{eq:M1End}
\end{align}
This, together with the definition of $\mathcal{M}_1$ in \eqref{eq:defM}, implies the desired estimate \eqref{eq:estM1}.

To complete the proof we need to prove \eqref{eq:Lhos1} and \eqref{eq:Lhos2}.

The proof of \eqref{eq:Lhos1} can be simplified since the function $R^{-k}(\tau)$ is equivalent to the function $\min\Big\{R^{-k}(\tau_0), \big(\ln (2+\tau-\tau_0)\big)^{-k}\Big\}$ in the sense that, for some constant $C_k$, 
\begin{align}
\frac{1}{C_k}\leq \frac{R^{-k}(\tau)}{ \min\Big\{R^{-k}(\tau_0), \big(\ln (2+\tau-\tau_0)\big)^{-k}\Big\}}\leq C_k.\label{eq:comparable}
\end{align} It is easier to estimate the equivalent function: compute directly to find
\begin{align}
\int_{\tau_0}^{\tau} e^{-\frac{2}{5}(\tau-s)} R^{-k}(\tau_0)\ ds=R^{-k}(\tau_0)\int_{\tau_0}^{\tau} e^{-\frac{2}{5}(\tau-s)} ds \lesssim R^{-k}(\tau_0),
\end{align} and apply L$'$Hospital's rule to obtain, for some $A_k>0,$
\begin{align}
\int_{\tau_0}^{\tau} e^{-\frac{2}{5}(\tau-s)} \Big(\ln (2+s-\tau_0)\Big)^{-k} ds\leq A_k\Big(\ln (2+\tau-\tau_0)\Big)^{-k}.
\end{align}
We take the minimum of these estimates, and then use \eqref{eq:comparable} to obtain the desired \eqref{eq:Lhos1}. 

The estimate \eqref{eq:Lhos2} will be proved similarly, hence we skip the details here.

\subsection{Proof of Proposition \ref{Prop:weight3}}\label{subsec:weight3}

\begin{proof}
Compared to the subsequent subsections, the proof here is the most detailed so that we can focus on some other details later. Thus the readers are advised to read this part first.

To prove \eqref{eq:est3Psi}, we use that $\sum_{|k|+l\leq 2}|\partial_{\theta}^{l}\partial_{y}^{k}w|\leq \tilde\delta$ in \eqref{eq:estW} to find
\begin{align}
\begin{split}\label{eq:3Psiw}
\|\langle y\rangle^{-3}\Psi(w)\|_{\infty}& \lesssim  \tilde\delta\Big\{ \Big\|\langle y\rangle^{-3}(\frac{1}{2}y\cdot \nabla_{y} \chi_{R})\ (1-\tilde\chi_{R}) \Big\|_{\infty}+\Big\|\langle y\rangle^{-3}\partial_{\tau}\chi_{R}
\Big\|_{\infty}\\
&+\Big\|\langle y\rangle^{-3}\Delta_{y}\chi_{R}\Big\|_{\infty}+\Big\|\langle y\rangle^{-3}\nabla_{y}\chi_{R}\Big\|_{\infty}+\Big\|1_{|y|\leq (1+\epsilon)R} \Big(V_{a,B}^{-2}-\frac{1}{2}\Big)\Big\|_{\infty}\Big\}.
\end{split}
\end{align}

For the terms on the right hand side, we apply \eqref{eq:tails} to the first one to find
\begin{align}
\Big\|\langle y\rangle^{-3}(\frac{1}{2}y\cdot \nabla_y \chi_{R})(1-\tilde\chi_{R})\Big\|_{\infty}\leq R^{-5};
\end{align}
the second, third and fourth terms are supported by the set $|y|\geq R$ by \eqref{eq:defChi3}-\eqref{eq:reCutoff}, hence
\begin{align}
\sum_{|k|=1,2}\langle y\rangle^{-3}|\nabla_y^{k} \chi_{R}|\leq \kappa(\epsilon)R^{-3-|k|},
\end{align}
where, recall the definition of $\kappa(\epsilon)$ in \eqref{eq:defKappa}, and thus
\begin{align}
\Big\|\langle y\rangle^{-3}\partial_{\tau}\chi_{R}
\Big\|_{\infty}+\Big\|\langle y\rangle^{-3}\Delta_{y}\chi_{R}\Big\|_{\infty}+\Big\|\langle y\rangle^{-3}\nabla_{y}\chi_{R}\Big\|_{\infty}\lesssim \kappa(\epsilon) R^{-4};
\end{align}
and we control the last one by using \eqref{eq:paramePrelim}, 
\begin{align}\label{eq:difV2}
\Big|
V_{a,B}^{-2}-\frac{1}{2}\Big|\leq |B||y|^2+|a-\frac{1}{2}| \lesssim \tau^{-\frac{1}{2}}\ \text{when} \ |y|\leq (1+\epsilon)R=\mathcal{O}(\sqrt{\ln \tau}).
\end{align}

So far we have completed estimating all the terms on the right hand side of \eqref{eq:3Psiw}. What is left is to collect the estimates to complete the proof of \eqref{eq:est3Psi},
\begin{align}
\|\langle y\rangle^{-3}\Psi(w)\|_{\infty} 
\lesssim \tilde\delta\kappa(\epsilon) R^{-4}+\tau^{-\frac{1}{2}}.
\end{align}

It is easy to prove \eqref{eq:est3FG} by the estimates provided by \eqref{eq:paramePrelim} and \eqref{eq:TauAB}.

To prepare for proving \eqref{eq:est3N1}, we use that $\sum_{|k|+l=1,2}|\partial_{\theta}^{l}\nabla_{y}^{k}v|\lesssim \delta$ in \eqref{eq:condition1} to obtain
\begin{align}
\chi_{R}|N_1|\lesssim \chi_{R}|\nabla_{y}v|^2+\chi_{R}|\partial_{\theta}v|^2.\label{eq:zeroDN1}
\end{align}
For the first term on the right hand side, we decompose $v$ as in \eqref{eq:decomVToW} and then apply \eqref{eq:paramePrelim} to obtain, when $|y|\leq (1+\epsilon)R=\mathcal{O}(\sqrt{\ln \tau}),$
\begin{align*}
|\nabla_y v|^2\leq \Big(|\nabla_y V_{a,B}|+\sum_{k=1}^{3}|\Omega_k|+|\nabla_y w|\Big)^2\leq 2\tau^{-\frac{1}{2}}+2|\nabla_y w|^2.
\end{align*} 
Consequently
\begin{align}
\begin{split}
\Big\|\langle y\rangle^{-3}\chi_{R} (\nabla_{y}v)^2\Big\|_{\infty} \lesssim & \tau^{-\frac{1}{2}}+\Big\|\langle y\rangle^{-3} \chi_{R} |\nabla_y w|^2\Big\|_{\infty}.
\end{split}
\end{align} 
Then we change the orders of the operators $\chi_{R}$ and $\nabla_y$ to find
\begin{align*}
\begin{split}
&\Big\|\langle y\rangle^{-3} \chi_{R} |\nabla_y w|^2\Big\|_{\infty}\\
\leq &\Big\|\langle y\rangle^{-2}  \nabla_y \chi_{R} w\Big\|_{\infty}\Big\|\langle y\rangle^{-1} 1_{|y|\leq (1+\epsilon)R} \nabla_y w \Big\|_{\infty}+\Big\|\langle y\rangle^{-3}  (\nabla_y \chi_{R}) \Big\|_{\infty}\Big\|1_{|y|\leq(1+\epsilon)R}w\nabla_y w\Big\|_{\infty}\\
\leq &\|\langle y\rangle^{-2} \nabla_y\chi_{R}  w\|_{\infty}\Big\{\Big\|\langle y\rangle^{-1} \nabla_y \chi_{R} w \Big\|_{\infty}+\Big\|\langle y\rangle^{-1} 1_{|y|\leq (1+\epsilon)R}\nabla_y(1-\chi_{R})  w \Big\|_{\infty}\Big\}+\tilde\delta^2 \kappa(\epsilon) R^{-4}\\
\lesssim & \|\langle y\rangle^{-2}  \nabla_y\chi_{R} w\|_{\infty}\Big[R\|\langle y\rangle^{-2} \nabla_y\chi_{R}  w \|_{\infty}+\tilde\delta R^{-1}\Big]+\tilde\delta \kappa(\epsilon) R^{-4},
\end{split}
\end{align*}where we used that the functions $1_{|y|\leq(1+\epsilon)R}(1-\chi_{R})$ and $\nabla_y \chi_{R}$ are supported by the set $\Big\{y\ \Big|\ |y|\in [R,\ (1+\epsilon)R]\Big\},$ and thus  here $\langle y\rangle^{-1}\leq R^{-1}$; and that $|\nabla_y w|+ |w|\lesssim \tilde\delta$ in \eqref{eq:estW}; and $1_{|y|\leq (1+\epsilon)R}$ is the standard Heaviside function, see \eqref{eq:heavi}. 
Thus
\begin{align}\label{eq:zeroDN1Term1}
\|\langle y\rangle^{-3}\chi_{R} (\nabla_{y}v)^2\|_{\infty} \lesssim & \tilde\delta \kappa(\epsilon) R^{-4}+\tau^{-\frac{1}{2}}+\tilde\delta R^{-1}\|\langle y\rangle^{-2}  \nabla_y\chi_{R} w\|_{\infty} +R \|\langle y\rangle^{-2}  \nabla_y\chi_{R} w\|_{\infty}^2\nonumber\\
\lesssim&\tilde\delta \kappa(\epsilon) R^{-4}+\tau^{-\frac{1}{2}}+\tilde\delta R^{-4} \mathcal{M}_2+R^{-5} \mathcal{M}_2^2.
\end{align} 

This completes the treatment for the first term on the right hand side of \eqref{eq:zeroDN1}.

It is easier to estimate the second term of \eqref{eq:zeroDN1} since $\partial_{\theta}$ and $\chi_{R}$ commute,
\begin{align}\label{eq:zeroDN1Term2}
\|\langle y\rangle^{-3}\chi_{R} (\partial_{\theta}v)^2\|_{\infty} \lesssim 
\tilde\delta \kappa(\epsilon) R^{-4}+\tau^{-\frac{1}{2}}+\tilde\delta R^{-4} \mathcal{M}_3+R^{-5} \mathcal{M}_3^2.
\end{align}

\eqref{eq:zeroDN1Term1}, \eqref{eq:zeroDN1Term2} and \eqref{eq:zeroDN1} imply the desired \eqref{eq:est3N1}.

To prove \eqref{eq:est3N2}, the definition of $N_2(\eta)$ in \eqref{eq:defN2eta} implies
\begin{align}
\|\langle y\rangle^{-3}\chi_{R}N_2(\eta)\|_{\infty}\lesssim \|\langle y\rangle^{-3}\chi_{R}\eta^2\|_{\infty}+ \|\langle y\rangle^{-3}\chi_{R}\eta\partial_{\theta}^2 \eta\|_{\infty}
\lesssim \tau^{-\frac{1}{2}} +\tilde\delta \|\langle y\rangle^{-3}\chi_{R}w\|_{\infty},\label{eq:y3N2}
\end{align} where we used $|\eta|+ |\partial_{\theta}^2 \eta|\lesssim \tilde\delta$ implied by the decomposition $v=V_{a,B}+\eta$ and \eqref{eq:condition1}.
This, together with the definition of $\mathcal{M}_1$ in \eqref{eq:defM}, implies the desired result \eqref{eq:est3N2}.

\end{proof}


\section{Proof of \eqref{eq:estM2} }\label{sec:estM2}

 The treatment is similar to that in Section \ref{sec:estM1}, thus we will skip some details.

We derive a governing equation for $\nabla_y \chi_{R}w$ by taking a $y$-gradient on \eqref{eq:weightLinfw},
\begin{align}
\begin{split}\label{eq:ychiw}
\partial_{\tau}\nabla_y \chi_{R}w=&-(L_2+\frac{1}{2})\nabla_y \chi_{R}w+\nabla_y \chi_{R}(F+G+N_1+N_2)+ \tilde\Psi(w)
\end{split}
\end{align} where $\frac{1}{2}$ in the linear operator is from the commutation relation \eqref{eq:commt},
$\tilde\Psi(w)$ is defined as
\begin{align*}
\tilde\Psi(w):=\nabla_y \Psi(w)
+ \frac{1}{2}\Big(\nabla_y\frac{\tilde\chi_{R}\ y\nabla_{y} \chi_{R}  }{\chi_{R}}\Big) \chi_{R}w.
\end{align*}
Multiply both sides by $e^{-\frac{1}{8}|y|^2}$and find, recall the definition of $\mathcal{L}$ from (\ref{eq:selfadw}),
\begin{align}
\begin{split}\label{eq:ychiw2}
\partial_{\tau}e^{-\frac{1}{8}|y|^2}\nabla_y \chi_{R}w=&-(\mathcal{L}+\frac{1}{2})e^{-\frac{1}{8}|y|^2}\nabla_y \chi_{R}w\\
&+e^{-\frac{1}{8}|y|^2}\nabla_y \chi_{R}\Big(F+G+N_1+N_2\Big)+e^{-\frac{1}{8}|y|^2} \tilde\Psi(w).
\end{split}
\end{align}

(\ref{eq:orthow}) implies that $e^{-\frac{1}{8} |y|^2}\nabla_{y}\chi_{R}w$ enjoys some orthogonality conditions,
\begin{align}
e^{-\frac{1}{8} |y|^2}\nabla_{y}\chi_{R}w\perp e^{-\frac{1}{8} |y|^2}, \ e^{-\frac{1}{8} |y|^2} cos\theta,\ e^{-\frac{1}{8} |y|^2} sin\theta,\ e^{-\frac{1}{8} |y|^2} y_k, \ k=1,2,3.
\end{align} Denote, by $P_6$, the orthogonal projection onto the subspace orthogonal to these six functions.

Apply $P_{6}$ and then Duhamel's principle to rewrite \eqref{eq:ychiw2}
\begin{align}
\begin{split}\label{eq:transform3}
e^{-\frac{1}{8}|y|^2}\nabla_y \chi_{R}w(\tau) =&U_2(\tau,\tau_0) e^{-\frac{1}{8}|y|^2}\nabla_y \chi_{R}w(\tau_0)\\
+\int_{\tau_0}^{\tau} U_2(\tau,s)& P_{6}e^{-\frac{1}{8}|y|^2}\Big[\nabla_y \chi_{R}(F+G+N_1+N_2)+\tilde\Psi(w)\Big](s)\ ds,
\end{split}
\end{align} where $U_2(\tau,\sigma)$, $\tau\geq \sigma,$ is the propagator generated by the linear operator $-P_{6}(\mathcal{L}+\frac{1}{2})P_{6}$.

The propagator generates a decay rate.
\begin{lemma}\label{LM:propagator3}
For any function $g$ and times $\sigma_1$ and $\sigma_2$ with $\sigma_1\geq\sigma_2,$
\begin{align}
\Big\|\langle y\rangle^{-2} e^{\frac{1}{8}|y|^2} U_2(\sigma_1,\sigma_2) P_{6}g\Big\|_{\infty}\lesssim e^{-\frac{2}{5}(\sigma_1-\sigma_2)} \Big\|\langle y\rangle^{-2}e^{\frac{1}{8}|y|^2} g\Big\|_{\infty}.
\end{align}
\end{lemma}
The lemma will be proved in Section \ref{sec:propagator}.

Returning to \eqref{eq:transform3}, we apply Lemma \ref{LM:propagator3} to obtain
\begin{align}
\begin{split}\label{eq:appU2}
\Big\|\langle y\rangle^{-2}\nabla_y \chi_{R}w(\cdot,\tau)\Big\|_{\infty}\lesssim &e^{-\frac{2}{5}(\tau-\tau_0)}\Big\|\langle y\rangle^{-2}\nabla_y \chi_{R}w(\cdot,\tau_0)\Big\|_{\infty}\\
+\int_{\tau_0}^{\tau} e^{-\frac{2}{5}(\tau-s)}&\Big\|\langle y\rangle^{-2}\Big[\nabla_y \chi_{R}(F+G+N_1+N_2)+\tilde\Psi(w)\Big](s)\Big\|_{\infty}\ ds.
\end{split}
\end{align}

Next we estimate the terms on the right hand side. Since similar estimates will be needed for the terms in \eqref{eq:esty2Theta} below, the following results are more than the present need.
\begin{proposition}\label{Prop:2weiDy}
For $|k|+l=1$, the following estimates hold,
\begin{align}
\|\langle y\rangle^{-2}\partial_{\theta}^{l}\nabla_{y}^{k} \chi_{R}(F+G)\|_{\infty}\leq &\tau^{-\frac{1}{2}},\label{eq:2weiFG}\\
\|\langle y\rangle^{-2}\partial_{\theta}^{l}\nabla_{y}^{k} \chi_{R}N_1\|_{\infty}\lesssim &\tau^{-\frac{1}{2}}+\tilde\delta \kappa(\epsilon) R^{-3}+\tilde\delta R^{-3} [\mathcal{M}_2+\mathcal{M}_3],\label{eq:est2WeiN1}\\
\|\langle y\rangle^{-2}\partial_{\theta}^{l}\nabla_{y}^{k} \chi_{R}N_2\|_{\infty}\lesssim &\tau^{-\frac{1}{2}}+\tilde\delta R^{-3} [\mathcal{M}_1+\mathcal{M}_2],\label{eq:est2WeiN2}\\
\|\langle y\rangle^{-2}\tilde\Psi(w)\|_{\infty}\lesssim & \tilde\delta \kappa(\epsilon)R^{-3}+\tau^{-\frac{1}{2}}+R^{-\frac{13}{4}}\mathcal{M}_1.\label{eq:est2WeiPsi}
\end{align}
\end{proposition}
The proposition will be proved in subsection \ref{subsec:2weiDy}.

Suppose the proposition holds, then we prove the desired result \eqref{eq:estM2} as in \eqref{eq:integ}-\eqref{eq:M1End}. Hence we skip the details.

What is left is to prove Proposition \ref{Prop:2weiDy}. This will be done in the next subsection.

\subsection{Proof of Proposition \ref{Prop:2weiDy}}\label{subsec:2weiDy}
\begin{proof}
Each of \eqref{eq:2weiFG}-\eqref{eq:est2WeiN2} contains two estimates. Here we only consider the cases $(|k|,l)=(1,0)$, the other cases are easier since $\partial_{\theta}$ and $\chi_{R}$ commute.
Moreover since the treatment is similar to that in Subsection \ref{subsec:weight3}, here we will skip some details.

It is easy to prove \eqref{eq:2weiFG} by \eqref{eq:paramePrelim} and \eqref{eq:TauAB}.

To prepare for proving \eqref{eq:est2WeiN1}, we use $\sum_{|m|+n=1}^3|\nabla_{y}^{m}\partial_{\theta}^{n}v|\leq \delta$ from \eqref{eq:condition1} to obtain
\begin{align}
\begin{split}\label{eq:yn112}
|\nabla_{y}\chi_{R}N_{1}|=&|\nabla_y \chi_{R}|\ |N_{1}|+|\chi_{R} \nabla_y N_1|\lesssim \delta \Big[|\nabla_y \chi_{R}|+ \chi_{R} |\nabla_y v|+\chi_{R}|\partial_{\theta}v|\Big].
\end{split}
\end{align}
The first term is easy to control,  the definition $\chi_{R}(y):=\chi(\frac{y}{R})$ and that $|\nabla_y\chi_{R}(y)|=R^{-1}|\chi^{'}(\frac{y}{R})|$ is supported by the set $\Big\{y\ \Big|\ |y|\in [R,\ (1+\epsilon)R]\Big\}$ imply, 
\begin{align}
|\langle y\rangle^{-2} \nabla_y \chi_{R}|\leq \kappa(\epsilon)R^{-3},\label{eq:dychiR}
\end{align} where, recall that $\kappa(\epsilon)$ is defined in \eqref{eq:defKappa}.

For the second and third terms, the decomposition of $v$ in \eqref{eq:decomVToW} and \eqref{eq:paramePrelim} imply
\begin{align}\label{eq:SecThird}
\begin{split}
\Big\|\chi_{R}\langle y\rangle^{-2}\nabla_{y}v\Big\|_{\infty}+ &\Big\|\chi_{R}\langle y\rangle^{-2} \partial_{\theta}v\Big\|_{\infty}\\
\lesssim & \tau^{-\frac{3}{5}}+\tilde\delta \kappa(\epsilon) R^{-3}+\sum_{|k|+l=1}\Big\|\langle y\rangle^{-2}\nabla_y^k\partial_{\theta}^{l}\chi_{R} w\Big\|_{\infty}.
\end{split}
\end{align} Here $\tilde\delta \kappa(\epsilon) R^{-3}$ is used to bound $\langle y\rangle^{-2} (\nabla_{y}\chi_{R}) w$, produced in changing the orders of $\chi_{R}$ and $\nabla_y.$ 

So far we have estimated all the terms in \eqref{eq:yn112}. \eqref{eq:dychiR} and \eqref{eq:SecThird}. These, together with the definitions of $\mathcal{M}_2$ and $\mathcal{M}_3$ in \eqref{eq:defM}, imply the desired \eqref{eq:est2WeiN1}.

Next we prove \eqref{eq:est2WeiN2}. $|\partial_{\theta}^2 v|+ |\partial_{\theta}^2\nabla_y v|+ |\eta|\lesssim \tilde\delta$, from \eqref{eq:condition1}, implies
\begin{align}
\begin{split}
|\nabla_y \chi_{R}N_2|\leq &\tilde\delta\Big[ |\nabla_y (\chi_{R}\eta)| +\chi_{R}|\eta| \Big]
\lesssim \tilde\delta \Big[\tau^{-\frac{1}{2}}+|\nabla_y \chi_{R}w|+|\chi_{R}w|\Big].
\end{split}
\end{align}
This directly implies the desired result
\begin{align}
\Big\|\langle y\rangle^{-2}\nabla_y \chi_{3,R}N_2\Big\|_{\infty}\lesssim \delta \Big[\tau^{-\frac{1}{2}}+R^{-3}(\mathcal{M}_1+\mathcal{M}_2)\Big].
\end{align}

Lastly, the proof of \eqref{eq:est2WeiPsi} is similar that of \eqref{eq:est3Psi}. By the definition of $\tilde\Psi(w)$
\begin{align}\label{eq:ym2Psi}
\Big\|\langle y\rangle^{-2}\tilde\Psi(w)\Big\|_{\infty}\leq \|\langle y\rangle^{-2}\nabla_y\Psi(w)\|_{\infty}+\frac{1}{2}\sup_{y}\Big| \nabla_y \frac{\tilde\chi_{R} y\cdot \nabla_y\chi_{R}}{\chi_{R}}\Big|\ 
\Big\|\langle y\rangle^{-2}\chi_{R}w\Big\|_{\infty}.
\end{align}For the second term, the decay estimate $| \nabla_y \frac{\tilde\chi_{R} y\cdot \nabla_y\chi_{R}}{\chi_{R}}|\leq R^{-\frac{1}{4}}$ in \eqref{eq:deriveSmooth} and $\|\langle y\rangle^{-2}\chi_{R}w\|_{\infty}\leq 2 R \|\langle y\rangle^{-3}\chi_{R}w\|_{\infty},$ which is implied by that $\chi_{R}$ is supported by the set $|y|\leq 2R$, imply
\begin{align}
\sup_{y}\Big| \nabla_y \frac{\tilde\chi_{R} y\cdot \nabla_y\chi_{R}}{\chi_{R}}\Big|\ 
\|\langle y\rangle^{-2}\chi_{R}w\|_{\infty}\lesssim R^{-\frac{13}{4}} \mathcal{M}_1.
\end{align}
For the first term, the definition of $\Psi(w)$ in \eqref{eq:defPsiw} and that $\sum_{|k|+l\leq 3}|\nabla_y^{k}\partial_{\theta}^{l}w|\lesssim \tilde\delta$ and $|V_{a,B}^{-2}-\frac{1}{2}|+ |\nabla_y V_{a,B}|\leq \tau^{-\frac{1}{2}}$ derived from \eqref{eq:nablayV} and \eqref{eq:difV2}, imply
\begin{align*}
|\nabla_y \Psi(w)|\lesssim \tilde\delta\tau^{-\frac{1}{2}}+\tilde\delta\sum_{|k|=0,1}\bigg[\Big|\nabla_y^{k}\big[y\cdot \nabla_y \chi_{R} (1-\tilde{\chi}_{R})\big]\Big|+\Big|\nabla_y^{k} \partial_{\tau}\chi_{R}\Big|\bigg]+\tilde\delta\sum_{|k|=1}^{3}\Big|\nabla_y^{k} \chi_{R}\Big|.
\end{align*}
This, the definitions of $\chi_{R}(y)=\chi(\frac{|y|}{R})$ and $\tilde\chi_{R}$, and the properties in \eqref{eq:properties}, imply
\begin{align}
\|\langle y\rangle^{-2}\nabla_y \Psi(w)\|_{\infty}\lesssim \tilde\delta \kappa(\epsilon)R^{-3} +\tilde\delta \tau^{-\frac{1}{2}}.
\end{align}

Until now we estimated all the terms on the right hand side of \eqref{eq:ym2Psi}. What is left is to collect the estimate to obtain the desired results \eqref{eq:est2WeiPsi}.

\end{proof}
\section{Proof of \eqref{eq:estM3}}\label{sec:estM3}
Take a $\theta-$derivative on both sides of \eqref{eq:selfadw} to derive
\begin{align}
\begin{split}
\partial_{\tau} e^{-\frac{1}{8}|y|^2}\partial_{\theta}\chi_{R}w=&-\mathcal{L}e^{-\frac{1}{8}|y|^2}\partial_{\theta}\chi_{R}w+e^{-\frac{1}{8}|y|^2}\partial_{\theta}\Big[\chi_{R}(F+G+N_1+N_2)+
\Psi(w)\Big].
\end{split}
\end{align} 
The function $\partial_{\theta}\chi_{R}w$ enjoys the same orthogonality conditions as $\chi_{R}w$,
thus, by the definition of the orthogonal projection $P_{18}$ in \eqref{eq:defP21},
\begin{align}
P_{18} e^{-\frac{1}{8}|y|^2}\partial_{\theta}w\chi_{R}=e^{-\frac{1}{8}|y|^2}\partial_{\theta}w\chi_{R}.
\end{align} 
Apply Duhamel's principle to obtain
\begin{align}
\begin{split}\label{eq:18ThetaW}
e^{-\frac{1}{8}|y|^2}\partial_{\theta}\chi_{R}w(\cdot,\tau)=&U_{1}(\tau,\tau_0) e^{-\frac{1}{8}|y|^2}\partial_{\theta}\chi_{R}w(\cdot,\tau_0)\\
+\int_{\tau_0}^{\tau} U_{1}(\tau,s) & P_{18}e^{-\frac{1}{8}|y|^2}\partial_{\theta}\Big[\chi_{R}(F+G+N_1+N_2)+
\Psi(w)\Big](s)\ ds,
\end{split}
\end{align} where the propagator $U_1(\sigma_1,\sigma_2)$ is defined in \eqref{eq:3w}.

The propagator satisfies the following estimate.
\begin{lemma}\label{LM:propagator4}
For any function $g,$
\begin{align}
\Big\|\langle y\rangle^{-2} e^{\frac{1}{8}|y|^2} U_1(\tau,s)P_{18} \partial_{\theta}g \Big\|_{\infty}\lesssim e^{-\frac{2}{5}(\tau-s)} \Big\|\langle y\rangle^{-2} e^{\frac{1}{8}|y|^2} \partial_{\theta}g \Big\|_{\infty}
\end{align}
\end{lemma}
The lemma will be proved in Section \ref{sec:propagator}. 

Apply this to \eqref{eq:18ThetaW} to obtain
\begin{align}
\begin{split}
\|\langle y\rangle^{-2}\partial_{\theta}\chi_{R}w(\cdot,\tau)\|_{\infty}\lesssim & e^{-\frac{2}{5}(\tau-\tau_0)}\|\langle y\rangle^{-2}\partial_{\theta}\chi_{R}w(\cdot,\tau_0)\|_{\infty}\\
+\int_{\tau_0}^{\tau} e^{-\frac{2}{5}(\tau-s)}& \Big\|\langle y\rangle^{-2}\partial_{\theta}\Big[\chi_{R}(F+G+N_1+N_2)+
\Psi(w)\Big](s)\Big\|_{\infty}\ ds.\label{eq:esty2Theta}
\end{split}
\end{align}
$\Big\|\langle y\rangle^{-2}\partial_{\theta}\chi_{R}\Big(F+G+N_1+N_2\Big)\Big\|_{\infty}$ was estimated in Proposition \ref{Prop:2weiDy}. What is left is to control $\|\langle y\rangle^{-2}\partial_{\theta}\Psi(w)\|_{\infty}=\|\langle y\rangle^{-2}\Psi(\partial_{\theta}w)\|_{\infty}$. Similar to estimating  $\|\langle y\rangle^{-3}\Psi(w)\|_{\infty}$ in \eqref{eq:est3Psi},
\begin{align}
\Big\|\langle y\rangle^{-2}\partial_{\theta}\Psi(w)\Big\|_{\infty}\lesssim &\tilde\delta \kappa(\epsilon)R^{-3}+\tau^{-\frac{1}{2}}.
\end{align}

What is left is to prove the desired result \eqref{eq:estM3} by going through the same procedure as in \eqref{eq:integ}-\eqref{eq:M1End} in the previous section. Here we choose to skip the details.


\section{Proof of \eqref{eq:estM4} when $(|k|,l )=(0,2)$}\label{sec:estM402}
The treatment is only slightly different from that in Section \ref{sec:estM3}, except that we need to overcome a minor difficulty: in Section \ref{sec:estM3} we used the norm $\|\langle y\rangle^{-2}\cdot\|_{\infty}$, but here the adopted norm is $\|\langle y\rangle^{-1}\cdot\|_{\infty}$. To overcome this difficulty we have to make some transformation to make a different propagator estimate, specifically Lemma \ref{LM:U4} below, applicable.

We start with decomposing $\chi_{R}w$.
Define two functions $w_{\pm 1}$ as
\begin{align}
w_{\pm 1}(y,\tau):=\frac{1}{2\pi} \int_{0}^{2\pi} e^{\mp i\theta} \partial_{\theta}^2w(y,\theta,\tau) \ d\theta=-\frac{1}{2\pi} \int_{0}^{2\pi} e^{\mp i\theta} w(y,\theta,\tau) \ d\theta.
\end{align}
These two functions can be controlled by the previously estimated $\|\langle y\rangle^{-2}\chi_{R}\partial_{\theta}w(\cdot,\tau)\|_{\infty}$ since
\begin{align}
\|\langle y\rangle^{-1} \chi_{R}w_{\pm 1}\|_{\infty}\leq \|\langle y\rangle^{-1}\chi_{R}\partial_{\theta}w(\cdot,\tau)\|_{\infty}\leq 2 R\|\langle y\rangle^{-2}\chi_{R}\partial_{\theta}w(\cdot,\tau)\|_{\infty}.
\end{align} 

Now we need to estimate the remaining part 
$
\|\langle y\rangle^{-1} \Gamma (\chi_{R} \partial_{\theta}^2 w)\|_{\infty},
$ where $\Gamma$ is an orthogonal projection defined as, for any function $g$ of the form
$
g(y,\theta)=\displaystyle\sum_{n=-\infty}^{\infty} e^{i n\theta }g_{n}(y),
$
\begin{align}
\Gamma (  g):= \sum_{n\not= -1,0,1}  e^{in\theta} g_{n}=g-\frac{1}{2\pi}\sum_{n=-1,0,1}e^{inx}\langle g,\ e^{inx}\rangle_{\theta}.\label{eq:defGamma}
\end{align}

To derive a governing equation for $\Gamma(\partial_{\theta}^2\chi_{R}w)$, we take two $\theta-$derivatives and then apply the operator $\Gamma$ on both sides of \eqref{eq:selfadw},
\begin{align*}
\partial_{\tau} \Gamma(e^{-\frac{1}{8}|y|^2}\partial_{\theta}^2\chi_{R}w)=&-\mathcal{L}\Gamma(e^{-\frac{1}{8}|y|^2}\partial_{\theta}^2\chi_{R}w)+\Gamma\Big(e^{-\frac{1}{8}|y|^2}\partial_{\theta}^2\big[\chi_{R}(F+G+N_1+N_2)+
\Psi(w)\big]\Big).
\end{align*} 
Apply Duhamel's principle to obtain,
\begin{align}
\begin{split}\label{eq:ThetaTwoDeriva}
e^{-\frac{1}{8}|y|^2}\Gamma\big(\chi_{R} \partial_{\theta}^2 w(\cdot,\tau)\big)=&U_3(\tau,\tau_0)e^{-\frac{1}{8}|y|^2}\Gamma \big(\chi_{R} \partial_{\theta}^2 w(\cdot,\tau_0)\big)\\
+\int_{\tau_0}^{\tau} U_3(\tau,\sigma) & e^{-\frac{1}{8}|y|^2}\Gamma \Big(\partial_{\theta}^2\Big[\chi_{R}(F+G+N_1+N_2)+
\Psi(w)\Big](\sigma)\Big)\ d\sigma,
\end{split}
\end{align} where $U_3(\tau,\sigma)$ is the propagator generated by $-\mathcal{L}$ from $\sigma$ to $\tau.$

The propagator generates decay rate:
\begin{lemma}\label{LM:U4}
For any function $g$,
\begin{align}
\|\langle y\rangle^{-1} e^{\frac{1}{8}|y|^2}U_3(\tau,\sigma) \Gamma(g)\|_{\infty}\lesssim e^{-\frac{2}{5}(\tau-\sigma)} \|\langle y\rangle^{-1}e^{\frac{1}{8}|y|^2} g\|_{\infty}.\label{eq:remove3}
\end{align}

\end{lemma}
The lemma will be proved in Appendix \ref{sec:propagator}.

Apply this to \eqref{eq:ThetaTwoDeriva},
\begin{align}
\begin{split}\label{eq:theta2w3}
\Big\|\langle y\rangle^{-1}\Gamma (\chi_{R} \partial_{\theta}^2 w(\cdot,\tau))\Big\|_{\infty}
&\lesssim  e^{-\frac{2}{5}(\tau-\tau_0)}\Big\|\langle y\rangle^{-1}\chi_{R} \partial_{\theta}^2 w(\cdot,\tau_0)\Big\|_{\infty}\\
+\int_{\tau_0}^{\tau} e^{-\frac{2}{5}(\tau-s)}&\ \Big\|\langle y\rangle^{-1}\partial_{\theta}^2\Big[\chi_{R}(F+G+N_1+N_2)+
\Psi(w)\Big](\sigma)\Big\|_{\infty} \ d\sigma.
\end{split}
\end{align} 

Next we estimate the terms on the right hand side. To avoid controlling similar terms in \eqref{eq:TwoYder} below, we will include some estimates for them.
\begin{proposition}\label{prop:twoderiveLinf}
For any $k\in (\mathbb{N}\cup \{0\})^3$ and $l\in \mathbb{N}\cup \{0\}$ satisfying $|k|+l=2,$ we have
\begin{align}
\|\langle y\rangle^{-1}\partial_{\theta}^{l}\nabla_{y}^{k} \chi_{R}(F+G)\|_{\infty}\leq &\tau^{-\frac{1}{2}},\label{eq:2deriv1wFG}\\
\|\langle y\rangle^{-1}\partial_{\theta}^{l}\nabla_{y}^{k} \chi_{R}N_1\|_{\infty}\lesssim &\tilde\delta \kappa(\epsilon)R^{-2}+\tau^{-\frac{1}{2}}+\tilde\delta R^{-2} \big(\mathcal{M}_2+\mathcal{M}_3+\mathcal{M}_4\big),\label{eq:2deriv1wN1}\\
\|\langle y\rangle^{-1}\partial_{\theta}^{l}\nabla_{y}^{k} \chi_{R}N_2\|_{\infty}\lesssim &\tilde\delta R^{-2} \big(\mathcal{M}_1+\mathcal{M}_2+\mathcal{M}_3+\mathcal{M}_4\big), \label{eq:2deriv1wN2}\\
\|\langle y\rangle^{-1}\partial_{\theta}^{2}\Psi(w)\|_{\infty}\lesssim &\tilde\delta \kappa(\epsilon)R^{-2}+\tau^{-\frac{1}{2}}.\label{eq:2deriv1wPsiw}
\end{align}
\end{proposition}
The proposition will be proved in Subsection \ref{subsec:twoderiveLinf}.

Returning to \eqref{eq:theta2w3}, we go through the same procedure as in \eqref{eq:integ}-\eqref{eq:M1End} to obtain
\begin{align}
\Big\|\langle y\rangle^{-1}\Gamma (\chi_{R} \partial_{\theta}^2 w)\Big\|_{\infty}\lesssim R^{-2}\Big\{1+\tilde\delta \kappa(\epsilon)+\tilde\delta  \sum_{k=1}^{4} \mathcal{M}_k\Big\}.
\end{align}This will lead to the desired estimate for $\|\langle y\rangle^{-1}\chi_{R} \partial_{\theta}^2 w(\cdot,\tau)\|_{\infty}$:
By the definition of $\Gamma$,
\begin{align}
\begin{split}
\|\langle y\rangle^{-1}\chi_{R} \partial_{\theta}^2 w\|_{\infty}\leq &\|\langle y\rangle^{-1}\Gamma (\chi_{R} \partial_{\theta}^2 w)\|_{\infty}+\frac{1}{\pi} \|\langle y\rangle^{-1}\partial_{\theta}\chi_{R}  w\|_{\infty}\\
\leq &\|\langle y\rangle^{-1}\Gamma (\chi_{R} \partial_{\theta}^2 w)\|_{\infty}+R \|\langle y\rangle^{-2}\partial_{\theta}\chi_{R}  w\|_{\infty}\\
\lesssim & R^{-2}\Big\{1+\tilde\delta \kappa(\epsilon)+\mathcal{M}_3+\tilde\delta  \sum_{k=1}^{4} \mathcal{M}_k\Big\}.
\end{split}
\end{align} 

\subsubsection{Proof of Proposition \ref{prop:twoderiveLinf}}\label{subsec:twoderiveLinf}
\begin{proof}
Since all the techniques have been used and explained in detail in the previous subsections, here we only sketch the proof.

It is easy to prove \eqref{eq:2deriv1wFG} by applying \eqref{eq:paramePrelim} and \eqref{eq:TauAB}.

For \eqref{eq:2deriv1wN1}, compute directly to find that, for $|k|+l=2,$
\begin{align}
\begin{split}
|\chi_{R}\nabla_{y}^{k}\partial_{\theta}^l N_1|\lesssim &\delta \chi_{R}\sum_{|m|+n=1,2}\Big|\nabla_y^m\partial_{\theta}^{n}v\Big|\\
\leq & \delta \sum_{|m|+n=1,2}\Big\{\big|\nabla_y^m\partial_{\theta}^{n}\chi_{R}v\big|+  \big|(\nabla_y^m\partial_{\theta}^{n}\chi_{R}-\chi_{R}\nabla_y^m\partial_{\theta}^{n})v\big|\Big\},
\end{split}
\end{align}where we use \eqref{eq:condition1}; in the second step, change the orders of $\chi_{R}$ and $\nabla_y^m\partial_{\theta}^{n}.$

This, together with the decomposition of $v$ in \eqref{eq:decomVToW} and that $\|\langle y\rangle^{-1}\partial_{\theta}^{l}
\nabla_{y}^{k}\chi_{R}w \|_{\infty}\leq 2 R \|\langle y\rangle^{-2}\partial_{\theta}^{l}
\nabla_{y}^{k}\chi_{R}w \|_{\infty}$, implies the desired result.

The proof of \eqref{eq:2deriv1wN2} is similar, hence is skipped.

For \eqref{eq:2deriv1wPsiw}, the definition of $\Psi$ implies $\partial_{\theta}^2\Psi(w)=\Psi(\partial_{\theta}^2 w).$ What is left is very similar to estimating $\|\langle y\rangle^{-3}\Psi(w)\|_{\infty}$ in \eqref{eq:est3Psi}, hence we omit the details.

\end{proof}

\section{Proof of \eqref{eq:estM4} when $(|k|,l )=(2,0)$}\label{sec:estM420}

We start with estimating $\partial_{y_1}\partial_{y_2}\chi_{R}w$.

Similar to deriving the governing equation for $\nabla_y \chi_{R}w$ in \eqref{eq:ychiw},
\begin{align}
\begin{split}\label{eq:yychiw}
\partial_{\tau}e^{-\frac{1}{8}|y|^2}\partial_{y_1}\partial_{y_2} \chi_{R}w=&-(\mathcal{L}+1)e^{-\frac{1}{8}|y|^2}\partial_{y_1}\partial_{y_2} \chi_{R}w\\
&+e^{-\frac{1}{8}|y|^2}\partial_{y_1}\partial_{y_2} \chi_{R}\Big(F+G+N_1+N_2\Big)+e^{-\frac{1}{8}|y|^2} \Phi(w)
\end{split}
\end{align} where the term $\Phi(w)$ is defined as,
\begin{align*}
\Phi(w):= \partial_{y_1}\partial_{y_2}\Psi(w)+\frac{1}{2}\partial_{y_1}\Big[\Big(\partial_{y_2}\frac{\tilde\chi_{R}\ y\nabla_{y} \chi_{R}  }{\chi_{R}}\Big) \chi_{R}w\Big]+\frac{1}{2}\Big(\partial_{y_2}\frac{\tilde\chi_{R}\ y\nabla_{y} \chi_{R}  }{\chi_{R}}\Big)(\partial_{y_1}\chi_{R}w), 
\end{align*} and $\Psi(w)$ is defined in \eqref{eq:defPsiw}.

The orthogonality conditions imposed on $e^{-\frac{1}{8}|y|^2} \chi_{R}w$ in \eqref{eq:decomVToW} imply that
\begin{align}
e^{-\frac{1}{8}|y|^2}\partial_{y_1}\partial_{y_2} \chi_{R}w\perp e^{-\frac{1}{8}|y|^2}.
\end{align}We denote by $P_1$ the orthogonal projection onto the subspace orthogonal to $e^{-\frac{1}{8}|y|^2}$.

Apply $P_1$ on both sides of \eqref{eq:yychiw} and then use Duhamel's principle to find
\begin{align}
\begin{split}\label{eq:yyDur}
e^{-\frac{1}{8}|y|^2}\partial_{y_1}\partial_{y_2} &\chi_{R}w(\cdot,\tau)=U_4(\tau,\tau_0)e^{-\frac{1}{8}|y|^2}\partial_{y_1}\partial_{y_2} \chi_{R}w(\cdot,\tau_0)\\
&+\int_{\tau_0}^{\tau} U_4(\tau,\sigma) P_1e^{-\frac{1}{8}|y|^2}\Big[\partial_{y_1}\partial_{y_2} \chi_{R}\Big(F+G+N_1+N_2\Big)+ \Phi(w)\Big](\sigma)\ d\sigma,
\end{split}
\end{align}
where $U_4(\tau,\sigma)$ is the propagator generated by the linear operator $-P_1(\mathcal{L}+1)P_1$ from $\sigma$ to $\tau.$

The propagator satisfies the following estimate,
\begin{lemma}\label{LM:U5}
For any function $g$, and times $\tau, \ \sigma$ with $\tau\geq \sigma,$
\begin{align}\label{eq:estU4}
\Big\|\langle y\rangle^{-1}e^{\frac{1}{8}|y|^2}U_4(\tau,\sigma)P_1g\Big \|_{\infty}\lesssim e^{-\frac{2}{5}(\tau-\sigma)} \Big\|\langle y\rangle^{-1}e^{\frac{1}{8}|y|^2}g\Big\|_{\infty}.
\end{align}
\end{lemma}
This lemma will be proved in Section \ref{sec:propagator}.

Apply the propagator estimate to \eqref{eq:yyDur} to obtain
\begin{align}
\begin{split}\label{eq:TwoYder}
\Big\|\langle y\rangle^{-1}\partial_{y_1}\partial_{y_2}& \chi_{R}w(\cdot,\tau)\Big\|_{\infty}
\lesssim  e^{-\frac{2}{5}(\tau-\tau_0)}\Big\|\langle y\rangle^{-1}\partial_{y_1}\partial_{y_2} \chi_{R}w(\cdot,\tau_0)\Big\|_{\infty}\\
&+\int_{\tau_0}^{\tau}e^{-\frac{2}{5}(\tau-\sigma)} \Big\|\langle y\rangle^{-1}\Big[\partial_{y_1}\partial_{y_2} \chi_{R}\Big(F+G+N_1+N_2\Big)+ \Phi(w)\Big](\sigma) \Big\|_{\infty} \ d\sigma.
\end{split}
\end{align}
Here all the terms, except $\|\langle y\rangle^{-1} \Phi(w) \|_{\infty} $, were treated in Proposition \ref{prop:twoderiveLinf}. 

To control $\|\langle y\rangle^{-1} \Phi(w) \|_{\infty} $, we use that $\sum_{|k|=1,2}\Big\|\nabla_{y}^{k}\frac{\tilde\chi_{R}\ y\cdot\nabla_{y} \chi_{R}  }{\chi_{R}} \Big\|_{\infty}\leq R^{-\frac{1}{4}}$ in \eqref{eq:deriveSmooth} to obtain
\begin{align}
\begin{split}
\|\langle y\rangle^{-1} \Phi(w) \|_{\infty}\lesssim& \|\langle y\rangle^{-1}\partial_{y_1}\partial_{y_2}\Psi(w)\|_{\infty}
+R^{-\frac{1}{4}} \sum_{|k|=0,1}\|\langle y\rangle^{-1}\nabla_{y}^k\chi_{R}w\|_{\infty}\\
\lesssim &\tilde\delta \kappa(\epsilon) R^{-2}+\tau^{-\frac{1}{2}}+R^2  \|\langle y\rangle^{-3}\chi_{R}w\|_{\infty}+R\|\langle y\rangle^{-2}\nabla_{y}\chi_{R}w\|_{\infty}\\
\leq &\tilde\delta \kappa(\epsilon) R^{-2}+\tau^{-\frac{1}{2}}+R^{-\frac{5}{4}}\big(\mathcal{M}_1+\mathcal{M}_2\big),
\end{split}
\end{align}
where in the second step we used the techniques of proving \eqref{eq:est3Psi} to control the $\Psi-$term.

What is left is to go through the same procedure as in \eqref{eq:integ}-\eqref{eq:M1End} to obtain a satisfactory estimate for $\|\langle y\rangle^{-1}\partial_{y_1}\partial_{y_2}\chi_{R}w(\cdot,\tau)\|_{\infty}:$
\begin{align}
\Big\|\langle y\rangle^{-1}\partial_{y_1}\partial_{y_2}\chi_{R}w(\cdot,\tau)\Big\|_{\infty}\lesssim  R^{-2}\Big\{1+\tilde\delta \kappa(\epsilon)+\tilde\delta  \sum_{k=1}^{4} \mathcal{M}_k\Big\}.
\end{align}
For the other terms in $\nabla_y^k \chi_R w,$ $|k|=2,$ we use the same method to obtain the desired result
\begin{align}
\sum_{|k|=2}\Big\|\langle y\rangle^{-1}\nabla_{y}^k\chi_{R}w(\cdot,\tau)\Big\|_{\infty}\lesssim  R^{-2}\Big\{1+\tilde\delta \kappa(\epsilon)+\tilde\delta  \sum_{k=1}^{4} \mathcal{M}_k\Big\}.
\end{align}


\appendix


\section{Proof of Lemmata \ref{LM:propagator}, \ref{LM:propagator3}, \ref{LM:propagator4}, \ref{LM:U4} and \ref{LM:U5}}\label{sec:propagator}
The present problem is very similar to those in \cite{BrKu} and \cite{DGSW} for the blowup problem of one-dimensional nonlinear heat equation, \cite{GS2008, GaKnSi, GaKn20142} for MCF, and \cite{MultiDHeat} for multidimensional nonlinear heat equations. In those works the linear operator takes the form,
\begin{align}
\mathcal{L}_W=-\Delta+\frac{1}{16}|y|^2-\frac{n}{4}-1-\frac{1}{2}\partial_{\theta}^2+W,
\end{align}where $n$ is the dimension, the multiplier $W$ is defined as
\begin{align*}
W:=\frac{y^{T}B(\tau)y}{2(n-1)+y^{T}B(\tau)y},
\end{align*}and $B(\tau)=\mathcal{O}(\frac{1}{\tau})$ is a positive definite real $n\times n$ matrix.

Here we consider the case $n=3$, and the linear operator takes the form
\begin{align}
\mathcal{L}_V=-\Delta+\frac{1}{16}|y|^2-\frac{3}{4}-1-\frac{1}{2}\partial_{\theta}^2+V,
\end{align} and $V$ is defined as
\begin{align*}
V:=\frac{1}{2}\Big|\frac{\tilde\chi_{R}\ y\cdot \nabla_{y} \chi_{R}  }{\chi_{R}}\Big|.\label{eq:potential}
\end{align*} 

The proofs of the known cases and the present are very similar, since the multipliers $V$ and $W$ are both (favorably) nonnegative, and (favorably) slowly varying in the $y$-variable. 

Before the proof we simplify the problem.

Suppose that $e^{-\frac{1}{8}|y|^2} g $ is orthogonal to the eigenvectors of the linear operator $\mathcal{L}_0:=-\Delta+\frac{1}{16}|y|^2-\frac{3}{4}-\frac{1}{2}\partial_{\theta}^2$ with eigenvalues $0, \ \frac{1}{2},\ 1$:
\begin{itemize}
\item[(1)] for eigenvalue $0$: $e^{-\frac{1}{8} |y|^2}$,
\item[(2)] for eigenvalue $\frac{1}{2}$: $\\ e^{-\frac{1}{8} |y|^2} y_k, \ 
 e^{-\frac{1}{8} |y|^2} cos\theta,\ e^{-\frac{1}{8} |y|^2} sin\theta, \ k=1,2,3,$
\item[(3)] for eigenvalue $1:$ $e^{-\frac{1}{8} |y|^2}(\frac{1}{2}y_k^2-1),\ e^{-\frac{1}{8} |y|^2} y_k cos\theta,\  e^{-\frac{1}{8} |y|^2} y_k sin\theta ,\ e^{-\frac{1}{8} |y|^2} y_m y_n, \ m\not=n, \ \ k, m,n=1,2,3.$
\end{itemize} 

Since $\mathcal{L}_0$, mapping the $L^2-$space into itself, is self-adjoint, the spectral analysis above implies that $\|e^{-\mathcal{L}_0 \tau}e^{-\frac{1}{8}|y|^2}g\|_2\leq e^{-\frac{3}{2}\tau}\|e^{-\frac{1}{8}|y|^2}g\|_2$. The present situation is different since we want a decay rate for the propagator generated by the time dependent $-\mathcal{L}_V$ in weighted $L^{\infty}-$norms.

To derive a convenient form we Fourier-expand the function $g$ in the $\theta-$variable to find
\begin{align}
 g(y,\theta)=\sum_{n=-\infty}^{\infty}  e^{in \theta}g_{n}(y)
\end{align} with the functions $g_{n}$ defined as
\begin{align}
g_{n}(y):=\frac{1}{2\pi} \Big\langle g(y,\cdot),\ e^{in\theta}\Big\rangle_{\theta}.\label{eq:defGn}
\end{align}

The orthogonality conditions imposed on $g$ imply that $e^{-\frac{1}{8}|y|^2}g_k,\ k=-1,0,1$, enjoy certain orthogonality conditions. Specifically $e^{-\frac{1}{8}|y|^2}g_0$ is orthogonal to the following $10$ eigenvectors of the operator $-\Delta+\frac{1}{16}|y|^2-\frac{3}{4}$,
\begin{align}
\begin{split}\label{eq:13Fun}
 e^{-\frac{1}{8}|y|^2},\ & y_{k}e^{-\frac{1}{8}|y|^2}, \ e^{-\frac{1}{8} |y|^2}(\frac{1}{2}y_k^2-1),\ 
e^{-\frac{1}{8} |y|^2} y_m y_n, \ m\not=n, \ \ k, m,n=1,2,3,
\end{split}
\end{align}
and $e^{-\frac{1}{8}|y|^2}g_{k}$, $k=\pm 1$, are orthogonal to the following eigenvectors,
\begin{align}
 e^{-\frac{1}{8}|y|^2},\ y_{k}e^{-\frac{1}{8}|y|^2},\ k=1,2,3.\label{eq:4Fun}
\end{align}

Accordingly, we define $\tilde{P}_{10}$ and $\tilde{P}_4$ to be the orthogonal projections onto the subspace orthogonal to these 10 and 4 functions listed in \eqref{eq:13Fun} and \eqref{eq:4Fun} respectively. Thus
\begin{align*}
\tilde{P}_{10}e^{-\frac{1}{8}|y|^2}g_{0}=&e^{-\frac{1}{8}|y|^2}g_{0},\\
\tilde{P}_{4}e^{-\frac{1}{8}|y|^2}g_{\pm 1}=&e^{-\frac{1}{8}|y|^2}g_{\pm 1}.
\end{align*}

Consequently, the propagator $U_1(\tau,\sigma)$ in Lemma \ref{LM:propagator} takes a new form,
\begin{align}
\begin{split}\label{eq:propDecom}
U_{1}(\tau,\sigma)P_{18} g=& e^{(\tau-\sigma)}\tilde{U}_1(\tau,\sigma)\tilde{P}_{10}g_0(y)+e^{\frac{1}{2}(\tau-\sigma)}\sum_{n=\pm 1}e^{in\theta} \tilde{U}_{2}(\tau,\sigma) \tilde{P}_4 g_n(y)\\
&+\sum_{|n|\geq 2}  e^{in \theta}e^{-\frac{n^2-2}{2}(\tau-\sigma)}\tilde{U}_3(\tau,\sigma)g_n(y),
\end{split}
\end{align}where $\tilde{U}_1(\tau,\sigma)$ is generated by the linear operator $-\tilde{P}_{10}\mathcal{L}_1 \tilde{P}_{10}$, and $\tilde{U}_{2}(\tau,\sigma)$ is generated by $-\tilde{P}_{4}\mathcal{L}_1 \tilde{P}_{4},$ and $\tilde{U}_3$ is generated by $-\mathcal{L}_1$,  which is defined as
\begin{align}
\mathcal{L}_1 :=&-\Delta+\frac{1}{16}|y|^2-\frac{3}{4}+V.
\end{align} Similarly, the propagators $U_k$, $k=1,2,3,4,$ in Lemmas \ref{LM:propagator4}, \ref{LM:propagator3}, \ref{LM:U4} and \ref{LM:U5}
become
\begin{align}
\begin{split}
U_1(\tau,\sigma)P_{18} \partial_{\theta}g=&e^{\frac{1}{2}(\tau-\sigma)}\sum_{n=\pm 1}ine^{in\theta} \tilde{U}_{2}(\tau,\sigma) \tilde{P}_4 g_n(y)\\
&+\sum_{|n|\geq 2} i n  e^{in \theta}e^{-\frac{n^2-2}{2}(\tau-\sigma)}\tilde{U}_3(\tau,\sigma)g_n(y),\\
U_2(\tau,\sigma) P_6 g =&e^{\frac{1}{2}(\tau-\sigma)} \tilde{U}_2(\tau,\sigma) \tilde{P}_4 g_0+\sum_{n=\pm 1}e^{in\theta} \tilde{U}_{4}(\tau,\sigma) \tilde{P}_1 g_n(y)\nonumber\\
&\hspace{1 cm}+\sum_{|n|\geq 2}  e^{in \theta}e^{-\frac{n^2-1}{2}(\tau-\sigma)}\tilde{U}_3(\tau,\sigma)g_n(y),\\
U_3(\tau,\sigma)\Gamma(g)=&\sum_{|n|\geq 2}  e^{in \theta}e^{-\frac{n^2-2}{2}(\tau-\sigma)}\tilde{U}_3(\tau,\sigma)g_n(y),\\
U_4(\tau,\sigma) P_1g=&\tilde{U}_4(\tau,\sigma)\tilde{P}_1 g_0+\sum_{n\not= 0}  e^{in \theta}e^{-\frac{n^2}{2}(\tau-\sigma)}\tilde{U}_3(\tau,\sigma)g_n(y)
\end{split}
\end{align}where, the projection $\Gamma$ is defined in \eqref{eq:defGamma}, all the other operators, except $\tilde{U}_4$ and $\tilde{P}_1$, have been defined in \eqref{eq:propDecom}, and $\tilde{P}_1$ is the orthogonal projection into the subspace orthogonal to $e^{-\frac{1}{8}|y|^2}$and $\tilde{U}_4(\tau,\sigma)$ is the propagator generated by the linear operator $\tilde{P}_1\mathcal{L}_1\tilde{P}_1.$

These propagators satisfies the following estimates:
\begin{theorem}\label{THM:frequencyWise}
There exists a constant $C$ such that, for any function $g: \ \mathbb{R}^3\rightarrow \mathbb{C}$ and $\tau\geq \sigma,$
\begin{align}
\|\langle y\rangle^{-3}e^{\frac{1}{8}|y|^2}\tilde{U}_1(\tau,\sigma)\tilde{P}_{10} g\|_{\infty}\leq &C e^{-\frac{7}{5}(\tau-\sigma)}\|\langle y\rangle^{-3}e^{\frac{1}{8}|y|^2} g\|_{\infty},\label{eq:13Ortho}\\
\|\langle y\rangle^{-k}e^{\frac{1}{8}|y|^2}\tilde{U}_2(\tau,\sigma)\tilde{P}_{4} g\|_{\infty}\leq &C e^{-\frac{9}{10}(\tau-\sigma)}\|\langle y\rangle^{-k}e^{\frac{1}{8}|y|^2} g\|_{\infty},\ k=2,3, \label{eq:4Ortho}\\
\|\langle y\rangle^{-1}e^{\frac{1}{8}|y|^2}\tilde{U}_4(\tau,\sigma)\tilde{P}_{1} g\|_{\infty}\leq &Ce^{-\frac{2}{5}(\tau-\sigma)}\|\langle y\rangle^{-1}e^{\frac{1}{8}|y|^2} g\|_{\infty},\label{eq:1Ortho}\\
\|\langle y\rangle^{-k}e^{\frac{1}{8}|y|^2}\tilde{U}_3(\tau,\sigma) g\|_{\infty}\leq &C \|\langle y\rangle^{-k}e^{\frac{1}{8}|y|^2} g\|_{\infty},\ k=0,1,2,3.\label{eq:noOrtho}
\end{align}

\end{theorem}
The theorem will be proved shortly. 

Assuming Theorem \ref{THM:frequencyWise} holds, we are ready to prove the desired results. These estimates, \eqref{eq:propDecom} and that $\sum_{n\geq 2}e^{-\frac{n^2-2}{2}s}\leq e^{-\frac{s}{2}}$ when $s\geq 1$ imply Lemma \ref{LM:propagator} when  $\tau-\sigma\geq 1$. When $\tau-\sigma<1$, we prove the desired result by applying the maximum principle. This will require a different formulation and it is easy, hence we choose to skip the proof. The proofs of the other lemmas are easier since we only need some of \eqref{eq:13Ortho}-\eqref{eq:noOrtho}, thus are skipped.

In the rest of the section we prove Theorem \ref{THM:frequencyWise}.
\begin{proof}
We will only prove \eqref{eq:1Ortho} and \eqref{eq:noOrtho}, and will skip the proof of the others since the key ideas are the same. For more details we refer to \cite{MultiDHeat}.

We start with proving \eqref{eq:noOrtho} since the techniques will be used in proving \eqref{eq:1Ortho}.

The crucial step is to derive an integral kernel. By the technique of path integral technique, see \cite{BrKu, DGSW}, we find that, for any function $g$,
\begin{align}
\tilde{U}_{3}(\tau,\sigma)g(y)=\int_{\mathbb{R}^3}K_{\tau-\sigma}(y,z) \langle e^{-V}\rangle(y,z)e^{-\frac{1}{8}|z|^2} g(z)\ dz.\label{eq:IntKernel}
\end{align} where the function $\langle e^{-V}\rangle(y,z)$ is defined in terms of path integral,
\begin{align*}
\langle e^{-V}\rangle(y,z):=\int e^{-\int_{\sigma}^{\tau} V(\omega_0(s)+\omega(s),s)\ ds} d\mu(\omega),
\end{align*} and $K_{\tau-\sigma}(y,z)e^{-\frac{1}{8}|z|^2}$ is the integral kernel of $e^{-(\tau-\sigma)\mathcal{L}_0}$, given by Mehler's formula,
\begin{align}\label{eq:meh}
K_{\tau-\sigma}(y,z):=(2\sqrt{2}\pi )^3 (1-e^{-(\tau-\sigma)})^{-\frac{3}{2}}e^{\frac{1}{8}|y|^2 }e^{-\frac{|y-e^{-\frac{\tau-\sigma}{2}}z|^2}{4(1-e^{-(\tau-\sigma)})}},
\end{align}
and $d\mu(\omega)$ is a probability measure on the continuous paths $\omega:[\sigma,\tau]\rightarrow \mathbb{R}^{3}$ with $\omega(\sigma)=\omega(\tau)=0$, and $\omega_0(s)$ is a path, with $\omega_0(\sigma)=z$ and $\omega_0(\tau)=y,$ defined as
\begin{align*}
\omega_0(s)=e^{\frac{1}{2}(\tau-s)} \frac{e^{\sigma}-e^{s}}{e^{\sigma}-e^{\tau}}y+e^{\frac{1}{2}(\sigma-s)} \frac{e^{\tau}-e^{s}}{e^{\tau}-e^{\sigma}}z.
\end{align*} Here $\mathcal{L}_0$ is a linear operator defined as
\begin{align}
\mathcal{L}_0:=-\Delta+\frac{1}{16}|y|^2-\frac{3}{4}.
\end{align}
Since $V$ is nonnegative, $0\leq \langle e^{-V}\rangle\leq 1$, and thus
\begin{align}
|\tilde{U}_{3}(\tau,\sigma)g|\leq e^{-(\tau-\sigma)\mathcal{L}_0}|g|.
\end{align} This implies the desired \eqref{eq:noOrtho}, see e.g. \cite{DGSW}, for $k=0,1,2,3,$
\begin{align}
\Big\|\langle y\rangle^{-k} e^{\frac{1}{8}|y|^2}\tilde{U}_{3}(\tau,\sigma)g\Big\|_{\infty}\leq \Big\|\langle y\rangle^{-k} e^{\frac{1}{8}|y|^2}e^{-(\tau-\sigma)\mathcal{L}_0}g\Big\|_{\infty} \lesssim \Big\|\langle y\rangle^{-k} e^{\frac{1}{8}|y|^2}g\Big\|_{\infty}.\label{eq:fi123}
\end{align}

We remark that this can be proved directly by applying the maximum principle, but since the representation above will be needed in what follows, we do not want to have a separate and long formulation for the sole purpose of proving \eqref{eq:noOrtho}.

Next we sketch a proof for \eqref{eq:1Ortho}, where a decay rate is wanted. 

To cast the problem into a convenient form, we define a new function $g$ as
$$g(y,\tau):=\tilde{U}_{4}(\tau,\sigma)\tilde{P}_{1} g=\tilde{P}_{1} \tilde{U}_{4}(\tau,\sigma)\tilde{P}_{1} g,$$ then $g(y,\tau)=\tilde{P}_{1} g(y,\tau)$ is the solution to the equation
\begin{align}
\begin{split}\label{eq:linearEvo}
\partial_{\tau}g(y,\tau)=&-\mathcal{L}_1 g(y,\tau)=-(\mathcal{L}_0+V) g(y,\tau)+(1-\tilde{P}_{1})V g(y,\tau).
\end{split}
\end{align}

Apply Duhamel's principle to obtain
\begin{align}
g(y,\tau)=\tilde{U}_3(\tau,\sigma) \tilde{P}_{1}g(y,\sigma)+\int_{\sigma}^{\tau} \tilde{U}_3(\tau,s) (1-\tilde{P}_{1})V g(y,s)\ ds,\label{eq:durhamelFin}
\end{align} where $\tilde{U}_3(\tau,s)$ is the propagator in \eqref{eq:noOrtho}.

We claim that, for some constant $C$ independent of $\tau$, $\sigma$ and $g$,
\begin{align}
\Big\|\langle y\rangle^{-1} e^{\frac{1}{8}|y|^2}\tilde{U}_{3}(\tau,\sigma)\tilde{P}_{1}g(\cdot,\sigma)\Big\|_{\infty}\leq  C \Big[e^{-\frac{1}{2}(\tau-\sigma)}+& R^{-\frac{1}{4}}(\sigma)\Big]  \Big\|\langle y\rangle^{-1} e^{\frac{1}{8}|y|^2}g(\cdot,\sigma)\Big\|_{\infty}, \label{eq:firsInt}\\
\Big\|\langle y\rangle^{-1} e^{\frac{1}{8}|y|^2}\int_{\sigma}^{\tau}\tilde{U}_3(\tau,s) (1-\tilde{P}_{1})V g(\cdot,s) ds\Big\|_{\infty} \leq & C\int_{\sigma}^{\tau} e^{-\frac{1}{5}R^2(s)}\ \Big\|\langle y\rangle^{-1}e^{\frac{1}{8}|y|^2}g(\cdot,s)\Big\|_{\infty} ds. \label{eq:secoInt}
\end{align}

These two claims will be proved in the subsection below. 

Suppose the claims hold, then they and \eqref{eq:durhamelFin} imply that
\begin{align}
\begin{split}\label{eq:tauSigma}
\Big\|\langle y\rangle^{-1} e^{\frac{1}{8}|y|^2} g(\cdot,\tau)\Big\|_{\infty}\leq C\Big\{ \Big[& e^{-\frac{1}{2}(\tau-\sigma)}+R^{-\frac{1}{4}}(\sigma)\Big]  \Big\|\langle y\rangle^{-1} e^{\frac{1}{8}|y|^2}g(\cdot,\sigma)\Big\|_{\infty}\\
&+\int_{\sigma}^{\tau} e^{-\frac{1}{5}R^2(s)}\ \Big\|\langle y\rangle^{-1}e^{\frac{1}{8}|y|^2}g(\cdot,s)\Big\|_{\infty}\ ds\Big\}.
\end{split}
\end{align}
This estimate is crucial, but does not directly imply the desired \eqref{eq:1Ortho}: for any fixed $\sigma$ we can not prove that $\|\langle y\rangle^{-1} e^{\frac{1}{8}|y|^2} g(\cdot,\tau)\|_{\infty}$ decay exponentially fast as $\tau\rightarrow \infty.$ However this is a necessary preparation since it implies that, provided that $R(\sigma)$ is sufficiently large,
\begin{align}
\Big\|\langle y\rangle^{-1} e^{\frac{1}{8}|y|^2} g(\cdot,\tau)\Big\|_{\infty}\leq 2C e^{-\frac{3}{2}(\tau-\sigma)} \Big\|\langle y\rangle^{-1} e^{\frac{1}{8}|y|^2} g(\cdot,\sigma)\Big\|_{\infty},\  \text{when}\ e^{\frac{3}{2}(\tau-\sigma)}\leq R^{\frac{1}{4}}(\sigma) .\label{eq:subIn}
\end{align}

To prove the desired \eqref{eq:1Ortho} we have to remove the condition $ e^{\frac{3}{2}(\tau-\sigma)}\leq R^{\frac{1}{4}}(\sigma)$ above.
This turns out to be easy: we divide the interval $[\sigma,\ \tau]$ into finitely many subintervals, each of them is of length $R(\sigma)$ except the last one, and apply \eqref{eq:subIn} on each interval, and then finally prove the desired \eqref{eq:1Ortho} by putting them together.

To complete the proof we need to prove the claims \eqref{eq:firsInt} and \eqref{eq:secoInt}, this will be achieved in the following subsection.

\subsection{Proof of  \eqref{eq:firsInt} and \eqref{eq:secoInt}}
To prove \eqref{eq:secoInt} we apply \eqref{eq:noOrtho} to find
\begin{align}
\Big\|\langle y\rangle^{-1} e^{\frac{1}{8}|y|^2}\int_{\sigma}^{\tau}\tilde{U}_3(\tau,s) (1-\tilde{P}_{1})V g(\cdot,s)\ ds\Big\|_{\infty}\leq & C \int_{\sigma}^{\tau} \Big\|\langle y\rangle^{-1}e^{\frac{1}{8}|y|^2}(1-\tilde{P}_{1})Vg(\cdot,s)\Big\|_{\infty}\ ds\nonumber\\
=&C \int_{\sigma}^{\tau} \|\langle y\rangle^{-1} \|_{\infty} \Big|\Big\langle Vg, \ e^{-\frac{1}{8}|y|^2}\Big\rangle\Big|\ ds\nonumber\\
\leq & C \int_{\sigma}^{\tau}  \Big|\Big\langle Vg, \ e^{-\frac{1}{8}|y|^2}\Big\rangle\Big|\ ds.\label{eq:firsVunb}
\end{align}Here we only have a minor difficulty, caused by that $V$ is not uniformly bounded.  
Since the function $V$ grows modestly as $|V(y,\tau)|\leq c(\epsilon) R^{\frac{1}{4}}(\tau)$ by \eqref{eq:NewPoten}, and is supported by the set $|y|\geq R,$ and since the function $e^{-\frac{1}{4}|y|^2}$ decays rapidly, it is easy to obtain
\begin{align}
\Big|\Big\langle Vg, \ e^{-\frac{1}{8}|y|^2}\Big\rangle\Big|\lesssim e^{-\frac{1}{4}R^2}\|\langle y\rangle^{-1}e^{\frac{1}{8}|y|^2}g\|_{\infty}.
\end{align}This together with \eqref{eq:firsVunb} implies the desired \eqref{eq:secoInt}.

In what follows we sketch a proof for \eqref{eq:firsInt}. 

When $\tau-\sigma\in [0,1]$ it is implied by \eqref{eq:noOrtho}.

Next we consider the regime $\tau-\sigma\geq 1.$ We start with considering the easiest case, specifically estimating $e^{\frac{1}{8}|y|^2}\widetilde{U}_{3}(\tau,\sigma)g$ with $g$ satisfying the conditions
\begin{align}
\int_{-\infty}^{\infty} e^{-\frac{1}{8}|y|^2} g(y)\ dy_k=0,\ k=1,2,3.\label{eq:3intZero}
\end{align}Without loss of generality, we assume that 
\begin{align}\label{eq:assumOrtho}
\Big\|(1+|y_1|)^{-1} e^{\frac{1}{8}|y|^2}g\Big\|_{\infty}=\min_{k=1,2,3}\Big\{\Big\|(1+|y_k|)^{-1} e^{\frac{1}{8}|y|^2}g\Big\|_{\infty}\Big\}.
\end{align}
Returning to the identity \eqref{eq:IntKernel}, we integrate by parts in $z_1$ to find
\begin{align}
\widetilde{U}_{3}(\tau,\sigma)g(y)=\int_{\mathbb{R}^3}\partial_{z_1}\Big[K_{\tau-\sigma}(y,z) \langle e^{-V}\rangle(y,z)\Big] \int_{-\infty}^{z_1}e^{-\frac{1}{8}(v_1^2+z_2^2+z_3^2)} g(v_1,z_2,z_3)\ dv_1 d^3 z.\label{eq:integz1}
\end{align}
The $z_1-$derivative generates a decay estimate: when $\tau-\sigma\geq 1$, $$
\Big|\partial_{z_1}[K_{\tau-\sigma}(y,z) \langle e^{-V}\rangle(y,z)]\Big|\lesssim  \Big[ e^{-\frac{1}{2}(\tau-\sigma)}+R^{-\frac{1}{4}}(\sigma)\Big] (1+|y_1|+|z_1|)K_{\tau-\sigma}(y,z).
$$ For the second factor of the integrand in (\ref{eq:integz1}), \eqref{eq:3intZero} implies that $$\int_{-\infty}^{z_1}e^{-\frac{1}{8}(v_1^2+z_2^2+z_3^2)} g(v_1,z_2,z_3) dv_1=-\int_{z_1}^{\infty} e^{-\frac{1}{8}(v_1^2+z_2^2+z_3^2)} g(v_1,z_2,z_3) dv_1, $$
and by l'hopital's rule we obtain, for any $z_1,$
\begin{align}
\begin{split}
\Big|\int_{-\infty}^{z_1}e^{-\frac{1}{8}(v_1^2+z_2^2+z_3^2)} g(v_1,z_2,z_3) dv_1\Big|&\lesssim e^{-\frac{1}{4}|z|^2}\Big \|(1+|y_1|)^{-1}e^{\frac{1}{8}|y|^2} g\Big\|
_{\infty}\\
&\lesssim e^{-\frac{1}{4}|z|^2} \Big\|\langle y\rangle^{-1}e^{\frac{1}{8}|y|^2} g\Big\|_{\infty},
\end{split}
\end{align} where in the last step we use \eqref{eq:assumOrtho}.

Collect the estimates to find, when $\tau-\sigma\geq 1$
\begin{align}
\begin{split}
&|\widetilde{U}_{3}(\tau,\sigma)g|\\
&\lesssim \Big[ e^{-\frac{1}{2}(\tau-\sigma)}+R^{-\frac{1}{4}}(\sigma)\Big] \int_{\mathbb{R}^3} K_{\tau-\sigma}(y,z)(1+|y_1|+|z_1|) e^{-\frac{1}{4}|z|^2} \ d^3z\ \Big\|\langle y\rangle^{-1}e^{\frac{1}{8}|y|^2} g\Big\|_{\infty}.
\end{split}
\end{align} Recall that $K_{\tau-\sigma}(y,z)e^{-\frac{1}{8}|z|^2}$ is the integral kernel of $e^{-(\tau-\sigma)\mathcal{L}_0}$. Apply \eqref{eq:fi123} to obtain the desired estimate
\begin{align}
\|\langle y\rangle^{-1}e^{\frac{1}{8}|y|^2}\widetilde{U}_{3}(\tau,\sigma)g\|_{\infty}\lesssim \Big[ e^{-\frac{1}{2}(\tau-\sigma)}+R^{-\frac{1}{4}}(\sigma)\Big] \|\langle y\rangle^{-1}e^{\frac{1}{8}|y|^2} g\|_{\infty}.
\end{align}

Now we turn to the general case, i.e. controlling $\tilde{U}_{3}(\tau,\sigma)g$ with $g$ satisfying $$\int_{\mathbb{R}^3}e^{-\frac{1}{8}|y|}g(y)\ d^3 y=0.$$
The strategy is to transform it to the easiest case. 

Define orthogonal projections $\Omega_{k,1}$ and $\Omega_{k,2}$, $k=1,2,3,$ as $$\Omega_{k,1}f(y):=\frac{1}{\int_{-\infty}^{\infty}e^{-\frac{1}{4}y_k^2} dy_k}e^{-\frac{1}{8}y_k^2}\int_{-\infty}^{\infty}e^{-\frac{1}{8}y_k^2}f(y) dy_k,$$ and$$\Omega_{k,2}:=1-\Omega_{k,1}.$$Decompose the function $g$ into two parts
\begin{align}
g(y)=\sum_{l_k=1,2} \prod_{k} \Omega_{k,l_k}g(y).
\end{align}Obviously $\prod_k\Omega_{k,1}g=0$ by the orthogonal condition imposed on $g$. $\prod_k\Omega_{k,2}g$ satisfies the condition \eqref{eq:3intZero}, thus can be controlled the technique for the easiest case. Controlling the other terms is even easier since, for example, $\Omega_{11}\Omega_{22}\Omega_{32}g(y)=e^{-\frac{1}{8}y^2_1}\tilde{g}(y_2,y_3)$ for some function $\tilde{g}$ satisfying 
$$\int_{-\infty}^{\infty}e^{-\frac{1}{8}y_2^2} \tilde{g}(y_2,y_3)\ dy_2=\int_{-\infty}^{\infty}e^{-\frac{1}{8}y_3^2} \tilde{g}(y_2,y_3)\ dy_3=0.$$We skip these parts. For more details see \cite{MultiDHeat}.

\end{proof}

\section{Proof of Lemma \ref{LM:ColdMini}}\label{sec:LMColdMini}
We start with recalling some definitions from \cite{ColdingMiniUniqueness}.

The shrinker scale $R_s:[0,\infty)\rightarrow \mathbb{R}^{+}$ is a scalar function defined by the identity
\begin{align}
e^{-\frac{1}{2}R_s^2(\tau)}=F(\Sigma_{\tau-1})-F(\Sigma_{\tau+1})=\int_{\tau-1}^{\tau+1} |\nabla_{\Sigma_s}F|^2 \ ds,\label{eq:defShr}
\end{align}
where the function $F(\Sigma_{\tau})$, for the rescaled surface $\Sigma_{\tau}$ at time $\tau,$ is defined as
\begin{align}
F(\Sigma_{\tau}):=(4\pi)^{-2}\int_{\Sigma_{\tau}} e^{-\frac{|x|^2}{4}}\ d\mu,
\end{align} and 
$|\nabla_{\Sigma_\tau}F|^2$ is defined as, with $\phi:=H-\frac{\langle x, \ \bf{n}\rangle}{2},$ 
\begin{align}
|\nabla_{\Sigma_\tau}F|^2:=\int_{\Sigma_{\tau}} \phi^2 e^{-\frac{|x|^2}{4}}\ d\mu.\label{eq:defNaF}
\end{align} 

It was proved that, in the coordinate constructed for time $\tau$, inside the ball $B_{R_{s}}(0)$, the rescaled MCF takes the form, up to some tilts of axis,
\begin{align}
\Sigma_{\tau}=\left[
\begin{array}{c}
y\\
v(y,\theta,\tau) cos\theta\\
v(y,\theta,\tau) sin\theta
\end{array}
\right].
\end{align}

Now we study this part of MCF: $\Sigma_{\tau}\cap B_{R_{s}}(0)$, or equivalently the part satisfying $|y|^2+|v|^2\leq R_{s}^2$. Here the graph function $v$ satisfies the estimates: there exists a small $\kappa>0$ and a large integer $l$ such that, for some $C_{l}>0,$
\begin{align}
\Big|v-\sqrt{2}\Big|_{C^{2,\alpha}}\leq \kappa,\ \text{and}\ \sum_{|k|+m\leq l}\Big|\nabla_{y}^{k}\partial_{\theta}^{m}v\Big|\leq C_{l}.\label{eq:cylinder}
\end{align}
Moreover, by choosing $\epsilon$ in (6.10) and (6.11) of \cite{ColdingMiniUniqueness} to be $\frac{1}{3-\epsilon_0}$ with $\epsilon_0$ be an arbitrarily small positive constant, which is allowed by Theorem 6.1 of \cite{ColdingMiniUniqueness}, and by applying Lemma 6.9 of \cite{ColdingMiniUniqueness}, we have that there exists a $C_{\epsilon_0}$ such that if $l$ in \eqref{eq:cylinder} is large enough, then
\begin{align}
F(\Sigma_{\tau-1})-F(\Sigma_{\tau+1})\leq F(\Sigma_{\tau-1})-F(\mathcal{C}) \leq C_{\epsilon_0} \tau^{-3+\epsilon_0}.\label{eq:ODE}
\end{align} 
Here $F(\mathcal{C})$ is a constant defined as $F(\mathcal{C})=\lim_{\tau\rightarrow \infty} F(\Sigma_{\tau})$ and recall that $F(\Sigma_{\tau})$ is monotonically decreasing.

\eqref{eq:ODE} and \eqref{eq:defShr} imply that the shrinker scale $R_s(\tau)$ satisfies the estimate
\begin{align}
e^{\frac{1}{2}R_{s}^2(\tau)}\geq C_{\epsilon_0} \tau^{3-\epsilon_0},\ \text{equivalently}, \ R_s\geq \sqrt{2} \big[\ln C_{\epsilon_0}+(3-\epsilon_0) \ln \tau\big]^{\frac{1}{2}}.\label{eq:CloseSharp}
\end{align}

They also provided an estimate for $\|\phi\|_{L^1(B_{R_s})}$, defined as
\begin{align}
\|\phi\|_{L^1(B_{R_s})}:=\int_{\Sigma_{\tau}\cap B_{R_s}} |\phi| e^{-\frac{|x|^2}{4}}\ d\mu.
\end{align}

By the estimate (5.33) in \cite{ColdingMiniUniqueness}, with $\beta=0.5$, there exists a constant $C_1>0$ such that
\begin{align}
|\nabla_{\Sigma_\tau}F|^2_{B_{R_s}}:=\int_{\Sigma_{\tau}\cap B_{R_s}} \phi^2 e^{-\frac{|x|^2}{4}}\ d\mu \leq C_1 \Big[F(\Sigma_{\tau-1})-F(\Sigma_{\tau+1})\Big].
\end{align}
Apply Schwartz inequality and
\eqref{eq:ODE} to find, for some $C_2,\ C_3>0,$
\begin{align}
\|\phi\|_{L^1(B_{R_s})} \leq C_2  \Big[F(\Sigma_{\tau-1})-F(\Sigma_{\tau+1})\Big]^{\frac{1}{2}}\leq C_3 \tau^{-\frac{3-\epsilon_0}{2}}.\label{eq:L1normPhi}
\end{align}

Next we derive pointwise estimates for $v$ and its derivatives. 
To facilitate later discussions, we define a scalar function $R$ as
\begin{align}
R:=\min\{R_{s},\ \sqrt{6}\sqrt{\ln \tau}\}.\label{eq:defRRR}
\end{align}

By (2.51) of \cite{ColdingMiniUniqueness} with $n=4,\ k=1$, in the region $|y|^2+|v|^2\leq r^2-3$ which is equivalent to $\Sigma_{\tau}\cap B_{\sqrt{r^2-3}}(0)$, we have that if $r\leq R-2$, then
\begin{align}
\big|v(y,\theta,\tau)-\sqrt{2}\big|+ \sum_{|m|+j=1}\big|\nabla_y^{m}\partial_{\theta}^{j} v(y,\theta,\tau)\big|\leq C_{\lambda_0, l, C_{l}} R^{13} \Big[ e^{-d_{l,4}\frac{(R-1)^2}{8}}+\|\phi\|^{\frac{d_{l,4}}{2}}_{L^1(B_{R})}\Big] e^{\frac{1}{8}r^2}.
\end{align} here $d_{l,4}$ is a constant satisfying $d_{l,4}\leq 1$ and $d_{l,4}\rightarrow 1$ as $l\rightarrow \infty$, with $l$ from \eqref{eq:cylinder}, where, recall the definition of $\lambda_0$ in \eqref{eq:generic}. 
By fixing $|y|^2+|v|^2= r^2-3$ with $r\leq R-2$, we find,
\begin{align*}
|v(y,\theta,\tau)-\sqrt{2}|+ \sum_{|m|+j=1}|\nabla_y^{m}\partial_{\theta}^{j} v(y,\theta,\tau)|
\leq C_{\lambda_0, l, C_{l}} R^{13} \Big[ e^{-d_{l,4}\frac{(R-1)^2}{8}}+\|\phi\|^{\frac{d_{l,4}}{2}}_{L^1(B_{R})}\Big] e^{\frac{1}{8}(|y|^2+|v|^2+3)}.
\end{align*}
The facts $|v-\sqrt{2}|$ is small in \eqref{eq:cylinder} and $R\gg 1$ imply that, for $|y|\leq R-3$, there exists a new constant $\tilde{C}_{\lambda_0, l, C_{l}} $ such that
\begin{align}
\Big|v(y,\theta,\tau)-\sqrt{2}\Big|+ \sum_{|m|+j=1}\Big|\nabla_y^{m}\partial_{\theta}^{j} v(y,\theta,\tau)\Big|
\leq \tilde{C}_{\lambda_0, l, C_{l}} R^{13} \Big[ e^{-d_{l,4}\frac{(R-1)^2}{8}}+\|\phi\|^{\frac{d_{l,4}}{2}}_{L^1(B_{R})}\Big] e^{\frac{1}{8}|y|^2}.
\end{align} 

We have one more difficulty: the estimates above hold only in the coordinate constructed for time $\tau,$ but we need pointwise estimates in the limit coordinate.

It was proved in \cite{ColdingMiniUniqueness} that the sequence of the chosen coordinates converge. In what follows we adapt their arguments to show that the convergence rate is $\tau^{-\frac{3}{4}-\sigma},$ for some $\sigma>0.$ To this end we need Lemma A.48 of \cite{ColdingMiniUniqueness} and the estimate 
\begin{align}
\int_{\tau}^{\infty} \|\phi(\cdot,s)\|_{1}\ ds\lesssim \tau^{-\frac{3}{4}-\sigma}.\label{eq:desiredRate}
\end{align}
This will be proved in subsection \ref{subsec:converRate}.

Suppose \eqref{eq:desiredRate} holds. Denote the graph function, in the coordinate for $\tau=\infty$, by $\tilde{v}(\cdot,\tau)$. Compute directly to find that, for $|y|\leq R-3=\mathcal{O}(\sqrt{\ln\tau})$,
\begin{align}
|\tilde{v}(y,\theta,\tau)-\sqrt{2}|+ \sum_{|m|+j=1}|\nabla_y^{m}\partial_{\theta}^{j} \tilde{v}(y,\theta,\tau)|
\leq \tilde{C}_{\lambda_0, l, C_{l}} R^{13} \Big[ e^{-d_{l,4}\frac{(R-1)^2}{8}}+\tau^{-\frac{3}{4}}+\|\phi\|^{\frac{d_{l,4}}{2}}_{L^1(B_{R})}\Big] e^{\frac{1}{8}|y|^2}.\label{eq:point10}
\end{align} 
We choose a sufficiently large $l$ to make $d_{l,4}$ sufficiently close to $1$. This, together with the definition of $R$ in \eqref{eq:defRRR} and the estimates in \eqref{eq:CloseSharp} and \eqref{eq:L1normPhi}, implies that for any $\epsilon_1>0$, there exists a constant $C_{\epsilon_1}$ such that
\begin{align}
e^{-d_{l,4}\frac{(R-1)^2}{8}}+\|\phi\|^{\frac{d_{l,4}}{2}}_{L^1(B_{R})}\leq e^{-d_{l,4}\frac{(R-1)^2}{8}}+\|\phi\|^{\frac{d_{l,4}}{2}}_{L^1(B_{R_s})}\leq C_{\epsilon_1}\tau^{-\frac{3}{4}+\epsilon_1}.
\end{align}
Hence, if $\tau$ is large enough, then 
\begin{align}
\tilde{C}_{\lambda_0, l, C_{l}} R^{13}\Big[e^{-d_{l,4}\frac{(R-1)^2}{8}}+\tau^{-\frac{3}{4}}+\|\phi\|^{\frac{d_{l,4}}{2}}_{L^1(B_{R})}\Big]  \leq \tau^{-\frac{18}{25}},\label{eq:number10}
\end{align} by the obvious fact $\frac{18}{25}<\frac{3}{4}$ and that $R=\mathcal{O}(\sqrt{\ln \tau})$.

This directly implies the desired \eqref{eq:cm1}:  Since $R-3-\frac{12}{5}\sqrt{\ln \tau}\gg 1$, the estimates \eqref{eq:point10} and \eqref{eq:number10} hold when $|y|\leq \frac{12}{5}\sqrt{\ln \tau}$. Hence, the desired estimate follows,
\begin{align}
\big|\tilde{v}(y,\theta,\tau)-\sqrt{2}\big|,\ |\nabla_y \tilde{v}(y,\theta,\tau)|,\ |\partial_{\theta}\tilde{v}(y,\theta,\tau)|\leq \tau^{-\frac{18}{25}} e^{\frac{1}{8}|y|^2}.\label{eq:c18}
\end{align}

The estimates in \eqref{eq:cm2} are implied by \eqref{eq:cylinder} and the convergence rate of the coordinate systems in \eqref{eq:desiredRate}, and the fact we only consider the region $|y|\leq R_0(\tau)$.

What is left is to prove \eqref{eq:IniWeighted}. Since 
$\tau^{-\frac{18}{25}} e^{\frac{1}{8}|y|^2}\leq \tau^{-\frac{1}{50}}$ when $|y|\leq R_1(\tau),$ \eqref{eq:c18} implies the desired estimates for $|\tilde{v}-\sqrt{2}|$ and $\partial_{\theta}^{l}\nabla_{y}^{k}\tilde{v}$ with $|k|+l=1$ in \eqref{eq:IniWeighted}.
When $|k|+l=2,3,4,$ we interpolate between the decay estimates of $\partial_{\theta}^{m}\nabla_{y}^{k}\tilde{v}$, $|k|+l=1$, and the estimates of $\partial_{\theta}^{m}\nabla_{y}^{k}\tilde{v}$, $|k|+l=5$, in \eqref{eq:cm2}.
\subsection{Proof of \eqref{eq:desiredRate}}\label{subsec:converRate}
It was proved in \cite{ColdingMiniUniqueness} that $\int_{0}^{\infty} \|\phi(\cdot,s)\|_{1}\ ds<\infty$. Here we need a decay rate for $\int_{\tau}^{\infty} \|\phi(\cdot,s)\|_{1}\ ds$, and we obtain this by slightly refining their arguments.

To prove \eqref{eq:desiredRate}, we apply Schwartz's inequality and use that $F(\Sigma_{\tau})>0$ is decreasing to find, as (6.23) in \cite{ColdingMiniUniqueness},
\begin{align}
\int_{\tau}^{\infty} \|\phi(\cdot,s)\|_{1}\ ds\leq \sqrt{F(\Sigma_{\tau})}\sum_{j=0}^{\infty} \Big[ F(\Sigma_{\tau+j})-F(\Sigma_{\tau+j+1}) \Big]^{\frac{1}{2}}.\label{eq:step101}
\end{align} Apply H\"older's inequality to find, for any $p\in (1,2]$,
\begin{align}\label{eq:sumrule}
\Big[\sum_{j=0}^{\infty} \big[ F(\Sigma_{\tau+j})-F(\Sigma_{\tau+j+1}) \big]^{\frac{1}{2}}\Big]^2 \leq \Big[\sum_{j=0}^{\infty} \Big(F(\Sigma_{\tau+j})-F(\Sigma_{\tau+j+1})\Big)(j+1)^{p}\Big]\sum_{l=0}^{\infty} (l+1)^{-p}.
\end{align} 
The second factor is bounded since $p>1$, and for the first one we have
\begin{align*}
\begin{split}
\sum_{j=0}^{\infty} \Big(F(\Sigma_{\tau+j})-F(\Sigma_{\tau+j+1})\Big)(j+1)^{p}=&F(\Sigma_{\tau})-F(\mathcal{C})\\
&+\sum_{j=0}^{\infty}[(j+2)^{p}-(j+1)^{p}][F(\Sigma_{\tau+j+1})-F(\mathcal{C})].
\end{split}
\end{align*}
This, \eqref{eq:ODE}, and the estimate, for some $C_{p}>0,$ $$(j+2)^{p}-(j+1)^{p}\leq C_{p}j^{p-1}\leq C_{p}(j+\tau)^{p-1}$$ 
imply, for some $C_{p,\epsilon_0}>0$
\begin{align}
\sum_{j=0}^{\infty}[(j+2)^{p}-(j+1)^{p}][F(\Sigma_{\tau+j+1})-F(\mathcal{C})]\leq C_{p,\epsilon_0} \tau^{-3+p+\epsilon_0}.
\end{align}
Returning to \eqref{eq:sumrule}, we find, for some $\tilde{C}_{p,\epsilon_0}>0,$
\begin{align}
\sum_{j=0}^{\infty} \big[ F(\Sigma_{\tau+j})-F(\Sigma_{\tau+j+1}) \big]^{\frac{1}{2}} \leq \tilde{C}_{p,\epsilon_0} \tau^{-\frac{3-p-\epsilon_0}{2}}.
\end{align}

Apply this to \eqref{eq:step101} and choose a $p$ sufficiently close to $1$ to obtain the desired
\eqref{eq:desiredRate}, recall that $F(\Sigma_{\tau})>0$ is decreasing in $\tau$ and is uniformly bounded.

\def\cprime{$'$}

\end{document}